\documentclass{article}

\usepackage[accepted]{icml2025}

\usepackage{microtype}
\usepackage{graphicx}
\usepackage{subfigure}
\usepackage{multirow}
\usepackage{booktabs} 

\usepackage{hyperref}

\usepackage{tikz}
\usetikzlibrary{tikzmark,decorations.pathreplacing}

\hypersetup{
  colorlinks   = true, 
  urlcolor     = blue, 
  linkcolor    = {red!75!black}, 
  citecolor   = {blue!50!black} 
}

\newcommand{\squeeze}{\textstyle}

\newcommand{\newalg}{\algname{BPM}}

\usepackage{amsmath}
\usepackage{amssymb}
\usepackage{mathtools}
\usepackage{amsthm}

\usepackage{amsfonts}

\usepackage[capitalize,noabbrev]{cleveref}

\usepackage{graphicx}
\usepackage{apptools}
\usepackage[flushleft]{threeparttable}
\usepackage[export]{adjustbox}
\usepackage{array,booktabs,makecell}
\usepackage{multirow}
\usepackage{nicefrac}
\usepackage{amsthm}
\usepackage{amsmath,amsfonts,bm}
\usepackage{wrapfig}
\usepackage{caption}
\usepackage{siunitx}
\usepackage{thm-restate}
\usepackage{nccmath}
\usepackage{bbm}
\usepackage[shortlabels]{enumitem}

\usepackage{color}
\usepackage{colortbl}
\definecolor{bgcolor}{rgb}{0.76,0.88,0.50}
\definecolor{bgcolor0}{rgb}{0.93,0.99,1}
\definecolor{bgcolor1}{rgb}{0.8,1,1}
\definecolor{bgcolor2}{rgb}{0.8,1,0.8}
\definecolor{bgcolor3}{rgb}{0.50,0.90,0.50}
\usepackage{tcolorbox}
\usepackage{pifont}

\newcommand{\norm}[1]{\left\| #1 \right\|}

\newcommand{\inp}[2]{\left\langle#1,#2\right\rangle} 

\newcommand{\parens}[1]{\left( #1 \right)}
\newcommand{\brac}[1]{\left\{ #1 \right\}}

\newcommand{\ceil}[1]{\left\lceil #1\right\rceil} 

\newcommand{\R}{\mathbb{R}} 

\newcommand{\Exp}[1]{{\rm \mathbb{E}}\left[#1\right]}

\newcommand{\ExpCond}[2]{{\mathbb{E}}\left[\left.#1\right\vert#2\right]}

\newcommand{\cB}{\mathcal{B}}
\newcommand{\cC}{\mathcal{C}}

\newcommand{\cH}{\mathcal{H}}

\newcommand{\cN}{\mathcal{N}}
\newcommand{\cO}{\mathcal{O}}

\newcommand{\cX}{\mathcal{X}}

\newcommand{\eqdef}{:=} 

\makeatletter
\newcommand{\vast}{\bBigg@{4}}

\newcommand{\rbrac}[1]{\left(#1\right)}

\newcommand{\cbrac}[1]{\left\{#1\right\}}

\newcommand{\BProxSub}[3]{\textnormal{{brox}}^{#1}_{#2}(#3)}
\newcommand{\BBProxSub}[4]{\textnormal{{brox}}^{#1}_{#2, #3}\rbrac{#4}}
\newcommand{\ProxSub}[2]{\textnormal{{prox}}_{#1}(#2)}
\newcommand{\ProxPSub}[2]{\textnormal{{prox}}^p_{#1}(#2)}

\newcommand{\BMoreauSub}[3]{{N}^{#1}_{#2}\rbrac{#3}}

\newcommand{\inner}[2]{\left\langle #1, #2 \right\rangle}
\newcommand{\dist}[2]{\textnormal{dist}\rbrac{#1, #2}}
\newcommand{\breg}[3]{D_{#1}\rbrac{#2, #3}}

\newcommand{\cmax}[1]{c^{\max}_t\rbrac{x_{#1}}}

\newcommand{\SBProjProxSub}[3]{\Pi\rbrac{#3, \BProxSub{#1}{#2}{#3}}}

\DeclareMathOperator*{\argmin}{arg\,min}

\definecolor{mydarkgreen}{RGB}{39,130,67}
\definecolor{mydarkred}{RGB}{192,47,25}
\definecolor{mydarkorange}{RGB}{39,130,67}
\newcommand{\green}{\color{mydarkgreen}}
\newcommand{\red}{\color{mydarkred}}
\newcommand{\cmark}{\green\ding{51}}%
\newcommand{\xmark}{\red\ding{55}}%

\newcommand{\algname}[1]{{\green \small \sf #1}}

\DeclareMathOperator{\ri}{ri}           
\DeclareMathOperator{\interior}{int}    
\DeclareMathOperator{\bdry}{bdry}       

\newtheorem{example}{Example}

\theoremstyle{plain}
\newtheorem{theorem}{Theorem}[section]
\newtheorem{proposition}[theorem]{Proposition}
\newtheorem{lemma}[theorem]{Lemma}
\newtheorem*{lemma*}{Lemma}        
\newtheorem{corollary}[theorem]{Corollary}
\newtheorem{fact}[theorem]{Fact}
\theoremstyle{definition}
\newtheorem{definition}[theorem]{Definition}
\newtheorem{assumption}[theorem]{Assumption}
\theoremstyle{remark}
\newtheorem{remark}[theorem]{Remark}

\usepackage[textsize=tiny]{todonotes}

\icmltitlerunning{The Ball-Proximal (=``Broximal'') Point Method}

\begin{document}

\twocolumn[
\icmltitle{The Ball-Proximal (=``Broximal'') Point Method: \\ a New Algorithm, Convergence Theory, and Applications}



\icmlsetsymbol{equal}{*}

\begin{icmlauthorlist}
\icmlauthor{Kaja Gruntkowska}{yyy}
\icmlauthor{Hanmin Li}{yyy}
\icmlauthor{Aadi Rane$^{\star}$}{yyy,sch}
\icmlauthor{Peter Richt\'{a}rik}{yyy}
\end{icmlauthorlist}

\icmlaffiliation{yyy}{King Abdullah University of Science and Technology, Thuwal, Saudi Arabia}
\icmlaffiliation{sch}{University of California, Berkeley, USA}


\icmlkeywords{Machine Learning, ICML}

\vskip 0.3in
]



\printAffiliationsAndNotice{}  

\begin{abstract}
    Non-smooth and nonconvex global optimization poses significant challenges across various applications, where standard gradient-based methods often struggle.
    We propose the {\em Ball-Proximal Point Method}, {\em Broximal Point Method}, or {\em Ball Point Method} (\newalg) for short -- a novel algorithmic framework inspired by the classical Proximal Point Method (\algname{PPM}) \citep{rockafellar1976monotone}, which, as we show, sheds new light on several foundational optimization paradigms and phenomena, including nonconvex and non-smooth optimization, acceleration, smoothing, adaptive stepsize selection, and trust-region methods.    
    At the core of \newalg\ lies the {\em ball-proximal (``broximal'')} operator, which arises from the classical proximal operator by replacing the quadratic distance penalty by a ball constraint. Surprisingly, and in sharp contrast with the sublinear rate of \algname{PPM} in the nonsmooth convex regime, we prove that \newalg\  converges {\em linearly} and in a {\em finite} number of steps in the same regime.  
     Furthermore, by introducing the concept of ball-convexity, we prove that \newalg\ retains the same global convergence guarantees under weaker assumptions, making it a powerful tool for a broader class of potentially nonconvex optimization problems.
Just like \algname{PPM} plays the role of a conceptual method inspiring the development of practically efficient algorithms and algorithmic elements, e.g., gradient descent, adaptive step sizes, acceleration~\citep{anh2020understanding}, and ``W'' in \algname{AdamW} \citep{Orabona-AdamW}, we believe that \newalg\ should  be understood in the same manner: as  a blueprint and inspiration for further development.
\end{abstract}

\section{Introduction}
\label{sec:intro}
The minimization of nonconvex functions is a fundamental challenge across many fields, including machine learning, optimization, applied mathematics, signal processing and operations research. 
Solving such problems is integral to most machine learning algorithms arising in both training and inference, where nonconvex objectives or constraints are often necessary to capture complex prediction tasks. 

\subsection{Global nonconvex optimization}
In this paper, we propose a new meta-algorithm (see \Cref{sec:bpm}) capable of finding \emph{global} minimizers for a specific (new) class of nonconvex functions.
In particular, we introduce an algorithmic framework designed to solve the (potentially nonconvex) optimization problem
\begin{align}
    \label{eq:main-prob}
    \min_{x \in \R^d} f(x),
\end{align}
where $f: \R^d \mapsto \R \cup \{+\infty\}$ is assumed to be proper (which means that the set ${\rm dom} f \eqdef \{x \in \R^d : f(x)<+\infty\}$ is nonempty), closed, and have at least one minimizer.
We let $\cX_f$ be the set of all minimizers of $f$, and $f_\star \eqdef \min_x f(x)$.

\subsection{The ball-proximal operator} \label{sec:broximal}

A key inspiration for our method stems from the well-known Proximal Point Method (\algname{PPM}) \citep{rockafellar1976monotone}, which iteratively adds a quadratic penalty term to the objective (see Section~\ref{sec:nonsmooth} for more details) and solves a modified subproblem  at each step.
Building on this idea, we introduce the \emph{ball-proximal (``broximal'') operator}:
\begin{definition}[Ball-Proximal Operator]
    \label{def:ballprox}
    The \emph{ball-proximal (``broximal'') operator} with radius $t > 0$ associated with a function $f: \R^d \mapsto \R \cup \{+\infty\}$ is given by 
    \begin{align}
        \label{eq:def:ballprox}
        \squeeze \BProxSub{t}{f}{x} \eqdef \argmin \limits_{z \in B_t(x)} f(z),
    \end{align}
    where $B_t(x) \eqdef \brac{z \in \R^d : \norm{z - x} \leq t}$ and $\norm{\cdot}$ is the standard Euclidean norm.
\end{definition}
According to this definition, for a given input point $x\in\R^d$, $\BProxSub{t}{f}{x}$ returns the minimizer(s) of~$f$ within the ball of radius~$t$ centered at~$x$.

\subsection{The Ball-Proximal Point Method}
\label{sec:bpm}
With the above definition in place, we turn to introducing our basic method, aiming to solve problem \eqref{eq:main-prob}, which we refer to as the {\em Ball-Proximal Point Method}, {\em Broximal Point Method}, or simply {\em Ball Point Method} (\newalg):
\begin{align}
    \label{eq:bppm}
    \squeeze \boxed{x_{k+1} \in \BProxSub{t_k}{f}{x_k}} \tag{\newalg}
\end{align}
Similarly to the classical \algname{PPM}, at each iteration \newalg\ solves an auxiliary optimization problem -- in this case, minimizing~$f$ over a ball centered at $x_k$ of radius~$t_k>0$.
The use of ``$\in$'' instead of ``$=$'' reflects the fact that $\BProxSub{t}{f}{x_k}$ may not in general be a singleton (unless further assumptions are made). 
In such cases, the algorithm allows the flexibility of selecting from the set of minimizers.
Notably, when this radius is large enough, i.e., $t_0 \geq \norm{x_0-x_{\star}}$ (where $x_0$ is a starting point and $x_{\star}$ is an optimal point), then $x_{\star} \in \BProxSub{t_0}{f}{x_0}$ and the algorithm finds a global solution in $1$ step.

\begin{table}[t]
\caption{Summary of equivalent formulations of \newalg\ and the corresponding assumptions under which they hold.}
\begin{adjustbox}{width=\columnwidth,center}
\begin{threeparttable}
\renewcommand{\arraystretch}{1.2}
    \begin{tabular}{c c c}
        \toprule
        \textbf{Variant} & \textbf{Expression} & \textbf{Assumptions} \\ \midrule
        \begin{tabular}{c} \newalg \\ \Cref{sec:bpm} \end{tabular}
        & $\begin{array} {lcl} x_{k+1} & \in & \BProxSub{t_k}{f}{x_k} \end{array}$
        & --- \\
\hline       
        \begin{tabular}{c} \algname{$\|$PPM$\|$} \\ \Cref{lemma:brox_prox} \end{tabular}
        & $\begin{array} {lcl} x_{k+1} & = & \ProxSub{\frac{t_k}{\norm{\nabla f(x_{k+1})}} f}{x_k} \\ & = & x_k - \frac{t_k}{\norm{\nabla f(x_{k+1})}} \nabla f(x_{k+1})\end{array}$        
        & \begin{tabular}{c} convexity \\ differentiability \end{tabular} \\
\hline       
        \begin{tabular}{c} \algname{$\|$GD$\|$} \\ \scriptsize{on ball envelope}\tnote{1} \\ \Cref{thm:bppm-gd-relation} \end{tabular}
        & $\begin{array} {lcl} x_{k+1} & = & x_k - \frac{t_k}{\norm{\nabla N^{t_k}_f(x_k)}} \nabla \BMoreauSub{t_k}{f}{x_k} \end{array}$
        & \begin{tabular}{c} convexity \\ differentiability, \\ $L$--smoothness \end{tabular} \\ \bottomrule
    \end{tabular}
    \begin{tablenotes}
    \item[1] $\BMoreauSub{t}{f}{x}$ is the \emph{ball envelope} of $f$:  $\BMoreauSub{t}{f}{x} \eqdef \min \limits_{z \in B_t(x)} f(z)$; see \Cref{def:ball-envelope}.
    \end{tablenotes}
\end{threeparttable}
\end{adjustbox}
\label{table:reformulations}
\end{table}

\subsection{Summary of contributions} 

Our key contributions are summarized as follows:
\begin{enumerate}
    \item {\bf New oracle: broximal operator.} Inspired by the classical proximal operator, we introduce the \emph{ball-proximal (=broximal) operator}\footnote{After the first version of this paper appeared online, we became aware of closely related prior work. We discuss the connections and differences in \Cref{sec:ball_oracles}.} (\Cref{sec:broximal}) mapping input points to minimizers of the objective function within a localized region (a ball) centered at the input. The broximal operator enjoys several useful properties.
    In certain scenarios (e.g., when~$f$ is convex or satisfies Assumption~\ref{as:defining_rel}), the operator is single-valued on $\{x: \BProxSub{t}{f}{x}\not\subseteq \cX_f\}$ (\Cref{prop:single_val}), and the minimizer is guaranteed to lie on the boundary of the ball (\Cref{thm:1stbroxthm} and~\Cref{prop:t_step}).
    For a full discussion of the relevant properties, see Appendices~\ref{sec:ap_convex} and~\ref{sec:ball-convex}. We relegate these and other auxiliary results related to the broximal operator to the appendix since in the main body of the paper we decided to focus on more high-level results, such as convergence and connections to existing works, phenomena and fields.
    \item {\bf New abstract method: Broximal Point Method.}
    We propose the \emph{Broximal Point Method} (\newalg), a novel abstract yet immensely powerful algorithmic framework (\Cref{sec:bpm}). \newalg\ is inherently linked to several existing methods, and admits multiple reformulations, summarized in \Cref{table:reformulations}, with detailed discussions in the subsequent sections. If $f$ is convex and differentiable, \newalg\  can be interpreted as a {\em normalized}  variant of \algname{PPM} (to the best of our knowledge, this is a new method). If, in addition, $f$ is smooth,  \newalg\  can be interpreted as normalized gradient descent performed on the {\em ball envelope} (a new concept) of $f$; which is analogous to \algname{PPM} being equivalent to gradient descent on the Moreau envelope of $f$. Given the importance of gradient normalization in modern deep learning, we believe that these observations alone make \newalg\ an interesting object of study.
    \item {\bf Connections to important phenomena and fields.}
    Surprisingly, \newalg\ is of relevance to and shares numerous connections with several important phenomena, works and sub-fields of optimization and machine learning, including  {\em nonconvex optimization} (\Cref{sec:nonconvex}), {\em non-smooth optimization} (\Cref{sec:nonsmooth} -- we interpret \algname{BPM} as a {\em normalized} variant of \algname{PPM}, and explain normalized gradient descent as an approximation of  \algname{BPM} via iterative linearization of $f$), {\em acceleration} (\Cref{sec:acceleration}), {\em smoothing} (\Cref{sec:moreau} -- \algname{BPM} can be seen as normalized gradient descent on the newly introduced {\em ball envelope} of $f$), {\em adaptive step size selection} (\Cref{sec:adaptive_stepsizes}), and {\em trust-region methods}  (\Cref{sec:trm} -- we interpret \algname{BPM} as an idealized trust-region method). We dedicate a considerable portion of the paper to explaining these connections since we believe this is what most readers will derive most insight from.    
    Our theoretical results  are supported by dedicated sections in the appendix, which provide the corresponding proofs (Appendices \ref{sec:ap_nonsmooth},  \ref{sec:ap_accel} and  \ref{sec:ap_envelope_gd}).
    \item {\bf Powerful convergence theory: convex case.} We establish a \emph{linear} convergence rate for \newalg\ in the non-smooth convex setting (\Cref{sec:convex}), eliminating the reliance on \emph{strong convexity} required by \algname{PPM} for similar performance.
    Moreover, while  \algname{PPM}  with a finite step size can find an approximate solution only, \newalg\ reaches the \emph{exact global minimum} in a \emph{finite number of iterations} using \emph{finite radii} -- see \Cref{table:comparison}.
    \item {\bf Powerful convergence theory: nonconvex case.}  We extend the analysis beyond convexity by introducing the concept of \emph{ball-convexity} and proving that \newalg\ can find the global minimum even under this weaker assumption (\Cref{sec:ball_convex}). In particular, \newalg\ retains the same theoretical guarantees as in the convex setting, bridging the gap between convex and nonconvex optimization. 
    \item {\bf Experiments.} We perform several  toy yet enlightening numerical experiments (see \Cref{fig:bppm} in \Cref{sec:nonconvex}; and \Cref{sec:experiments}), showing the potential of \newalg\ as a method for solving nonconvex optimization problems. 
    \item {\bf Extensions.} Finally, we extend \newalg\ to the distributed optimization setting (\Cref{sec:ap_stochastic}), and further introduce a generalization based on Bregman functions (\Cref{sec:bregman}), providing rigorous convergence results for both.
\end{enumerate}

\begin{table}[t]
    \caption{\textbf{Comparison of \newalg\ and \algname{PPM}.}  
    ``ND+NS'' = Non-Differentiable \& Non-Smooth (i.e., results do not require differentiability nor smoothness);
    ``Lin Cvx'' = linear convergence in the convex setting (without assuming strong convexity);
    ``$K < \infty$'' = finds the exact global minimizer in a finite \# of iterations;
    ``$\gamma_k, t_k < \infty$'' = finds exact global minimizer with a finite step size ($\gamma_k$ for \algname{PPM} and $t_k$ for \algname{BPM}); $d_k \eqdef \norm{x_k - x_{\star}}$; $h_k \eqdef f(x_k) - f_{\star}$. 
    }
    \centering 
    \begin{adjustbox}{width=\columnwidth,center}
    \begin{threeparttable}
    \renewcommand{\arraystretch}{1.5}
    \begin{tabular}{cccccc}
        \toprule
        \bf Method & \bf $\mathbf{1}$-step decrease & \rotatebox[origin=c]{90}{\,\,NS+ND} & \rotatebox[origin=c]{90}{Lin Cvx} & \rotatebox[origin=c]{90}{\,$K<\infty$} & \rotatebox[origin=c]{90}{\,\,$\gamma_k, t_k<\infty$} \\
        \toprule
        \makecell{\algname{PPM} \\ \scriptsize{\citet{guler1991convergence}}}
        & \small{$d_{k+1}^2 \leq \parens{1+\gamma_k\mu}^{-1} d_k^2$} \tnote{{\color{blue}(a)}}
        & \cmark & \xmark & \xmark & \xmark \\
         \midrule
        \makecell{\algname{BPM} \\
        \scriptsize{\Cref{thm:conv-bppm-lin}}}
        & \makecell{\small{$\begin{aligned}&\squeeze h_{k+1} \leq \parens{1 + \frac{t_k}{d_{k+1}}}^{-1} h_k \\ &d_{k+1}^2 \leq d_k^2 - t_k^2\end{aligned}$}}
        & \cmark & \cmark & \cmark & \cmark \\
        \bottomrule
    \end{tabular}
    \begin{tablenotes}
    \scriptsize
    \item [{\color{blue}(a)}] $\mu > 0$ is the \emph{strong convexity} parameter, i.e., a constant such that $f(x) - \nicefrac{\mu}{2} \norm{x}^2$ is convex.
    \end{tablenotes}          
    \end{threeparttable}
    \end{adjustbox}
    \label{table:comparison}
\end{table}

A complete list of notations used in the paper can be found in \Cref{table:unbalanced}.

\subsection{Comments on practical utility}

Since each step of \newalg\ is itself an optimization task, it can be very challenging. 
The method is therefore best understood as an \emph{abstract procedural framework} under the broximal operator oracle, offering a foundation for a class of algorithms with elegant \emph{global convergence guarantees} under weak assumptions.  
While the \ref{eq:bppm} scheme may not be directly implementable, it functions as a conceptual ``master'' method, providing a high-level algorithmic structure that should guide the development of practical variants aiming to approximate the idealized trajectory of \ref{eq:bppm}'s iterates.

For example, one may  {\em approximate} the broximal operator of $f$ by 
(i) the broximal operator of a suitably chosen {\em approximation} (e.g., linearization) of $f$ (see \Cref{sec:ngd} and \Cref{sec:acceleration}), or by  (ii)  {\em approximate} minimization of $f$ over the ball by running some iterative subroutine (e.g., sampling), or
(iii) both.

This paradigm mirrors the approach used in various fields.
The simplest example is the already mentioned classical proximal operator, which is expensive to evaluate \citep{beck2012smoothing}, yet inspires the development of practical methods \citep{parikh2014proximal}. Another instance is Stochastic Differential Equations, where exact solutions are often unavailable, necessitating the use of numerical approximations for practical implementation \citep{kloeden1992numerical}.

\section{BPM and nonconvex Optimization} \label{sec:nonconvex}

The study of nonconvex optimization has a rich history, with recent advances driven largely by its role in training deep neural networks.
Unlike convex problems, where global minimizers can be efficiently found \citep{nemirovski1983problem, nemirovskii1985optimal, nesterov2003introductory, bubeck2015convex}, solving nonconvex problems is generally NP-hard \citep{murty1987someNP}.
This difficulty arises from the complex landscape of nonconvex functions, which can have many local minima and saddle points that can trap optimization algorithms \citep{dauphin2014identifying, jin2021nonconvex}.
Consequently, much of the research in this area has shifted focus from global optimization to more attainable goals, such as finding stationary points or local minima. 
However, local minima can often be far from optimal when compared to global solutions \citep{kleinberg2018alternative}.

One of the key strategies to address these challenges is incorporating stochasticity into gradient-based methods \citep{kleinberg2018alternative, zhou2019toward, jin2021nonconvex}, with algorithms like Stochastic Gradient Descent (\algname{SGD}) and its variants being particularly popular for this application. 

Some nonconvex problems allow global optimization by exploiting structural properties of the objective function.
Examples include one-layer neural networks, where all local minima are guaranteed to be global \citep{feizi2017porcupine, haeffele2017global}.
It is also known that under additional assumptions, \algname{SGD} can converge to a global minimum for linear networks \citep{danilova2022recent, shin2022effects} and sufficiently wide over-parameterized networks \citep{allen2019convergence}.
Beyond neural networks, classes of nonconvex functions, such as weakly-quasi-convex functions and those satisfying the Polyak-Łojasiewicz condition, enjoy global convergence guarantees -- with sublinear and linear rates, respectively \citep{hinder2019near, garrigos2023handbook}.
Local minima are also globally optimal for certain nonconvex low-rank matrix problems, including matrix sensing, matrix completion, and robust PCA \citep{ge2017no}. 

\textbf{New approach to nonconvex optimization.}
Recent research has made significant progress in nonconvex optimization, providing theoretical guarantees for convergence to stationary points and, in some cases, even global minima.
However, the field remains relatively under-explored, with substantial room for innovation.
To extend this line of research, we consider a new approach that goes beyond merely finding stationary points, focusing on methods capable of escaping local minima.
By design, \newalg\ demonstrates this ability if the radius $t_k$ is chosen large enough, as illustrated in a simple experiment (\Cref{fig:bppm}).

\begin{figure}[t]
    \centering
    \subfigure[$t=1$]{
        \includegraphics[width=0.45\columnwidth]{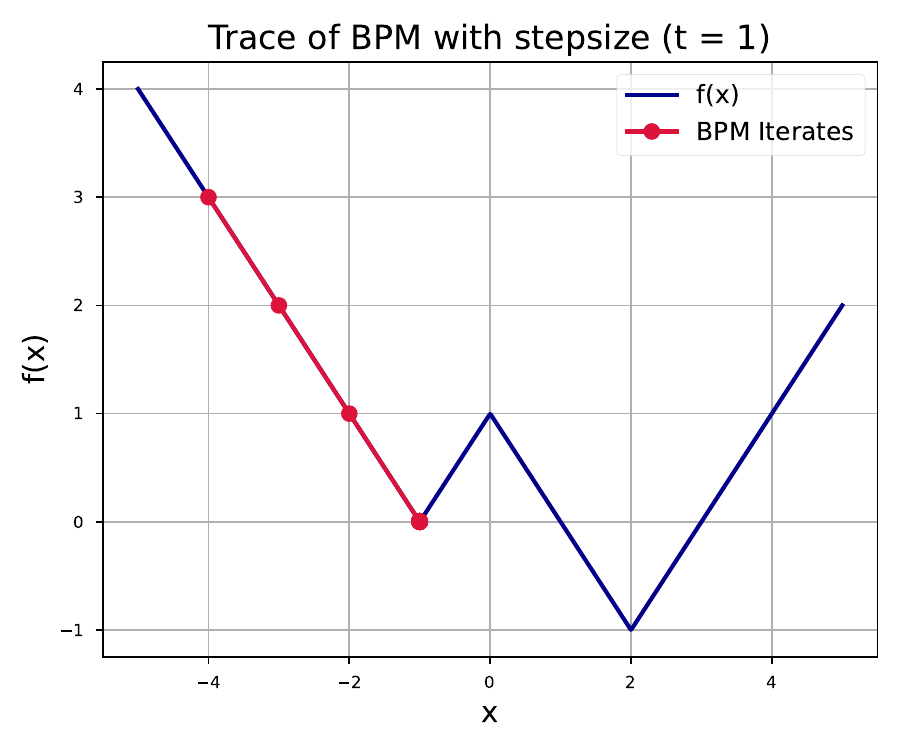}
        \label{fig:first}
    }
    \hfill
    \subfigure[$t=2$]{
        \includegraphics[width=0.45\columnwidth]{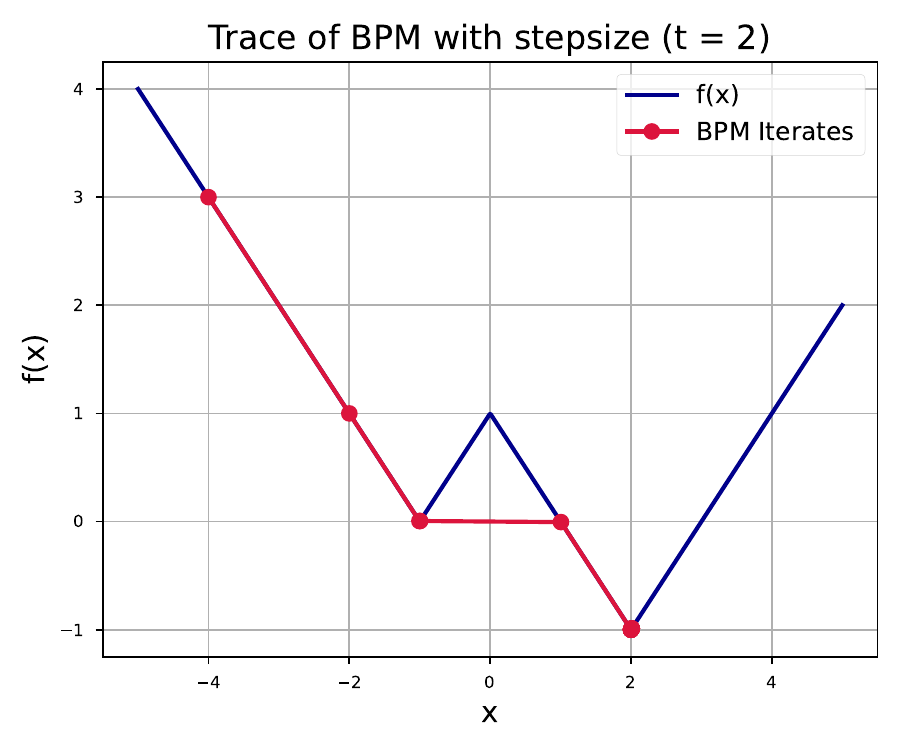}
        \label{fig:second}
    }
    \\
    \subfigure[$t=2.5$]{
        \includegraphics[width=0.45\columnwidth]{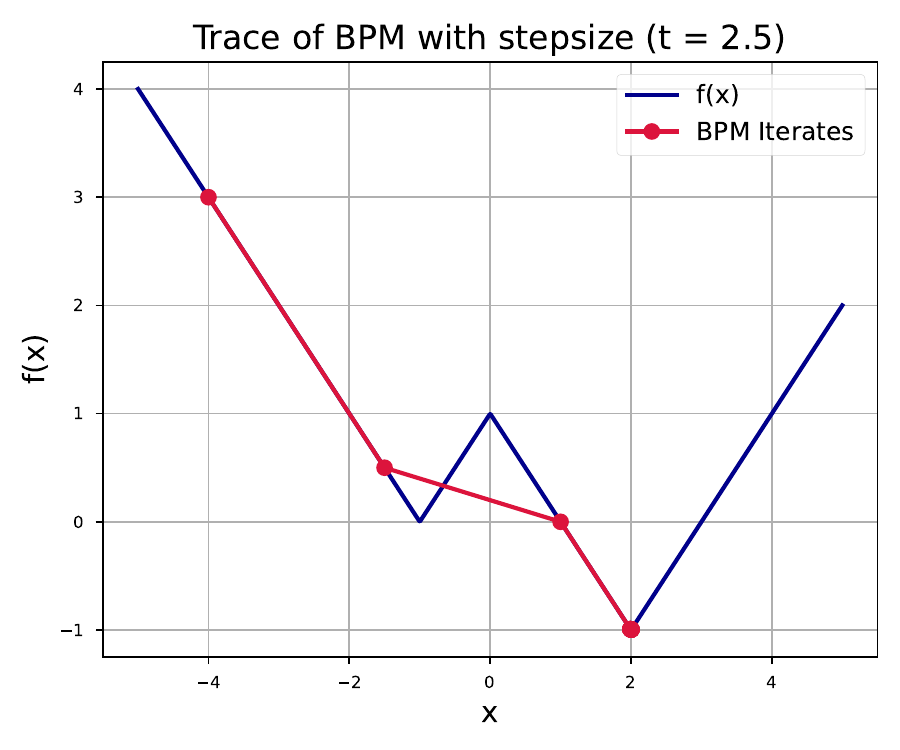}
        \label{fig:third}
    }
    \hfill
    \subfigure[$t=3$]{
        \includegraphics[width=0.45\columnwidth]{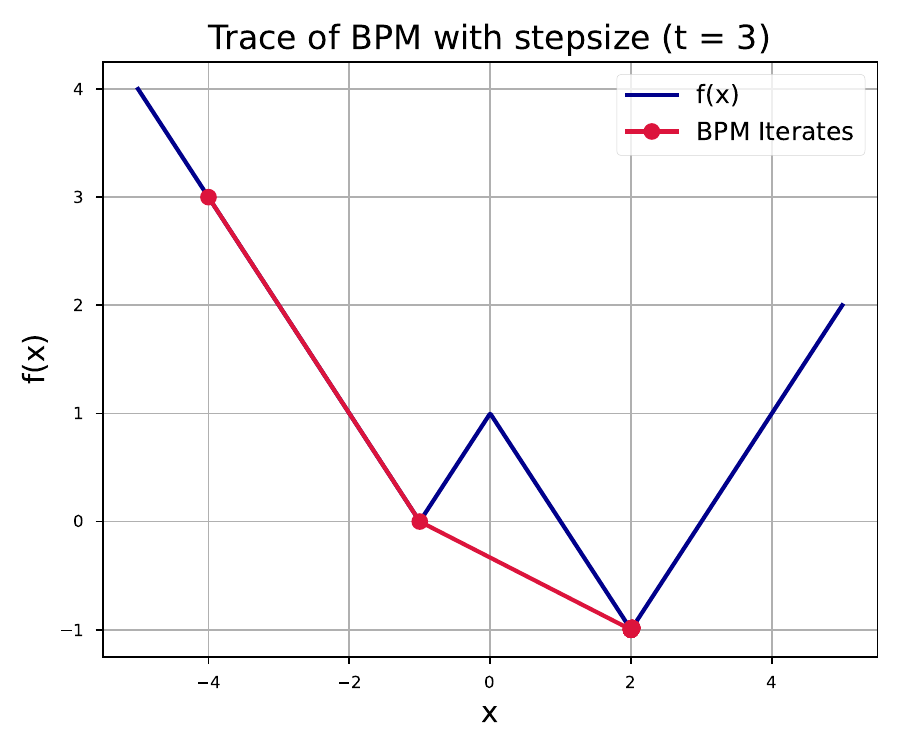}
        \label{fig:fourth}
    }
    \caption{Behavior of \newalg\ on a piecewise linear nonconvex function. The dark blue line represents the function~$f$, while the crimson line illustrates the iterates of \newalg.
    The algorithm is tested for $t \in \cbrac{1, 2, 2.5, 3}$, starting at $x_0=-4$.}
    \label{fig:bppm}
\end{figure}
As the radius $t_k\equiv t$ (kept constant across iterations) increases, the broximal operator gains stronger ability to escape local minima, allowing \newalg\ to converge to the global minimizer (in this example, it is clear that choosing $t_k \equiv t > 2$ is sufficient for the algorithm to achieve this for any initialization). Theoretical analysis confirms that \newalg\ converges to a global minimizer for a specific class of nonconvex functions (\Cref{sec:ball_convex}).
Additionally, numerical experiments (\Cref{sec:experiments}) show that this property holds for a broader range of functions.

\textbf{Sharpness-aware minimization.}
``Idealized'' sharpness-aware minimization (\algname{SAM}) performs the iteration 
\begin{equation} \label{eq:sam}x_{k+1} = x_k -\eta_k \nabla f(z_{k+1}),\end{equation}
where $z_{k+1}\eqdef \arg\max_{z\in B_{\rho}(x_k)} f(z)$, which  is typically approximated through linearization by $ \tilde{z}_{k+1} \eqdef \arg\max_{z\in B_{\rho}(x_k)} \left\{f(x_k) + \langle \nabla f(x_k), z-x_k \rangle \right\}  = x_k + \rho \tfrac{\nabla f(x_k)}{\|\nabla f(x_k)\|}$ \citep{SAM2021}. This leads to the practical \algname{SAM} method
\[x_{k+1} =  x_k -\eta_k \nabla f\left(x_k + \rho \tfrac{\nabla f(x_k)}{\|\nabla f(x_k)\|}\right).\]

Note that \eqref{eq:sam} is structurally similar to the \algname{$\|$PPM$\|$} reformulation of \algname{BPM} (see \Cref{table:reformulations}); the key difference being that $z_{k+1}$ in \algname{SAM} arises from {\em maximizing} $f$ over a ball, while \algname{BPM} employs {\em minimization}. We believe that exploring these similarities may lead to novel insights about \algname{SAM}.

\section{BPM and Non-smooth Optimization}\label{sec:nonsmooth}

Standard gradient-based methods heavily depend on the smoothness of the objective function. However, many real-world problems lack this property, making non-smooth optimization a critical challenge arising in a wide range of applications, such as sparse learning, robust regression, and deep learning \citep{shamir2013stochastic, zhang2019gradient}. A classic example is the support vector machine problem, where using the standard hinge loss makes the objective function non-smooth.
From a practical perspective, even when the objective is smooth, its smoothness constant -- commonly used to determine hyperparameters like the step size -- is often unknown. 
From a theoretical standpoint, the analysis of non-smooth problems typically differs significantly from the smooth case, requiring different tools and techniques.
To effectively address non-smooth problems, several strategies have been developed.
Two notable approaches in this domain are proximal-type updates and normalized gradient methods, both of which reveal a profound connection to the framework we propose.

\subsection{Proximal Point Method.}

Proximal algorithms \citep{rockafellar1976monotone} are a cornerstone of optimization.
They are powerful since their convergence does not rely on the smoothness of the objective function \citep{richtarik2024unified}. 
This property makes them especially attractive for deep learning applications, where loss functions often lack smoothness \citep{zhang2019gradient}.
Central to these methods is the \emph{proximal operator}, defined as
\begin{align*}
    \squeeze \ProxSub{ f}{x} \eqdef \underset{z\in\R^d}{\arg\min} \brac{f(z) + \frac{1}{2} \cdot \norm{z-x}^2},
\end{align*}
where $f : \R^d \to \R \cup \{+\infty\}$ is an extended real-valued function.
It is known that if $f$ is a proper, closed and convex function, then $\ProxSub{f}{x}$ is a singleton for all $x\in\R^d$ \citep{bauschke2011convex, beck2017first}. Furthermore, if $f$ is differentiable, then for any $\gamma>0$ the proximal operator satisfies the equivalence
\begin{align}\label{eq:prox_impl}
    \squeeze z = \ProxSub{\gamma f}{x} \quad\iff\quad z + \gamma\nabla f(z) = x.
\end{align}
The simplest proximal algorithm is the Proximal Point Method (\algname{PPM}) \citep{moreau1965proximite, martinet1970regularisation}.
Originally introduced to address problems involving variational inequalities, \algname{PPM} has since been adapted to stochastic settings to address the challenges of large-scale optimization. This led to the development of Stochastic Proximal Point Methods (\algname{SPPM}), which have been extensively studied and refined over time \cite{bertsekas2011incremental, khaled2022faster, anyszka2024tighter, li2024convergence}.

In its simplest form, the update rule of \algname{PPM} is
\begin{align}\label{eq:ppm_update}
    \squeeze x_{k+1} = \ProxSub{\gamma_k f}{x_k} \tag{\algname{PPM}}
\end{align}
for some step size $\gamma_k>0$.
Using the equivalence given in~\eqref{eq:prox_impl}, the above expression can be rewritten as
\begin{align}\label{eq:prox_implicit}
    \squeeze x_{k+1} = x_k - \gamma_k \nabla f (x_{k+1}).
\end{align}
While similar in form to \algname{GD}, a key distinction lies in the implicit nature of the update: the gradient is evaluated at the \emph{new} iterate 
$x_{k+1}$. 
This implicitness provides enhanced stability in practice \citep{ryu2014stochastic}.

As shown in the following lemma, \newalg\ and \algname{PPM} share a deep connection that goes beyond their similar names.
\begin{restatable}{lemma}{LEMMABROXPROX}
    \label{lemma:brox_prox}
    Let $f:\R^d \to \R$ be a differentiable convex function, and let $x_{k+1} = \BProxSub{t_k}{f}{x_k}$ be the iterates of \newalg. Provided that $x_{k+1}$ is not optimal,
    \begin{eqnarray*}
        \squeeze \boxed{x_{k+1} =  \ProxSub{\frac{t_k}{\norm{\nabla f(x_{k+1})}} f}{x_k}}
    \end{eqnarray*}
\end{restatable}
Consequently, for differentiable objectives, the update rule of \ref{eq:bppm} becomes
\begin{align*}
    \squeeze \boxed{x_{k+1} = x_k - \tfrac{t_k}{\norm{\nabla f(x_{k+1})}} \cdot \nabla f(x_{k+1})}
\end{align*}
(see \Cref{thm:alg_update_diff}), which can be interpreted as \emph{normalized} \algname{PPM} (\algname{$\|$PPM$\|$}).
Hence, computing broximal operator is equivalent to evaluating the proximal operator with a carefully chosen, adaptive step size.
Alternatively, computing the proximal operator of a convex function can be viewed as finding the point on the sphere of radius $t = \gamma\norm{\nabla f(\ProxSub{\gamma f}{x})}$ centered at $x$ that minimizes the function value (Lemma \ref{cor:prox_brox}).
Building on this reformulation, the \ref{eq:ppm_update} update rule can be expressed as
\begin{align*}
    \squeeze x_{k+1} = \ProxSub{\gamma_k f}{x_k}
    = \underset{z \in B_{t_k}(x_k)}{\arg\min} f(z),
\end{align*}
where $t_k = \gamma_k \norm{\nabla f(\ProxSub{\gamma_k f}{x_k})}$.
This highlights the motivation for \newalg, which originates from using alternative choices for the radius sequence~$\cbrac{t_k}_{k\geq0}$.

\subsection{Normalized gradient descent} \label{sec:ngd}
Normalized Gradient Descent (\algname{$\|$GD$\|$}) is another popular approach for non-smooth optimization, particularly useful when gradient norms provide little information about the appropriate choice of the step size. 
Originally introduced by \citet{nesterov1984minimization} for differentiable quasi-convex objectives, it was later extended by \citet{kiwiel2001convergence} to include upper semi-continuous quasi-convex functions, and analyzed in the stochastic setting by \citet{hazan2015beyond}.
To illustrate the intuition behind it, consider the simple example $f(x)=|x|$. 
In this case, the gradient norm is $1$ everywhere except the optimum, offering no guidance on the right choice of the step size. 
Normalization addresses this issue by removing the influence of the gradient norm, while preserving the descent direction (negative gradient). 
Beyond handling non-smoothness, \algname{$\|$GD$\|$} has demonstrated superior performance in nonconvex settings, escaping saddle points more effectively than standard \algname{GD} \cite{murray2019revisiting}.

Although normalization is intuitively justified, a rigorous explanation for its use has been missing.
It is well-established that applying \algname{PPM} to the linear approximation of $f$ at the current iterate yields the update rule of \algname{GD}.
However, an analogous result connecting \algname{$\|$GD$\|$} to a principled framework has yet to be established. 
Interestingly, such a result can be derived by adopting a similar approach, replacing \algname{PPM} with \newalg.
\begin{restatable}{theorem}{THMLINEARPROX}
    \label{thm:linearized-bppm}
    Define $f_k(z) \eqdef f(x_k) + \inp{\nabla f(x_k)}{z - x_k}$ and let $x_{k+1} = \BProxSub{t_k}{f_k}{x_k}$ be the iterates of \newalg\ applied to the first-order approximation of $f$ at the current iterate.
    Then, the update rule is equivalent to
    \begin{align}\label{eq:lin_approx_update}
        \squeeze x_{k+1} = x_k - \frac{t_k}{\norm{\nabla f(x_k)}} \cdot \nabla f(x_k).
    \end{align}
\end{restatable}
Just as \algname{GD} naturally arises from \algname{PPM}, we show that \algname{$\|$GD$\|$} follows directly from the mechanics of \newalg. 
This establishes normalization as an intrinsic property of the broximal operator, rather than a mere heuristic, providing a robust theoretical foundation for \algname{$\|$GD$\|$} and validating \newalg's design and applications.

\section{BPM and Acceleration}\label{sec:acceleration}

When the objective $f$ belongs to a certain function class (including convex functions, and defined in Section \ref{sec:ball_convex}), the result in \Cref{lemma:brox_prox} still holds, allowing for the derivation of an explicit update rule for \newalg.
\begin{restatable}{theorem}{THMALGUPDIFF}\label{thm:alg_update_diff}
    Let $f:\R^d\to \R$ be a differentiable function satisfying \Cref{as:defining_rel}.  Let $x_{k+1} = \BProxSub{t_k}{f}{x_k}$ be the iterates of \newalg. Provided that $x_{k+1}$ is not optimal,
    \begin{align}\label{eq:alg_update_diff}
        \squeeze x_{k+1} = x_k - \frac{t_k}{\norm{\nabla f(x_{k+1})}} \cdot \nabla f(x_{k+1}).
    \end{align}
\end{restatable}
The first observation is the similarity between \eqref{eq:lin_approx_update} and \eqref{eq:alg_update_diff}.
However, the above update rule is \emph{doubly implicit}, as both the gradient and the effective step size depend on the next iterate~$x_{k+1}$. 
A similar doubly implicit structure arises in $p$-th order proximal point methods \citep{nesterov2023inexact}. 
In particular, the $p$-th order proximal operator is defined as
\begin{align*}
    \squeeze \ProxPSub{\gamma f}{x} \eqdef \underset{z\in\R^d}{\arg\min} \brac{\gamma f(z) + \frac{1}{(p+1)} \cdot \norm{z - x}^{p+1}}.
\end{align*}
The corresponding $p$-th order Proximal Point Method (\algname{PPM}${\green \sf ^p}$) iterates
\begin{align}\label{eq:ppmp_update}
    \squeeze x_{k+1} =\ProxPSub{\gamma f}{x_k}, \tag{\algname{PPM}${\green \sf ^p}$}
\end{align}
which can be reformulated (see \Cref{thm:ppmpimplicit}) as
\begin{align}\label{eq:proxp}
    \squeeze x_{k+1} = x_k - \parens{\frac{\gamma}{\norm{\nabla f(x_{k+1})}^{p-1}}}^{\nicefrac{1}{p}} \cdot \nabla f(x_{k+1}).
\end{align}
A notable feature of higher-order proximal methods is their accelerated convergence rate of $\cO(\nicefrac{1}{K^p})$ for convex objective functions.
\newalg\ can achieve the same accelerated rate by carefully selecting the radius $t_k$ of the ball. 
Specifically, choosing $t_k = \parens{\gamma \norm{\nabla f(x_{k+1})}}^{1/p}$ leads to the update rule
\begin{align*}
    \squeeze x_{k+1} = \BProxSub{t_k}{f}{x} \overset{\eqref{lemma:brox_prox}}{=} x_k - \frac{t_k}{\norm{\nabla f(x_{k+1})}} \cdot \nabla f(x_{k+1}),
\end{align*}
which aligns with that of \algname{PPM}${\green \sf ^p}$ in \eqref{eq:proxp}, enabling \newalg\ to inherit the favorable convergence properties of higher-order proximal methods (\Cref{thm:bppmp_rate}).

\newalg\ can also achieve acceleration in the classical Nesterov sense, drawing on the work of \citet{anh2020understanding}, who interpret the Accelerated Gradient Method (\algname{AGM}) as an approximation of \algname{PPM}.
Specifically, let $\{y_k\}_{k\geq0}$ be an auxiliary sequence, define $l_{y}(x) \eqdef f(y) + \inp{\nabla f(y)}{x-y}$, $u_{y}(x) \eqdef f(y) + \inp{\nabla f(y)}{x-y} + \frac{L}{2} \norm{x-y}^2$, and fix $x_0=y_0\in\R^d$. 
Now, consider the algorithm
\begin{align}\label{eq:agm_bprox_update}
    \begin{split}
        \squeeze x_{k+1} &= \BProxSub{t^x_{k+1}}{l_{y_k}}{x_k}, \\
        \squeeze y_{k+1} &= \BProxSub{t^y_{k+1}}{u_{y_k}}{x_{k+1}},
    \end{split}
    \tag{\algname{A-BPM}}
\end{align}
where
\begin{align*}
    t^x_{k+1} &\eqdef \squeeze \frac{k+1}{2L} \norm{\nabla l_{y_k}(\ProxSub{\parens{\nicefrac{k+1}{2L}} l_{y_k}}{x_k})}, \\
    t^y_{k+1} &\eqdef \squeeze \frac{k+1}{2L} \norm{\nabla u_{y_k}(\ProxSub{\parens{\nicefrac{k+1}{2L}} u_{y_k}}{x_{k+1}})}.
\end{align*}
The convergence guarantee of the method is characterized in the theorem below.
\begin{restatable}{theorem}{THMAGMBRPOX}\label{thm:agm_bprox}
    Let $f:\R^d \to \R$ be  convex and $L$--smooth. Then the iterates of \ref{eq:agm_bprox_update} satisfy $f(x_K) - f_\star \leq \frac{2 L d_0^2}{K(K+1)}$.
\end{restatable}
Consequently, we recover the well-known $\cO\parens{\nicefrac{1}{K^2}}$ accelerated convergence rate of \algname{AGM}.

\section{BPM and Smoothing}\label{sec:moreau}

The proximal operator is closely related to the Moreau envelope \citep{moreau1965proximite}, also known as the Moreau-Yosida regularization.
It is well-established that running proximal algorithms on the original objective $f$ is equivalent to applying gradient methods to its Moreau envelope \cite{ryu2014stochastic}.
A key observation is that the algorithm's effectiveness is preserved, as the minima of the original objective and the Moreau envelope coincide \cite{planiden2016strongly,planiden2019proximal,li2024power}.
The Moreau envelope has applications beyond proximal minimization algorithms, finding use in areas such as personalized federated learning \cite{t2020personalized} and meta-learning \cite{mishchenko2023convergence}.

\subsection{Ball envelope and normalized gradient descent}

Analogously, we define a concept of an envelope function associated with the ball-proximal operator.
\begin{definition}[Ball envelope]
    \label{def:ball-envelope}
    The \emph{ball envelope} with radius $t > 0$ of  $f: \R^d \mapsto \R \cup \{+\infty\}$ is given by 
    \begin{align}
        \label{eq:def:ballenve}
        \squeeze \BMoreauSub{t}{f}{x} \eqdef \min \limits_{z \in B_t(x)} f(z).
    \end{align}
\end{definition}
The ball envelope has several interesting properties and enables theoretical insights into the behavior of the broximal operator and its applications.
One noteworthy observation is the relationship between the sets of minimizers of $f$ and its ball envelope.
Specifically, it turns out that $\cX_{N} = \cbrac{x: \dist{x}{\cX_f} \leq t} = \cX_f + B_t(0)$, where $\cX_f$ and $\cX_N$ are the sets of minimizers of $f$ and $N_f^t$, respectively (further details can be found in \Cref{sec:ball-envelope:app}).
This key observation enables us to interpret \newalg\ applied to $f$ as \algname{GD} on the ball envelope $N^t_f$, analogous to the interpretation of \algname{PPM} on $f$ as \algname{GD} on the Moreau envelope (which is known for its ``smoothing'' properties).

As in the standard proximal setting, the result requires a smoothness assumption.
Recall that  a differentiable function~$f:\R^d \mapsto \R$ is $L$--smooth if
    \begin{align*}
        \squeeze f(x) - f(y) - \inner{\nabla f(y)}{x - y} \leq \frac{L}{2}\norm{x - y}^2\, \forall x, y \in \R^d.
    \end{align*}
With the assumptions set, we can state the equivalence result.
\begin{restatable}{theorem}{THMBPPMGD}\label{thm:bppm-gd-relation}
    Let $f:\R^d\to \R$ be convex and $L$--smooth,  and let $x_{k+1} = \BProxSub{t_k}{f}{x_k}$ be the iterates of \newalg. Provided that $x_{k+1}$ is not optimal,
    \begin{align}\label{eq:moreau_gd}
        \squeeze \boxed{x_{k+1} = x_k - \tfrac{t_k}{\|\nabla N^{t_k}_f(x_k)\|} \cdot \nabla \BMoreauSub{t_k}{f}{x_k}}
    \end{align}
\end{restatable}
Therefore, \newalg\ can be viewed as normalized gradient descent (\algname{$\|$GD$\|$}) on the ball envelope.

\section{BPM and Adaptive Step Sizes} \label{sec:adaptive_stepsizes}

In gradient-based methods, selecting an appropriate step size is a notoriously challenging problem: choosing a step size that is too small results in slow convergence, while a step size that is too large risks divergence. 
This challenge has driven significant research into the development of adaptive methods that adjust the learning rate dynamically based on algorithm history \citep{polyak1987introduction, bach2019universal, malitsky2019adaptive, horvath2022adaptive, yang2023adaptive}. However, all of these algorithms come with inherent limitations and trade-offs.

A similar dilemma arises in trust-region methods \citep{conn2000trust, nocedal2006numerical}. 
If the trust-region radius is too small, progress is slow; if it is too large, the model function may fail to approximate the objective accurately, potentially resulting in an increase in the function value at the next iterate $x_{k+1}$. 
To address this, trust-region methods typically rely on heuristic rules to modify the radius, adjusting it up or down based on predefined criteria.

Proximal methods offer a different approach. They can, in theory, achieve convergence in a single step if the step size~$\gamma$ is sufficiently large \citep{guler1991convergence}. However, this advantage hinges on the assumption that the proximal operator is computationally easy to evaluate. 
In practice, each proximal step involves solving a nested optimization problem, which is often computationally expensive, making proximal methods more conceptual than practical for many applications.

\newalg\ preserves the desirable properties of proximal methods while facing similar computational challenges.  
Similar to \algname{PPM} -- which, as shown in Lemma \ref{lemma:brox_prox}, is a special case of \newalg\ -- it retains the ability to converge in a single step provided that a sufficiently large step size $t$ is used (\Cref{sec:convex}). However, just like \algname{PPM}, it inherits the drawback that solving the local optimization subproblem can be computationally challenging in general.

To gain a clearer insight into the step size sequence generated by \newalg, let us examine the setting where $f$ is differentiable. As demonstrated in \Cref{thm:alg_update_diff}, the algorithm can then be interpreted as normalized \algname{PPM} with the step size~$\frac{t_k}{\norm{\nabla f(x_{k+1})}}$. In the smooth setting, \newalg\ can be further expressed as \algname{GD} on the ball envelope with the adaptive step size given by
\begin{align}\label{eq:ada_step_bpm}
    \squeeze \gamma_{k,t}^E \eqdef \frac{t_k}{\norm{\nabla N^t_f\rbrac{x_k}}}
    \overset{\eqref{lemma:N3}}{=} \frac{t_k}{\norm{\nabla f(x_{k+1})}}
\end{align}
(\Cref{thm:bppm-gd-relation}). 
Unlike traditional methods, where step sizes decrease over time \citep{robbins1951stochastic},  the step size sequence of \newalg\ is \emph{increasing} even if the radius~$t_k$ is fixed across iterations (i.e., $t_k \equiv t>0$; see \Cref{thm:conv-bppm-lin}(v)).
To better understand the implications, we compare the step size~$\gamma_{k,t}^E$ in \eqref{eq:ada_step_bpm} with the classic Polyak step size \citep{polyak1987introduction}. The Polyak step size for \algname{GD} applied to the ball envelope is defined as
\begin{align}\label{eq:polyak_step}
    \squeeze \gamma_k^{P} \eqdef \frac{\BMoreauSub{t_k}{f}{x_k} - \BMoreauSub{t_k}{f}{x_{\star}}}{\norm{\nabla \BMoreauSub{t_k}{f}{x_k}}^2}
    \overset{\eqref{lemma:N3}, \eqref{lemma:N4}}{=} \frac{f(x_{k+1}) - f_{\star}}{\norm{\nabla f(x_{k+1})}^2},
\end{align}
which corresponds to the Polyak step size for the original objective $f$ evaluated at the next iterate.

Comparing \eqref{eq:ada_step_bpm} and \eqref{eq:polyak_step}, we see that the relationship between~$\gamma_{k,t}^E$ and~$\gamma_k^{P}$ is determined by the interplay between $t_k$ and $\nicefrac{(f(x_{k+1}) - f_{\star})}{\norm{\nabla f(x_{k+1})}}$. In particular, choosing $t_k=\nicefrac{(f(x_{k+1}) - f_{\star})}{\norm{\nabla f(x_{k+1})}}$ recovers the Polyak step size exactly.
Furthermore, when $f$ is $L$--smooth and $\mu$--strongly convex, the Polyak step size is uniformly bounded above and below by
\begin{align*}
    \squeeze \frac{1}{2L} \leq \frac{f(x_{k+1}) - f_{\star}}{\norm{\nabla f(x_{k+1})}^2} \leq \frac{1}{2\mu}.
\end{align*}
Under the same conditions, we also have
\begin{align*}
    \squeeze \frac{t_k}{L \norm{x_{k+1} - x_{\star}}} \leq \frac{t_k}{\norm{\nabla f(x_{k+1})}} \leq \frac{t_k}{\mu \norm{x_{k+1} - x_{\star}}}.
\end{align*}
These bounds coincide for~$t_k = \nicefrac{\norm{x_{k+1} - x_{\star}}}{2}$.
On the other hand, if the step size~$t_k \equiv t$ is kept constant, \newalg\ can initially take smaller steps, but as the iterates approach the solution, the lower bound on~$\gamma_{k,t}^E$ increases and can eventually surpass the upper bound associated with the Polyak step size, resulting in \newalg\ taking larger steps.

\section{BPM and Trust Region Methods}
\label{sec:trm}

Trust region methods represent another complementary connection to our work. 
These methods trace their origins to \citet{levenberg1944method}, who introduced a modified Gauss-Newton method to solve nonlinear least squares problems.
Their widespread recognition followed the influential work of \citet{marquardt1963algorithm}.
At their core, trust region methods minimize the function~$f$ by iteratively approximating it within a neighborhood around the current iterate, referred to as the \emph{trust region}, using a model $m_k(x)$ (e.g., a quadratic approximation) \citep{conn2000trust}. 
The trust region is typically defined as $B_{t_k}(x_k) \eqdef \brac{z\in\R^d: \norm{z-x_k} \leq t_k}$, where $t_k$ is the \emph{trust-region radius}.\footnote{Alternatively, more complex and problem-specific trust regions, like ellipsoidal or box-shaped ones, could be used.} 
The next iterate is determined by solving the constrained optimization problem
\begin{align*}
    \squeeze x_{k+1} = \underset{z\in B_{t_k}(x_k)}{\arg\min} m_k(x_k),
\end{align*}
after which the radius is adjusted, and the process is repeated.
While intuitive and reasonable, this approach is not conceptually aligned with the abstraction level of \algname{GD}, which can be interpreted as minimization of a quadratic upper bound on $f$, applied without constraints. 
Why, then, should such constraints arise in trust region methods?

\newalg\ provides a new perspective, elevating trust region methods to a more principled framework.
It functions as the \algname{PPM} of trust region methods, naturally explaining the emergence of neighborhoods in the optimization process.
Traditional trust region methods rely on \emph{approximate} models and introduce constraints to compensate for their limitations -- since local approximations of~$f$ become less reliable farther from~$x_k$, the trust-region radius must be continuously adjusted to maintain accuracy. Within this framework,~$t_k$ acts as a control mechanism for approximation quality.
In contrast, \newalg\ can be considered a ``supercharged'' version of these methods, where the model is assumed to be perfect, eliminating the need for radius adjustment when optimizing directly on~$f$.
By interpreting \newalg\ as a ``master'' trust region method, we can unify the two approaches at a higher level of abstraction.
In this context, trust regions emerge naturally -- not as a heuristic, but as an inherent component of the optimization process.

Although this paper does not focus on trust region methods explicitly, \newalg\ serves as a conceptual bridge to that field. 
By presenting a globally convergent framework without approximation, our approach lays the groundwork for advancing the theory and practice of trust region methods.

\section{Convergence Theory: Convex Case}\label{sec:convex}

We first analyze the algorithm assuming that the objective function is convex.\footnote{Extension to the nonconvex regime is presented in \Cref{sec:ball_convex}.}
In this setting, the broximal operator has several favorable properties. 
For example, $\BProxSub{t}{f}{x}$ is always a singleton and lies on the boundary of $B_t(x)$ unless $\BProxSub{t}{f}{x} \subseteq \cX_f$, meaning that the algorithm has reached the set of global minimizers.
In other words, at each iteration, \newalg\ moves from $x_k$ to a new point $x_{k+1}$ located on the boundary of $B_{t_k}(x_k)$, effectively traveling a distance of~$t_k$ at each step (possibly except for the very last iteration). Thus, we sometimes refer to the radius $t_k$ as the \emph{step size}. Let $d_k \eqdef \norm{x_k - x_{\star}}$.

The following theorem presents the main results.
\begin{restatable}{theorem}{THMBPPMCONVLIN}
    \label{thm:conv-bppm-lin}
    Assume that $f: \R^d \mapsto \R \cup \{+\infty\}$ is proper, closed and convex, and let $\{x_k\}_{k\geq0}$ be the iterates of \newalg\ run with any sequence of positive radii $\{t_k\}_{k\geq0}$, where $x_0\in {\rm dom} f$.  Then 
    \begin{enumerate}[label=(\roman*)]
        \item If $\cX_f\cap B_{t_k}(x_k)\neq\emptyset$, then $x_{k+1}$ is optimal.\label{pt:opt}
        
        \item \label{pt:dist_decr_bpm}
        If $\cX_f\cap B_{t_k}(x_k)=\emptyset$, then $\norm{x_{k+1} - x_k} = t_k$. Moreover, for any $x_{\star}\in\cX_f$, we have
        \begin{align*}
            \norm{x_{k+1} - x_{\star}}^2 \leq \norm{x_k - x_{\star}}^2 - t_k^2,
        \end{align*}
        \begin{align*}
            \textnormal{dist}^2(x_{k+1}, \cX_f) \leq \textnormal{dist}^2(x_k, \cX_f) - t_k^2.
        \end{align*}
        
        \item \label{pt:dist_bpm_lin}
        If $\sum_{k=0}^{K-1} t_k^2 \geq \textnormal{dist}^2(x_0, \cX_f)$, then $x_K\in\cX_f$.
        
        \item \label{pt:1step_bpm_lin}
        For any $k\geq 0$, 
        \begin{eqnarray*}
            \squeeze f(x_{k+1}) - f_{\star} \leq \parens{1 + \frac{t_k}{\norm{x_{k+1} - x_{\star}}}}^{-1} \parens{f(x_k) - f_{\star}}.
        \end{eqnarray*}

        \item If $f$ is differentiable, then $\norm{\nabla f(x_{k+1})} \leq \norm{\nabla f(x_k)}$ for all $k\geq 0$, and 
        \begin{eqnarray*}
            \squeeze \sum_{k=0}^{K-1} \parens{\frac{t_k}{\sum_{k=0}^{K-1} t_k} \norm{\nabla f(x_{k+1})}}
            \leq \frac{f(x_0) - f_{\star}}{\sum_{k=0}^{K-1} t_k}.
        \end{eqnarray*}
    \end{enumerate}
\end{restatable}

\begin{proof}[Proof sketch] The complete proof of \Cref{thm:conv-bppm-lin} is presented in \Cref{sec:ap_convex}. Here, we provide a brief sketch of its final part to emphasize the main ideas underlying the argument.
    Let us consider some iteration $k$ such that $x_{k+1} \not\in \cX_f$ (otherwise, the problem is solved in $1$ step).
    We start the proof by invoking \Cref{thm:3}, stating that
    \begin{align}\label{eq:qeobnre}
        f(y) - f(u) \geq c_t(x)\inner{x - u}{y - u}
    \end{align}
    for some $c_t(x)\geq 0$, $u \in \BProxSub{t}{f}{x}$ and all $y \in \R^d$.
    Substituting $y = x = x_k$ and $t=t_k$, we can bound $f(x_{k+1}) - f_{\star}$ by
$
  f(x_k) - f_{\star} - c_{t_k}(x_k) \norm{x_{k+1}-x_k}^2. 
$
    Next, applying the same inequality with $x = x_k$, $y = x_\star \in \cX_f$, and using the Cauchy-Schwarz inequality, we obtain
    \begin{eqnarray*}
        f(x_{k+1}) - f_{\star} \leq c_{t_k}(x_k) \norm{x_k-x_{k+1}} \norm{x_{k+1} - x_{\star}}.
    \end{eqnarray*}
    Since $x_{k+1} \not\in \cX_f$, it follows that 
    \begin{eqnarray*}
        \squeeze \parens{f(x_{k+1}) - f_{\star}} \frac{\norm{x_k-x_{k+1}}}{\norm{x_{k+1} - x_{\star}}} \leq c_{t_k}(x_k) \norm{x_k-x_{k+1}}^2.
    \end{eqnarray*}
    Applying this bound  and using the fact that $\norm{x_k - x_{k+1}} = t_k$, we rearrange the terms to obtain (iv).
\end{proof}

\begin{restatable}{corollary}{CORBPPMCONVLIN}\label{cor:conv-bppm-lin}
    Let the assumptions of \Cref{thm:conv-bppm-lin} hold. Then, for any $K\geq 1$, the iterates of \newalg\ satisfy
    \begin{align}\label{eq:bppm_lin_rate}
        \squeeze f(x_K) - f_{\star} \leq \prod \limits_{k=0}^{K-1} \parens{1 + \frac{t_k}{d_0}}^{-1} \parens{f(x_0) - f_{\star}}.
    \end{align}

\end{restatable}

Several important observations are in order:

\phantom{X} $\bullet$ \textbf{Large step sizes travel far.} Note that \Cref{thm:conv-bppm-lin} holds without any upper bound on the radii. Therefore, \newalg\ converges even after a single iteration provided that the radius~$t_0$ is large enough: $t_0 \geq \textnormal{dist}(x_0, \cX_f)$. \\
\phantom{X} $\bullet$  \textbf{Finite convergence.} If we fix the radii sequence to be constant, i.e., if $t_k \equiv t > 0$, then \ref{pt:dist_bpm_lin} implies convergence to the \emph{exact optimum} in a \emph{finite number of steps}. Indeed, $Kt^2 = \sum_{k=0}^{K-1} t^2 \geq \textnormal{dist}^2(x_0, \cX_f)$ holds for $K = \ceil{\nicefrac{\textnormal{dist}^2(x_0, \cX_f)}{t^2}}.$ This is in stark contrast to proximal methods, which never reach the exact solution. \\
\phantom{X} $\bullet$  \textbf{Linear convergence without smoothness nor strong convexity}. Surprisingly, inequality \eqref{eq:bppm_lin_rate} posits linear convergence of \newalg\ without assuming smoothness nor strong convexity (or any relaxation thereof, such as the P{\L} condition, which normally leads to linear convergence \citep{karimi2016linear}).

\begin{remark}\label{rem:sublinear_conv}
    Under the assumptions of \Cref{thm:conv-bppm-lin}, the iterates of \newalg\ with a fixed step size $t_k \equiv t > 0$ satisfy
    \begin{align*}
        \squeeze f(x_K) - f_\star \leq \frac{2d_0}{2d_0 + t} \cdot\frac{d_0^2}{2Kt^2} \parens{f(x_0) - f_{\star}}
    \end{align*}
    (see \Cref{thm:conv-bppm-convex}).
    This bound outperforms~\eqref{eq:bppm_lin_rate} when~$K$ is small and $t$ is large ($K\in\{1,2,3\}$ and $t\approx \norm{x_0 - x_{\star}}$).
    However, such a choice of $t$ is impractical, as the initial ``local'' search space essentially contains a global solution.
\end{remark}

\section{Ball oracles in the literature}\label{sec:ball_oracles}

Several prior works have leveraged the ability to minimize a convex function over a ball constraint. \citet{carmon2020acceleration} developed accelerated algorithms within this framework. The works of \citet{carmon2021thinking} and \citet{asi2021stochastic} applied it to minimizing the maximum loss, while \citet{carmon2023resqueing} and \citet{jambulapati2024closing} used it to design parallel optimization methods. Subsequent efforts have refined these approaches by improving logarithmic factors \citep{carmon2022optimal} and generalizing to non-Euclidean geometries \citep{adil2024convex}. Moreover, \citet{weigand2024adversarial} proposed a method that can be interpreted as a continuous-time gradient flow of the \algname{BPM}.

However, our motivation departs significantly from this line of work. Existing approaches largely treat the ball minimization oracle as a mechanism for implementing MS oracles \citep{monteiro2013accelerated}, relying on differentiability and convexity of the objective function, as well as additional regularity conditions such as Lipschitz continuity, smoothness, H{\"o}lder continuity, or stability properties of the Hessian. In contrast, our starting point was to reframe the penalty in the proximal operator as a hard constraint, with the goal of designing a method capable of effectively navigating nonconvex loss landscapes. This led us to formulate an abstract meta-algorithm and investigate its theoretical properties---initially in the convex setting, and then extending to more general, possibly nonconvex, objectives.
Unlike prior work, our focus is on understanding the ball-proximal operator itself, rather than using it as a means to an end.

\section*{Acknowledgements}

The research reported in this publication was supported by funding from King Abdullah University of Science and Technology (KAUST): i) KAUST Baseline Research Scheme, ii) Center of Excellence for Generative AI, under award number 5940, iii) SDAIA-KAUST Center of Excellence in Artificial Intelligence and Data Science.

\section*{Impact Statement}
This paper presents work whose goal is to advance the field of 
Machine Learning. There are many potential societal consequences 
of our work, none which we feel must be specifically highlighted here.



\bibliography{main}
\bibliographystyle{icml2025}

\newpage
\appendix

\onecolumn

\part*{Appendix}

\tableofcontents

\clearpage

\section{Convergence Theory: Beyond Convexity}\label{sec:ball_convex}

Convex geometry offers valuable insights into the properties of the objective function, enabling the design of efficient algorithms for finding globally optimal solutions. Consequently, many optimization methods rely on the convexity assumption to provide theoretical guarantees. However, many problems of practical interest involve functions that fail to be convex, while still retaining certain structural similarities to convex functions \citep{kleinberg2018alternative, hardt2018gradient, zhou2019sgd}. This motivates the search for broader function classes for which theoretical convergence results can still be provided.

In this section, we introduce \emph{ball-convexity}, a relaxed notion that extends standard convexity while maintaining sufficient structure to enable theoretical analysis. We demonstrate how this property preserves key inequalities used in our proofs, allowing us to extend the convergence guarantees of \newalg\ beyond the convex regime.

The proof of \Cref{thm:conv-bppm-lin} heavily relies on inequality~\eqref{eq:qeobnre}. It turns out that such an inequality holds beyond the convex case, and the method can be analyzed based exclusively on this weaker condition. Motivated by this, we introduce the following assumption:

\begin{restatable}[$B_t$--convexity]{assumption}{ASBALLCONV}\label{as:defining_rel}
    A proper function $f:\R^d\to\R \cup \{+\infty\}$ is said to be \emph{$B_t$--convex} if there exists a function $c_t:\R^d\to\R_{\geq 0}$ such that for all $x \in {\rm dom} f$,    \begin{align}\label{eq:defining_rel}
        f(y) \geq f(u) + c_t(x) \inp{x-u}{y-u}
    \end{align}
    for any $u \in \BProxSub{t}{f}{x}$ and for all $y \in \R^d$.
\end{restatable}

When referencing \Cref{as:defining_rel} without specifying a particular value of $t$, we refer to it as \emph{ball-convexity}.
While inequality~\eqref{eq:defining_rel} always holds for convex $f$ (see \Cref{thm:3}), ball-convexity extends beyond traditional convexity, as demonstrated in the following example.
\begin{example}\label{ex:not_conn}
    Consider the function $f:\R\to\R$ defined via
    \begin{align*}
        f(x) =
        \begin{cases}
            -x-1 & x \leq -1 \\
            x+1 & -1 < x \leq 0 \\
            -x+1 & 0 < x \leq1 \\
            x-1 & x > 1.
        \end{cases}
    \end{align*}
    The function is clearly nonconvex.
    However, taking $t=1$ for simplicity, one can show that
    \begin{align*}
        \BProxSub{t}{f}{x} =
        \begin{cases}
            x+1 & x < -2 \\
            \{-1\} & -2 \leq x < 0 \\
            \{-1,1\} & x = 0 \\
            \{1\} & 0 < x \leq 2 \\
            x-1 & x > 2
        \end{cases}
    \end{align*}
    and Assumption \ref{as:defining_rel} holds with
    \begin{align*}
        c_t(x) =
        \begin{cases}
            1 & |x| > 2, \\
            0 & |x| \leq 2.
        \end{cases}
    \end{align*}
    Furthermore, the example illustrates that for functions satisfying \Cref{as:defining_rel}, the set of global minima may not be a singleton, and it need not be connected. Additionally, the mapping $x \mapsto \BProxSub{t}{f}{x}$ is not necessarily single-valued on the set $\{x: \BProxSub{t}{f}{x} \subseteq \cX_f\}$. 
\end{example}

The class of ball-convex functions is broader than that of convex functions, yet the broximal operator preserves all its desirable properties for this extended family of objectives. Notably, the convergence guarantees in \Cref{thm:conv-bppm-lin} and \Cref{rem:sublinear_conv} hold unchanged.
Interestingly, \newalg\ retains the linear convergence rate under even weaker assumptions (see \Cref{sec:bpm_lin_conv_weak}). However, this comes at the cost of the broximal operator losing some of its favorable properties.

The formal statements and proofs of these results can be found in \Cref{sec:ball-convex}, where we provide more information about the function class defined by \Cref{as:defining_rel}, analyze the properties of the broximal operator, and establish the convergence guarantees.

\newpage

\section{Numerical Experiments}\label{sec:experiments}
\begin{figure}[t]
    \centering
    \subfigure{
        \includegraphics[width=0.6\textwidth]{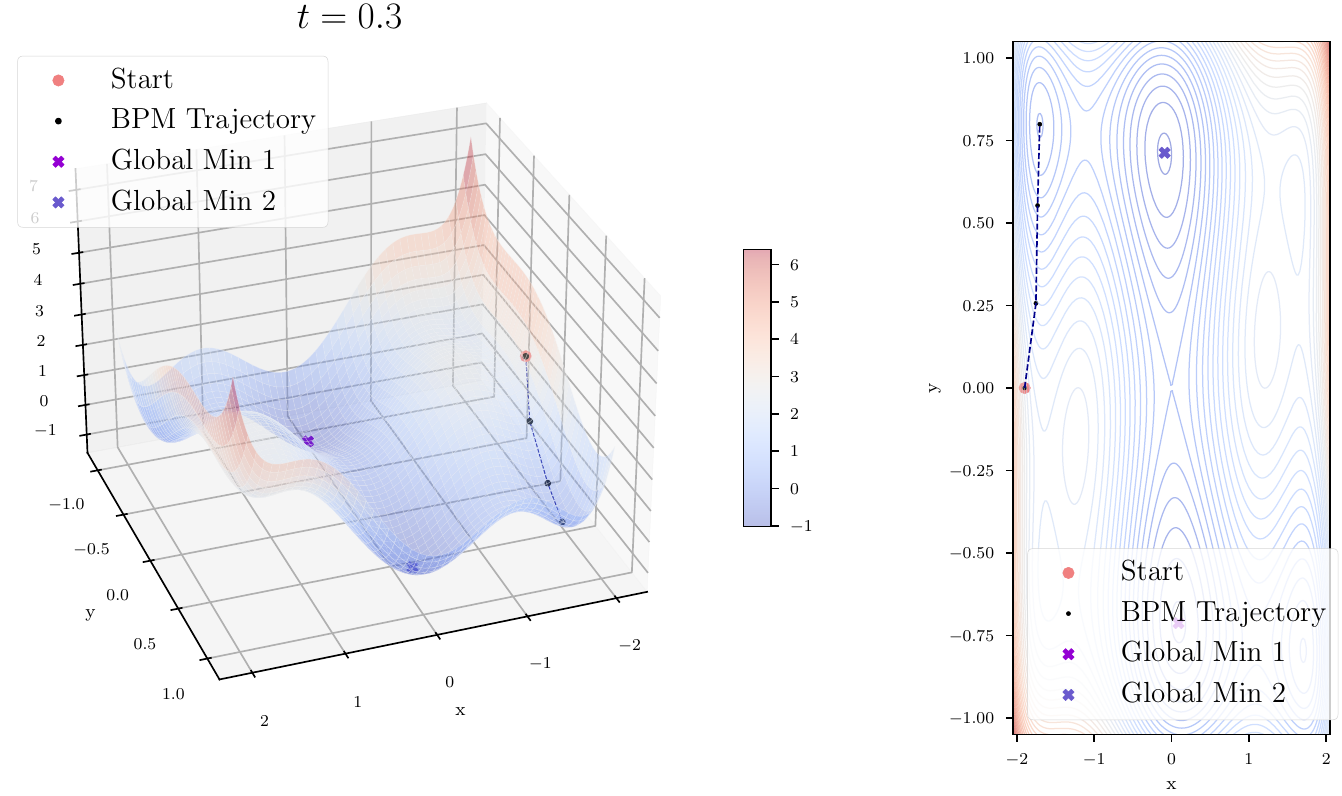}
        \label{fig:first1}
    }
    \\
    \subfigure{
        \includegraphics[width=0.6\textwidth]{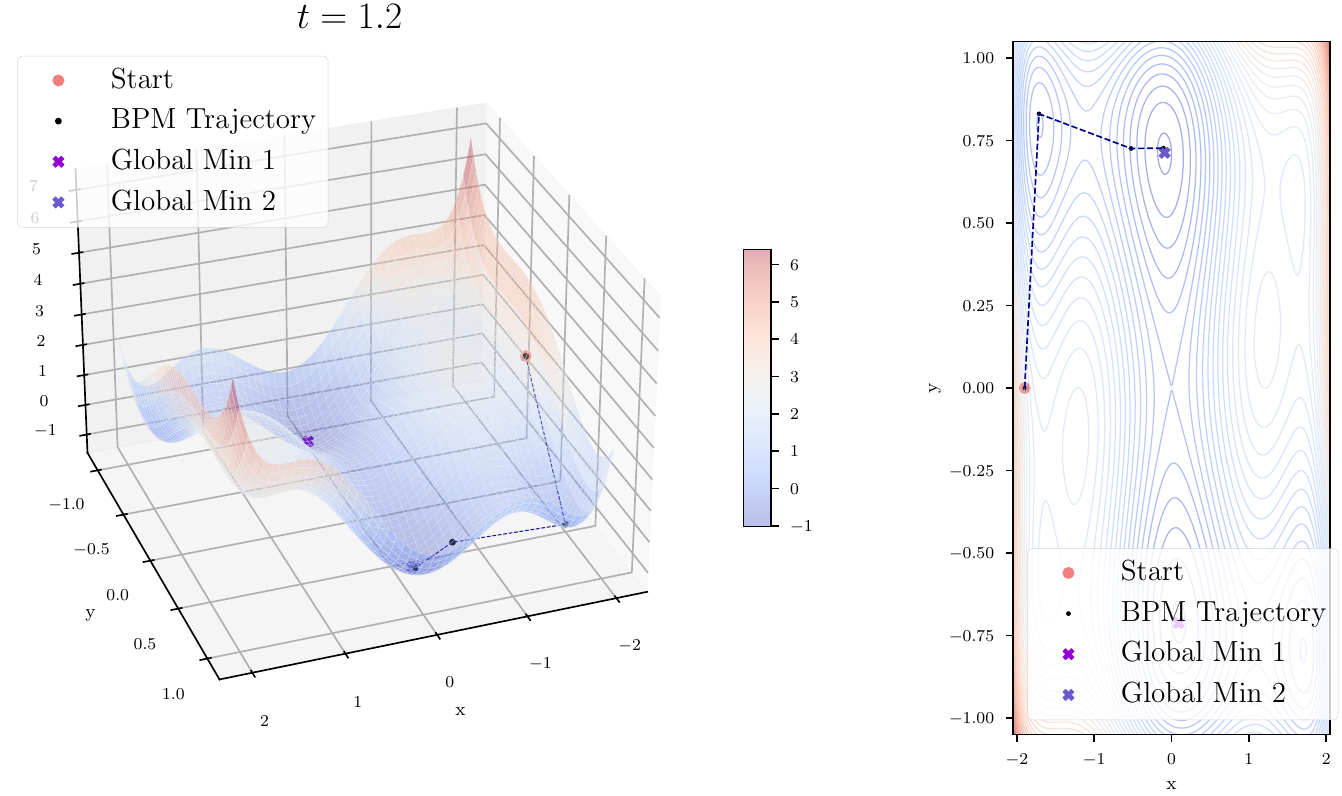}
        \label{fig:third3}
    }
    \caption{Visualization of \newalg\ applied to the Six-Hump Camel function, starting from the initial point $(-1.9, 0)$, with step sizes $t\in\{0.3, 1.2\}$.}
    \label{fig:bppm-3d}
\end{figure}
To validate the theoretical findings and further illustrate the mechanism of \newalg, we conduct numerical experiments on a simple optimization problem.
Specifically, we consider the minimization of the well-known Six-Hump Camel function \cite{molga2005test}, a classic benchmark for optimization algorithms, defined as
\begin{align*}
    \squeeze f(x, y) = \rbrac{4 - 2.1x^2 + \frac{x^4}{3}}x^2 + xy + \rbrac{-4 + 4y^2}y^2.
\end{align*}
This function is characterized by multiple local minima and two symmetric global minima located approximately at $(0.0898, -0.7126)$ and $(-0.0898, 0.7126)$, with global minimum value of $f_\star = -1.0316$.

As illustrated in \Cref{fig:bppm-3d}, the choice of step size $t$ plays a critical role in the algorithm's performance. 
A sufficiently large step size enables \newalg\ to bypass local minima and converge to a global minimum.
To further illustrate the impact of $t$ on the behavior of \newalg, we uniformly sample points within the ball $x^2 + y^2 \leq 16$ and evaluate the success rate of \newalg\ in reaching the global minimum for varying step sizes.
\Cref{fig:histo} highlights the relationship between $t$ and the algorithm's effectiveness: as  expected, larger values of $t$ improve \newalg's ability to converge to the global minimum.
\begin{figure}[t]
    \centering
    \subfigure{
        \includegraphics[width=0.6\textwidth]{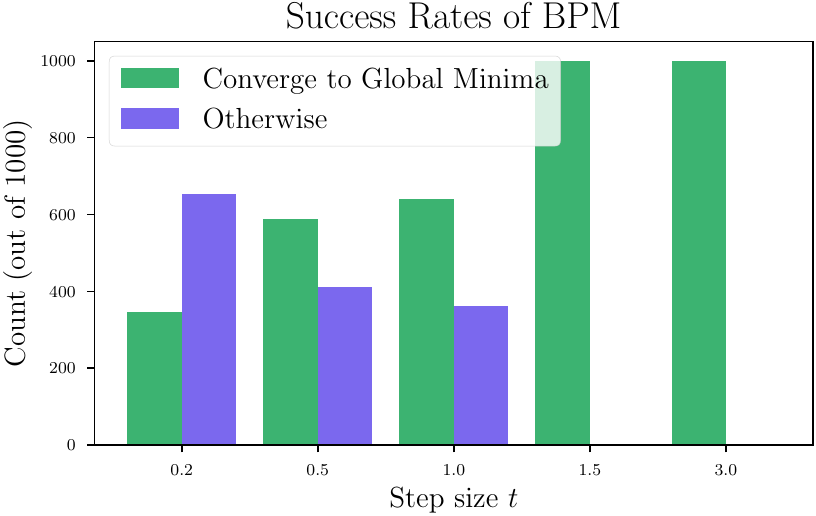}
        \label{fig:histo1}
    }
    \caption{Number of runs of \newalg\ (out of $1000$) that reached a global minimum for $t\in\cbrac{0.2, 0.5, 1, 1.5, 2}$. }
    \label{fig:histo}
\end{figure}

\clearpage

\section{Basic Facts}
\begin{fact}
    \label{fact:1:3pointidentity}
    For any $x, y, z \in \R^d$, we have 
    \begin{align}
        \label{eq:fact:1:3pointidentity}
        \inner{x - z}{y - z} = \frac{1}{2}\norm{x - z}^2 - \frac{1}{2}\norm{x - y}^2 + \frac{1}{2}\norm{z - y}^2.
    \end{align}
\end{fact}

\begin{fact}[Theorem 3.40 of \citet{beck2017first}]
    \label{fact:2:sum-rule-subdiff}
    Let $f_i: \R^d \mapsto \R \cup \cbrac{+\infty}$, $i \in [n]$, be proper convex functions such that $\cap_{i=1}^n \ri({\rm dom} f) \neq \emptyset$.
    Then
    \begin{align*}
        \partial \rbrac{\sum_{i=1}^{n} f_i}(x) = \sum_{i=1}^{n}\partial f_i(x)
    \end{align*}
    for any $x \in \R^d$.
\end{fact}

\begin{fact}
    \label{lemma:four-point-identity}
    Suppose that a proper, closed and strictly convex function $\phi: \R^d \mapsto \R \cup \brac{+\infty}$ is finite at $a, b, c, d \in \R^d$ and is differentiable at $a, b$.
    Then
    \begin{align*}
        \inner{\nabla \phi(a) - \nabla \phi(b)}{c - d} = \breg{\phi}{c}{b} + \breg{\phi}{d}{a} - \breg{\phi}{c}{a} - \breg{\phi}{d}{b}.
    \end{align*}
\end{fact}

\begin{lemma}[Subdifferential of indicator function]\label{lemma:subfiff_id}
    The subdifferential of an indicator function of a set $\cX\neq \emptyset$ at a point $y \in \cX$ is
    \begin{align*}
        \partial \delta_{\cX}(y) = \cN_{\cX}(y) \eqdef \brac{g\in\R^d : \inp{g}{z - y} \leq 0 \,\forall z \in \cX},
    \end{align*}
    where $\cN_{\cX}(y)$ is the normal cone of $\cX$ at $y$.
\end{lemma}
\begin{proof}
    For $y \notin \cX$, $\partial \delta_{\cX}(y) = \emptyset$. When $y \in \cX$, by definition of subdifferential, $g \in \partial \delta_{\cX}(y)$ if and only if
    \begin{align*}
        \delta_{\cX}(z) \geq \delta_{\cX}(y) + \inp{g}{z - y} \qquad\forall z \in \cX,
    \end{align*}
    which is equivalent to $\inp{g}{z - y} \leq 0$ for all $z \in \cX$.
\end{proof}

\begin{lemma}[Normal cone of the indicator function of a ball]\label{lemma:subdiff_id_ball}
    The normal cone of a ball $B_t(x)$ is
    \begin{align*}
        \cN_{B_t(x)}(y) = 
        \begin{cases}
            \R_{\geq0}(y-x) & \norm{x-y} = t, \\
            \{0\} & \norm{x-y} < t, \\
            \emptyset & \norm{x-y} > t.
        \end{cases}
    \end{align*}
\end{lemma}
\begin{proof}
    For $y \notin B_t(x)$, $\cN_{B_t(x)}(y) = \emptyset$.
    Now, let $y \in B_t(x)$. Then
    \begin{align*}
        g \in \partial \delta_{B_t(x)}(y) \qquad&\overset{\eqref{lemma:subfiff_id}}{\iff}\qquad
        \inp{g}{z - y} \leq 0 \quad\forall z \in B_t(x) \\\qquad&\iff\qquad
        \inp{g}{z} \leq \inp{g}{y} \quad\forall z \in B_t(x) \\\qquad&\iff\qquad
        \sup_{z: \norm{z-x} \leq t} \inp{g}{z} \leq \inp{g}{y} \\\qquad&\iff\qquad
        \sup_{z: \norm{\frac{z - x}{t}} \leq 1} \inp{g}{\frac{z - x}{t}} \leq \inp{g}{\frac{y - x}{t}} \\\qquad&\iff\qquad
        \sup_{z: \norm{w} \leq 1} \inp{g}{w} \leq \inp{g}{\frac{y - x}{t}} \\\qquad&\iff\qquad
        \norm{g} \leq \frac{\inp{g}{y - x}}{t}.
    \end{align*}
    On the other hand, Cauchy-Schwarz inequality gives
    \begin{align}\label{eq:ajvnbaow}
        t \norm{g} \leq \inp{g}{y - x} \leq \norm{g} \norm{y-x},
    \end{align}
    meaning that
    \begin{align*}
        0 \leq \norm{g} \parens{\norm{y-x} - t}.
    \end{align*}
    Since $y \in B_t(x)$, $\norm{y-x} - t \leq 0$, and hence we must have $\norm{y-x}=t$ or $\norm{g}=0$. In the former case, when $y$ lies on the boundary of $B_t(x)$, \eqref{eq:ajvnbaow} says that the Cauchy-Schwarz inequality is an equality, implying that $g$ and $y-x$ are linearly dependent.
    Otherwise, when $\norm{y-x}<t$, we get $g=0$, which finishes the proof.
\end{proof}

\newpage

\section{Convergence Theory: Convex Case}\label{sec:ap_convex}

\begin{algorithm}[t]
    \caption{Ball-Proximal Point Method (\newalg)}
    \begin{algorithmic}[1]\label{alg:bpm}
    \STATE \textbf{Input:} radii $t_k > 0$ for $k \geq 0$, starting point $x_0 \in {\rm dom} f$
    \FOR{$k = 0,1,2,\dots$}
    \STATE $x_{k+1} \in \BProxSub{t_k}{f}{x_k}$
    \ENDFOR
    \end{algorithmic}
\end{algorithm}

Before proving the convergence guarantee, we first introduce some useful preliminary results. The first result establishes that the proximity operator of a proper, closed and convex function is well-defined.

\begin{theorem}[First brox theorem]\label{thm:1stbroxthm}
    Let $f: \R^d \mapsto \R \cup \{+\infty\}$ is proper, closed and convex and choose $x \in {\rm dom} f$. Then
    \begin{enumerate}[label=(\roman*)]
        \item $\BProxSub{t}{f}{x} \neq \emptyset$. Moreover, if $B_t(x) \cap \cX_f \neq \emptyset$, then $\BProxSub{t}{f}{x}$ is a nonempty subset of $\cX_f$. \label{pt:nonempty}
        \item If $B_t(x) \cap \cX_f = \emptyset$, then $\BProxSub{t}{f}{x}$ is a singleton lying on the boundary of $B_t(x)$. \label{pt:singleton_bdry}
    \end{enumerate}
\end{theorem}
\begin{proof}
    \begin{enumerate}[label=(\roman*)]
        \item The broximal operator is minimizing a proper, closed and convex function over a closed set $B_t(x)$.
        Hence, by the Weierstrass Theorem, $f$ is lower bounded on $B_t(x)$ and attains its minimal value, proving that $\BProxSub{t}{f}{x} \neq \emptyset$.
        It follows that if $B_t(x) \cap \cX_f \neq \emptyset$, then $\BProxSub{t}{f}{x} \subseteq \cX_f$ is nonempty.
        
        \item Let $z_\star \in \BProxSub{t}{f}{x}$. 
        Then $z_\star$ is a minimizer of the function
        \begin{align*}
            A^t_x(z) \eqdef f(z) + \delta_{B_t(x)}(z),
        \end{align*}
        where $$\delta_{\cX}(z) \eqdef \begin{cases} 0 & z\in \cX \\
        +\infty & z\notin \cX\end{cases}$$ is the indicator function of the set $\cX \subseteq \R^d$.
        Suppose that $z_\star \in \interior B_t(x)$ and consider the line segment connecting~$z_\star$ and any global minimizer $x_\star \notin B_t(x)$ of $f$. 
        Obviously it intersects with $\bdry B_t(x)$ at some point $z_\lambda = \lambda z_\star + \rbrac{1 - \lambda}x_\star$, where $\lambda \in (0, 1)$. 
        Using convexity of $f$, we know that 
        \begin{equation}\label{eq:h08yfd08yhfd}
           f(z_\lambda) \leq \rbrac{1 - \lambda}f(x_\star) + \lambda f(z_\star) < f(z_\star),
        \end{equation}
        where the last inequality holds because $f(x_\star)<f(z_\star)$, which is true because $x_\star \in \cX_f$, $z_\star \in B_t(x)$ and $\cX_f\cap B_t(x) = \emptyset$.
        \Cref{eq:h08yfd08yhfd} clearly contradicts the assumption that $z_\star$ is a minimizer of $A^t_x(z)$, as $A^t_x(z_\lambda) < A^t_x(z_\star)$.
        Thus, we must have $z_\star \in \bdry B_t(x)$.
        
        Now, assume that $\BProxSub{t}{f}{x}$ is not a singleton and there exist $z_{\star, 1}, z_{\star, 2} \in \BProxSub{t}{f}{x}$. Then, by the argument above, $z_{\star, 1}, z_{\star, 2} \in \bdry B_t(x)$, and due to the convexity of $f$, all points on the line segment connecting $z_{\star, 1}$ and $z_{\star, 2}$ are also minimizers of $A^t_x(z)$.
        However, this contradicts the fact that no minimizers of $A^t_x(z)$ lie within $\interior B_t(x)$. 
        Therefore, $\BProxSub{t}{f}{x}$ must be a singleton.
    \end{enumerate}
\end{proof}

The second part of \Cref{thm:1stbroxthm} demonstrates that as long as $\BProxSub{t}{f}{x} \not\subseteq \cX_f$, the broximal operator is uniquely defined.
Furthermore, it shows that \newalg\ (\Cref{alg:bpm}) progresses with steps of length $t_k$, moving from the center to the surface of the ball until it reaches a global minimum.

\begin{theorem}[Second brox theorem]
    \label{thm:3}
    Let $f: \R^d \mapsto \R \cup \{+\infty\}$ be proper, closed and convex. Choose $x \in {\rm dom} f$ and $u \in \BProxSub{t}{f}{x}$ for some $t > 0$.
    Then, there exists $c_t(x) \geq 0$ such that
    \begin{enumerate}[label=(\roman*)]
        \item $c_t(x)(x - u) \in \partial f(u)$,
        \item $f(y) - f(u) \geq c_t(x)\inner{x - u}{y - u}$ for all $y \in \R^d$.
    \end{enumerate}
\end{theorem}
\begin{proof}
    Let us consider two cases:
    \paragraph{Case 1: $B_t(x) \cap \cX_f \neq \emptyset$:}
    Since $u \in \BProxSub{t}{f}{x}$, according to \Cref{thm:1stbroxthm}~\ref{pt:nonempty}, $u$ must be a global minimizer of $f$.
    In this case, it is evident that $0 \in \partial f(u)$, and statement $(i)$ holds with $c_t(x) = 0$.
    Furthermore, since $u$ is a global minimizer, we have $f(y) \geq f(u)$ for all $y \in \R^d$, so statement $(ii)$ also holds.

    \paragraph{Case 2: $B_t(x) \cap \cX_f = \emptyset$:}
    In this case, $u \in \BProxSub{t}{f}{x}$ indicates that $u$ is a minimizer of the function
    \begin{align*}
        A^t_x(z) \eqdef f(z) + \delta_{B_t(x)}(z).
    \end{align*}
    Since both $f$ and $\delta_{B_t(x)}$ are convex, $A^t_x$ is convex as well.
    Now, let us demonstrate that $\ri(\cB(x, t)) \cap \ri({\rm dom}(f)) \neq \emptyset$. This clearly holds when $x \in \ri({\rm dom}(f))$. If $x \not\in \ri({\rm dom}(f))$, then it must lie in its closure, since $\overline{\ri({\rm dom}(f))} = \overline{{\rm dom}(f)} \ni x$. As a result, there exists a sequence $\{z_k\}_{k\geq0} \subset \ri({\rm dom}(f))$ such that $z_k \to x$ as $k\to\infty$. Now, since $x \in \cB(x, t)$, there exists~$K\geq 0$ such that $z_k \in \ri(\cB(x, t))$ for all $k \geq K$, and hence we can conclude that $\ri(\cB(x, t)) \cap \ri({\rm dom}(f)) \neq \emptyset$.
    Therefore, by Fermat's optimality condition and \Cref{fact:2:sum-rule-subdiff}, we have 
    \begin{align*}
        0 \in \partial A^t_x(u) = \partial \rbrac{f + \delta_{B_t(x)}}(u) = \partial f(u) + \partial \delta_{B_t(x)}(u).
    \end{align*} 
    Using \Cref{lemma:subfiff_id} and the observation that  in view of \Cref{thm:1stbroxthm}~\ref{pt:singleton_bdry}, we have $\norm{x - u} = t$, the above identity can be further rewritten in terms of the normal cone, 
    \begin{align*}
        0 \in \partial f(u) + \cN_{B_t(x)}(u)
        \overset{\eqref{lemma:subdiff_id_ball}}{=} \partial f(u) + \R_{\geq0}(u-x),
    \end{align*}
    where $\R_{\geq0}(z)\eqdef \{\lambda z \;:\; \lambda\geq 0\}.$      
    Hence, there exists some $c_t(x) \geq 0$ such that $c_t(x)(x - u) \in \partial f(u)$.
    Lastly, using the definition of a subgradient, we obtain 
    \begin{align*}
        f(y) - f(u) \geq c_t(x)\inner{x - u}{y - u}
    \end{align*}
    for all $y \in \R^d$.
\end{proof}

\Cref{thm:3} is central to demonstrating the convergence of the \newalg\ algorithm.
Building on these results, additional properties can be derived. 
However, we postpone their discussion to \Cref{sec:ball-convex}, as they apply to a more general class of functions.

Instead, we proceed directly to the convergence result.
\THMBPPMCONVLIN*

\begin{proof}
    \begin{enumerate}[label=(\roman*)]
        \item This follows from \Cref{thm:1stbroxthm}~\ref{pt:nonempty}.
        
        \item The first part of the statement is a direct consequence of \Cref{thm:1stbroxthm}~\ref{pt:singleton_bdry}. 
        To prove the second claim, note that \Cref{thm:3} with $x = x_k$ and $y = x_\star \in \cX_f$ gives
        \begin{eqnarray*}
            f(x_{k+1}) - f_{\star} \leq c_{t_k}(x_k) \inp{x_k-x_{k+1}}{x_{k+1} - x_{\star}},
        \end{eqnarray*}
        where $c_{t_k}(x_k)\geq 0$. We argue that $c_{t_k}(x_k)> 0$. Indeed, if $c_{t_k}(x_k)$ was equal to $0$, then $f(x_{k+1}) - f_{\star} \leq 0$, which means that $x_{k+1}$ is optimal, contradicting the assumption that $\cX_f\cap B_{t_k}(x_k)=\emptyset$. Hence, dividing both sides of the inequality by~$c_{t_k}(x_k)$, we get
        \begin{eqnarray*}
            0 &\leq& \frac{f(x_{k+1}) - f_{\star}}{c_{t_k}(x_k)} \leq \inp{x_k-x_{k+1}}{x_{k+1} - x_{\star}} \\
            &\overset{\eqref{fact:1:3pointidentity}}{=}& \frac{1}{2} \parens{\norm{x_k-x_{\star}}^2 - \norm{x_{k+1} - x_{\star}}^2 - \norm{x_k-x_{k+1}}^2} \\
            &\overset{\eqref{thm:conv-bppm-lin}}{=}& \frac{1}{2} \parens{\norm{x_k-x_{\star}}^2 - \norm{x_{k+1} - x_{\star}}^2 - t_k^2},
        \end{eqnarray*}
        and hence
        \begin{eqnarray*}
            \norm{x_{k+1} - x_{\star}}^2 \leq \norm{x_k-x_{\star}}^2 - t_k^2,
        \end{eqnarray*}
        proving the first inequality.
        The above holds for any $x_{\star}\in\cX_f$, and hence it holds for the optimal point closest to $x_k$, too. Therefore, the last inequality can be obtained using the fact that $\textnormal{dist}^2(x_{k+1}, \cX_f) \leq \norm{x_{k+1} - x_{\star}}^2$.
        
        \item This follows directly from parts \ref{pt:opt} and \ref{pt:dist_decr_bpm}.
        
        \item Let us consider some iteration $k$ such that $x_{k+1} \not\in \cX_f$ (otherwise, the problem is solved in $1$ step).
        Using \Cref{thm:3} with $y = x = x_k$, we have 
        \begin{align}\label{eq:rauntbsvya}
            f(x_{k+1}) - f_{\star} \leq f(x_k) - f_{\star} - c_{t_k}(x_k) \norm{x_{k+1}-x_k}^2
            \overset{\eqref{thm:conv-bppm-lin}}{=} f(x_k) - f_{\star} - c_{t_k}(x_k) t_k^2.
        \end{align}
        Next, \Cref{thm:3} with $x = x_k$ and $y = x_\star \in \cX_f$ and Cauchy-Schwarz inequality give
        \begin{eqnarray*}
            f(x_{k+1}) - f_{\star} &\leq& c_{t_k}(x_k) \inp{x_k-x_{k+1}}{x_{k+1} - x_{\star}} \\
            &\leq& c_{t_k}(x_k) \norm{x_k-x_{k+1}} \norm{x_{k+1} - x_{\star}} \\
            &\overset{\eqref{thm:conv-bppm-lin}}{=}& c_{t_k}(x_k) t_k \norm{x_{k+1} - x_{\star}}.
        \end{eqnarray*}
        Since $x_{k+1} \not\in \cX_f$, we can divide both sides by $\norm{x_{k+1} - x_{\star}}$ and multiply by $t_k$, obtaining
        \begin{eqnarray}\label{eq:oaiasbfrf}
            \parens{f(x_{k+1}) - f_{\star}} \frac{t_k}{\norm{x_{k+1} - x_{\star}}} \leq c_{t_k}(x_k) t_k^2.
        \end{eqnarray}
        Applying the bound \eqref{eq:oaiasbfrf} in \eqref{eq:rauntbsvya} gives
        \begin{eqnarray}\label{eq:oiqwnvroan}
            f(x_{k+1}) - f_{\star} &\leq& f(x_k) - f_{\star} - \parens{f(x_{k+1}) - f_{\star}} \frac{t_k}{\norm{x_{k+1} - x_{\star}}}.
        \end{eqnarray}
        Rearranging the terms, we obtain the result.

        \item The claim obviously holds when $\norm{\nabla f(x_{k+1})} = 0$, so suppose that $\norm{\nabla f(x_{k+1})} \neq 0$. In the differentiable case, \Cref{lemma:brox_prox} states that the update rule of \newalg\ is
        \begin{align}\label{eq:cqecqdqwd}
            x_{k+1} = x_k - \tfrac{t_k}{\norm{\nabla f(x_{k+1})}} \nabla f(x_{k+1}).
        \end{align}
        Now, convexity and Cauchy-Schwarz inequality give
        \begin{align*}
            f(x_{k+1}) - f(x_k) \geq \inp{\nabla f(x_k)}{x_{k+1} - x_k}
            \geq - \norm{\nabla f(x_k)} \norm{x_{k+1} - x_k}.
        \end{align*}
        Rearranging the terms and using convexity again, we obtain
        \begin{eqnarray*}
            \norm{\nabla f(x_k)} \norm{x_{k+1} - x_k} &\geq& f(x_k) - f(x_{k+1})
            \geq \inp{\nabla f(x_{k+1})}{x_k - x_{k+1}} \\
            &\overset{\eqref{eq:cqecqdqwd}}{=}& \inp{\nabla f(x_{k+1})}{\frac{t_k}{\norm{\nabla f(x_{k+1})}} \nabla f(x_{k+1})}
            = t_k \norm{\nabla f(x_{k+1})}.
        \end{eqnarray*}
        The result follows from the fact that $\norm{x_{k+1} - x_k} = t_k$ (see part \ref{pt:dist_decr_bpm}).

        To prove the convergence result, we again start with convexity, obtaining
        \begin{eqnarray*}
            f(x_{k+1}) &\leq& f(x_k) - \inp{\nabla f(x_{k+1})}{x_k - x_{k+1}} \\
            &\overset{\eqref{eq:cqecqdqwd}}{=}& f(x_k) - \inp{\nabla f(x_{k+1})}{\frac{t_k}{\norm{\nabla f(x_{k+1})}} \nabla f(x_{k+1})} \\
            &=& f(x_k) - t_k \norm{\nabla f(x_{k+1})}.
        \end{eqnarray*}
        Rearranging the terms and summing over the first $K$ iterations gives
        \begin{eqnarray*}
            \sum_{k=0}^{K-1} \parens{t_k \norm{\nabla f(x_{k+1})}}
            \leq \sum_{k=0}^{K-1} \parens{f(x_k) - f(x_{k+1})} 
            \leq f(x_0) - f_{\star}.
        \end{eqnarray*}
        It remains to divide both sides of the inequality by the sum of radii $t_k$ to obtain
        \begin{eqnarray*}
            \sum_{k=0}^{K-1} \parens{\frac{t_k}{\sum_{k=0}^{K-1} t_k} \norm{\nabla f(x_{k+1})}}
            \leq \frac{f(x_0) - f_{\star}}{\sum_{k=0}^{K-1} t_k}.
        \end{eqnarray*}
    \end{enumerate}
\end{proof}

The next corollary is an immediate consequence of \Cref{thm:conv-bppm-lin}.
\CORBPPMCONVLIN*
\begin{proof}
    The result follows from repeatedly applying the inequality in \Cref{thm:conv-bppm-lin}~\ref{pt:1step_bpm_lin} and observing that the sequence~$\{d_k\}_{k \geq 0}$ is non-increasing, as established in \Cref{thm:conv-bppm-lin}~\ref{pt:dist_decr_bpm}.
\end{proof}

As promised, we also provide a $\cO(\nicefrac{1}{K})$ convergence guarantee.
\begin{theorem}\label{thm:conv-bppm-convex}
    Assume $f: \R^d \mapsto \R \cup \{+\infty\}$ is proper, closed and convex, and choose $x_0\in {\rm dom} f$. Then, for any $K\geq 1$, the iterates of \newalg\ run with $t_k \equiv t > 0$ satisfy
    \begin{align*}
        \squeeze f(x_K) - f_\star \leq \frac{2d_0}{2d_0 + t} \cdot\frac{d_0^2}{2Kt^2} \parens{f(x_0) - f_{\star}}.
    \end{align*}
\end{theorem}

\begin{proof}
    Let us consider some iteration $k$ such that $x_{k+1} \not\in \cX_f$ (otherwise, the problem is solved in $1$ step).
    Invoking \Cref{thm:3} with $y = x_\star \in \cX_f$, we have 
    \begin{align*}
        f_\star - f(x_{k+1}) \geq  c_t(x_k)\inner{x_k - x_{k+1}}{x_\star - x_{k+1}}.
    \end{align*}
    Rearranging terms and using \Cref{fact:1:3pointidentity}, we have 
    \begin{eqnarray}\label{eq:aguiorhnua}
        f(x_{k+1}) - f_\star &\leq& \frac{c_t(x_k)}{2}\rbrac{\norm{x_k - x_\star}^2 - \norm{x_k - x_{k+1}}^2 - \norm{x_{k+1} - x_\star}^2} \nonumber \\
        &\overset{\eqref{thm:conv-bppm-lin}}{=}& \frac{c_t(x_k)}{2}\rbrac{\norm{x_k - x_\star}^2 - \norm{x_{k+1} - x_\star}^2 - t^2}.
    \end{eqnarray}
    Since $x_{k+1} \neq x_k$, \Cref{rem:c_upper_rel} gives
    \begin{align*}
        c_t(x_k) \leq \frac{f(x_k) - f(x_{k+1})}{\norm{x_k - x_{k+1}}^2} \leq \frac{f(x_0) - f_\star}{t^2},
    \end{align*}
    where the second inequality follows from \Cref{thm:1stbroxthm}~\ref{pt:singleton_bdry} and the fact that $f(x_k) \geq f(x_{k+1}) \geq f_\star$ for any $k\geq0$.
    As a result, we have 
    \begin{align*}
        f(x_{k+1}) - f_\star + \frac{c_t(x_k)t^2}{2}
        \leq \frac{f(x_0) - f_\star}{2t^2}\rbrac{\norm{x_k - x_\star}^2 - \norm{x_{k+1} - x_\star}^2}.
    \end{align*}
    Averaging both sides over $k \in \cbrac{0, 1, \hdots, K-1}$, we obtain 
    \begin{eqnarray}
        \label{eq:temp:proof:q0}
        \frac{1}{K}\sum_{k=0}^{K-1}\rbrac{f(x_{k+1}) - f_\star} + \frac{t^2}{2K}\sum_{k=0}^{K-1}c_t(x_k)
        &\leq& \frac{1}{K}\sum_{k=0}^{K-1} \frac{f(x_0) - f_\star}{2t^2}\rbrac{\norm{x_k - x_\star}^2 - \norm{x_{k+1} - x_\star}^2} \nonumber \\
        &\leq& \frac{f(x_0) - f_\star}{2Kt^2}\norm{x_0 - x_\star}^2.
    \end{eqnarray}
    Now, let us bound the terms on the LHS of the above inequality. Since the sequence of function values is decreasing, the average function suboptimality can be bounded by
    \begin{align}
        \label{eq:temp:proof:q1}
        \frac{1}{K}\sum_{k=0}^{K-1}\rbrac{f(x_{k+1}) - f_\star} \geq f(x_K) - f_\star,
    \end{align}
    and using Remarks \ref{rem:c_lower} and \ref{rem:cxx_decr} gives 
    \begin{align}
        \label{eq:temp:proof:q2}
        \frac{t^2}{2K}\sum_{k=0}^{K-1}c_t(x_k) \geq \frac{t^2}{2}c_t(x_{K-1}) \geq \frac{t^2}{2}\cdot\frac{f(x_K) - f_\star}{\norm{x_{k-1} - x_K}\norm{x_K - x_\star}} = \frac{t\rbrac{f(x_K) - f_\star}}{2\norm{x_K - x_\star}} \geq \frac{t\rbrac{f(x_K) - f_\star}}{2\norm{x_0 - x_\star}}.
    \end{align}
    Combining \eqref{eq:temp:proof:q0}, \eqref{eq:temp:proof:q1} and  \eqref{eq:temp:proof:q2} and rearranging, we finally get 
    \begin{align*}
        f(x_K) - f_\star \leq \rbrac{1 + \frac{t}{2\norm{x_0 - x_\star}}}^{-1} \frac{f(x_0) - f_\star}{2Kt^2}\norm{x_0 - x_\star}^2,
    \end{align*}
    which finishes the proof.
\end{proof}

\newpage

\section{Convergence Theory: Beyond Convexity}
\label{sec:ball-convex}

We now turn to the ball-convex setting. For ease of reference, let us restate the main assumption.
\ASBALLCONV*

By \Cref{thm:3}, \Cref{as:defining_rel} is satisfied when the objective function is convex. Consequently, \textbf{all the results in this section remain valid if \Cref{as:defining_rel} is replaced by convexity}.

To begin, we look into the properties of the broximal operator and the function class defined by \Cref{as:defining_rel}.
\begin{theorem}\label{thm:c_global_min}
    Let Assumption \ref{as:defining_rel} hold. Choose $x \in {\rm dom} f$ and $u \in \BProxSub{t}{f}{x}$. Then
    \begin{enumerate}[label=(\roman*)]
        \item $c_t(x) = 0$ if and only if $u \in \cX_f$. \label{pt:c0_min}
        \item $x \in \BProxSub{t}{f}{x}$ if and only if $x\in\cX_f$. \label{pt:fixed_pt}
        \item $f(u) = f(x)$ if and only if $x \in \cX_f$. \label{pt:diff_val}
    \end{enumerate}
\end{theorem}
The theorem above establishes that under \Cref{as:defining_rel}, any fixed point of the mapping $\BProxSub{t}{f}{\cdot}$ is a global minimizer. It simultaneously captures the ``nonflatness'' property of ball-convex functions: as long as $x \not\in \BProxSub{t}{f}{x}$ (and the iterates of \newalg\ keep moving), $x$ is not a global minimum. Hence, the assumption essentially says that the radius~$t$ is large enough, so that $f$ is not constant on $B_t(x)$.

\begin{proof}[Proof of \Cref{thm:c_global_min}]
    \begin{enumerate}[label=(\roman*)]
        \item $c_t(x) = 0$ implies that $f(u) \leq f(y)$ for all $y \in \R^d$, and hence $u$ is a global minimizer of~$f$.
        Conversely, if $u$ is a global minimizer of $f$, then $f(y) \geq f(u)$ for all $y \in \R^d$, and hence inequality~\eqref{eq:defining_rel} holds with $c_t(x)=0$.
        \item If $x \in \BProxSub{t}{f}{x}$, then condition \eqref{eq:defining_rel} gives $f(y) \geq f(x)$ for all $y \in \R^d$, and hence $x$ is a global minimizer. The converse holds by the definition of broximal operator.
        \item If $f(u) = f(x)$, then inequality~\eqref{eq:defining_rel} with $y=x$ gives
        \begin{align*}
            f(u) = f(x) \geq f(u) + c_t(x) \norm{x-u}^2,
        \end{align*}
        so $c_t(x) \norm{x-u}^2 = 0$, and either $u=x$ or $c_t(x)=0$. In the former case, part \ref{pt:fixed_pt} implies that $x\in\cX_f$. In the latter case, from part \ref{pt:c0_min} we know that $u$ is a global minimizer of $f$. Since, by assumption, $f(u) = f(x)$, $x$ is also a global minimizer of $f$.
    \end{enumerate}
\end{proof}

The next proposition states that, similar to the convex case, under \Cref{as:defining_rel}, the iterates of \newalg\ are uniquely determined until they reach the optimal solution set.

\begin{proposition}\label{prop:single_val}
    Let \Cref{as:defining_rel} hold. If $B_t(x) \cap \cX_f = \emptyset$, then the mapping $x \mapsto \BProxSub{t}{f}{x}$ is single-valued.
\end{proposition}
\begin{proof}
    Fix $x\in\R^d$ and let $x_1, x_2 \in \BProxSub{t}{f}{x}$. Then, the defining property \eqref{eq:defining_rel} gives
    \begin{align*}
        f(y) \geq f(x_1) + c_t(x) \inp{x-x_1}{y-x_1}
    \end{align*}
    and
    \begin{align*}
        f(y) \geq f(x_2) + c_t(x) \inp{x-x_2}{y-x_2}
    \end{align*}
    for any $y\in\R^d$. Taking $y=x_2$ in the first inequality and $y=x_2$ in the second inequality and adding the two, we get
    \begin{align*}
        0 \geq c_t(x) \inp{x-x_1}{x_2-x_1} + c_t(x) \inp{x-x_2}{x_1-x_2}
        = c_t(x) \norm{x_1-x_2}^2.
    \end{align*}
    Now, by assumption, $B_t(x) \cap \cX_f = \emptyset$, so $\BProxSub{t}{f}{x}\not\subseteq \cX_f$. Hence, Theorem \ref{thm:c_global_min} gives $c_t(x)>0$, implying that $x_1=x_2$.
\end{proof}

\begin{lemma}\label{lemma:neg_inp}
    Let \Cref{as:defining_rel} hold and let $x \in {\rm dom} f$ be such that $\textnormal{dist}(x, \cX_f) > t$. Then
    \begin{align*}
        0 < \inp{x - u}{u - x_{\star}}
    \end{align*}
    for any $x_{\star}\in\cX_f$, where $u \in \BProxSub{t}{f}{x}$.
\end{lemma}
\begin{proof}
    First, $\textnormal{dist}(x, \cX_f) > t$ means that $u \not\in \cX_f$, and hence $c_t(x) > 0$ (see part \ref{pt:c0_min} of Theorem \ref{thm:c_global_min}). Now, letting $y=x_{\star}$ in~\eqref{eq:defining_rel}, we obtain
    \begin{align*}
        f_{\star} \geq f(u) + c_t(x) \inp{x-u}{x_{\star}-u},
    \end{align*}
    and consequently
    \begin{align*}
        0 < f(u) - f_{\star} \leq c_t(x) \inp{x-u}{u - x_{\star}}.
    \end{align*}
    Dividing by $c_t(x)>0$ proves the claim.
\end{proof}

\begin{proposition}\label{prop:conv_comb}
    Let Assumption \ref{as:defining_rel} hold and let $z_1, \ldots, z_n \in \cX_f$. Define $z(\lambda) \eqdef \lambda_1 z_1 + \ldots + \lambda_n z_n$, where $\lambda \eqdef \{(\lambda_1, \ldots, \lambda_n) \in[0,1]: \sum_{i=1}^n \lambda_i = 1\}$. Then $u \in \cX_f$ for any $u \in \BProxSub{t}{f}{z(\lambda)}$.
\end{proposition}

\begin{remark}
    As shown in \Cref{ex:not_conn}, the solution set need not be connected for Assumption \ref{as:defining_rel} to hold. However, as proven in Proposition \ref{prop:conv_comb}, any point in the convex hull of $\cX_f$ must be at a distance of at most $t$ from the solution set.
\end{remark}

\begin{proof}[Proof of \Cref{prop:conv_comb}]
    Assume that $c_t(z(\lambda))>0$ and let $u \in \BProxSub{t}{f}{z(\lambda)}$. Then, by Theorem \ref{thm:c_global_min}, we have $\textnormal{dist}(z(\lambda), \cX_f) > t$, and hence by Lemma \ref{lemma:neg_inp}
    \begin{align*}
        0 > \inp{z(\lambda) - u}{z_i - u} \qquad\forall i\in[n].
    \end{align*}
    Multiplying the $i$th inequality by $\lambda_i$ and adding them up, we get
    \begin{align*}
        0 > \inp{z(\lambda) - u}{\lambda_1 z_1 + \ldots + \lambda_n z_n - u}
        = \norm{z(\lambda)-u}^2,
    \end{align*}
    which is a contradiction. Thus, $c_t(z(\lambda))=0$, and Theorem \ref{thm:c_global_min} shows that $u \in \cX_f$.
\end{proof}

The next two propositions say that under Assumption \ref{as:defining_rel}, the radius $t$ must be large enough for the iterates to be able to move to a point with a strictly smaller function value.

\begin{proposition}\label{prop:local_escape}
    Let Assumption \ref{as:defining_rel} hold and let $x \in {\rm dom} f \backslash \cX_f$. Then, there exists $\bar{x} \in B_t(x)$ such that $f(\bar{x}) < f(x)$.
\end{proposition}
\begin{proof}
    If there existed $x \not\in \cX_f$ such that $f(x) \leq f(\bar{x})$ for all $\bar{x} \in B_t(x)$, by definition of broximal operator, we would have $x\in \BProxSub{t}{f}{x}$. But then Theorem \ref{thm:c_global_min} would imply that $x \in \cX_f$, which is a contradiction.
\end{proof}

\begin{proposition}\label{prop:f_dist_decr}
    Let Assumption \ref{as:defining_rel} hold and let $x \in {\rm dom} f \backslash \cX_f$. Then, for all $u \in \BProxSub{t}{f}{x}$ we have $f(u) < f(x)$ and $\textnormal{dist}(u, \cX_f) < \textnormal{dist}(x, \cX_f)$.
\end{proposition}
\begin{proof}
    By definition of broximal operator, we have $f(u) \leq f(x)$. Since by Theorem \ref{thm:c_global_min} equality can hold if and only if $x \in\cX_f$, the inequality must be strict, proving the first part.

    Now, fix any $x \in {\rm dom} f \backslash \cX_f$. By the reasoning above, we know that there exists $u \neq x$ such that $u \in \BProxSub{t}{f}{x}$. If $\textnormal{dist}(x, \cX_f) \leq t$, then $\BProxSub{t}{f}{x} \subseteq \cX_f$, so $\textnormal{dist}(u, \cX_f) = 0$ and the claim holds.
    Otherwise, if $\textnormal{dist}(x, \cX_f) > t$, then \Cref{lemma:neg_inp} says that
    \begin{align*}
        0 > \inp{x - u}{x_{\star} - u}
        \overset{\eqref{fact:1:3pointidentity}}{=} \frac{1}{2} \parens{\norm{x - u}^2 - \norm{x - x_{\star}}^2 + \norm{u - x_{\star}}^2}
    \end{align*}
    for all $x_{\star}\in\cX_f$. It follows that
    \begin{align*}
        \norm{u - x_{\star}}^2 < \norm{x - x_{\star}}^2 - \norm{x - u}^2
        < \norm{x - x_{\star}}^2,
    \end{align*}
    and taking infimum over $x_{\star}\in\cX_f$ gives $\textnormal{dist}(u, \cX_f) < \textnormal{dist}(x, \cX_f)$ as needed.
\end{proof}

Next, similar to convex functions, ball-convex functions guarantee that the steps taken by \newalg\ are of length $t$, as long as the algorithm has not reached the set of global minima.
\begin{proposition}\label{prop:t_step}
    Suppose that $f$ is continuous and Assumption \ref{as:defining_rel} holds. Let $x \in {\rm dom} f$ be such that $B_t(x) \cap \cX_f = \emptyset$. Then $\norm{x-u} = t$, where $u = \BProxSub{t}{f}{x}$.
\end{proposition}
\begin{proof}
    Suppose that there exists $x\in\R^d$ such that $u = \BProxSub{t}{f}{x} \not\in \cX_f$ and $\norm{x-u}<t$. According to Proposition~\ref{prop:single_val}, $u$ is the unique strict minimizer of $f$ over the ball $B_t(x)$.
    Since $u \not\in \cX_f$, by Proposition \ref{prop:local_escape}, there exists $u'\in\BProxSub{t}{f}{u}$ such that $f(u') < f(u)$.
    Next, $u\in \interior B_t(x)$ implies that $f(z) > f(u)$ for all boundary points $z \in B_t(x) \backslash \interior B_t(x)$.
    Furthermore, by single-valuedness of $\BProxSub{t}{f}{\cdot}$ (Proposition \ref{prop:single_val}), $L_f(u) \cap B_t(x) = \{u\}$, where $L_f(u) \eqdef \brac{x \in \R^d: f(x) = f(u)}$.
    By continuity of $f$ and Intermediate Value Theorem, for any path connecting $u$ and $u'$, there must be a point along the path where $f(\cdot)$ equals $f(u)$.
    Hence, $L_f(u)$ forms a closed loop surrounding $B_t(x)$.
    Let us denote the union of $L_f(u)$ and the region it surrounds by $L_f^{\geq}(u)$. 
    Now, $f(z) \geq f(u)$ for all $z\in B_t(x)$, $f(z) > f(u)$ for all $z\in \interior (L_f^{\geq}(u)) \backslash B_t(x)$ and $f(z) = f(u)$ for all $z\in L_f(u)$.

    Consider the balls with centers lying on the line connecting $x$ and $u$. Since $\norm{u-u'} \leq t$ and $f(u') < f(u)$, there exists $\bar{u} \in [u,u'] \cap L_f(u)$.
    Then $\norm{u-\bar{u}} < t$, and hence there exists a ball $B_t(\bar{z})$, where $\bar{z} \in (x,u)$, that is contained in $L_f^{\geq}(u)$, is tangent to $L_f(u)$, and contains $u$.
    But now, $f(y) \geq f(u)$ for all $y \in B_t(\bar{z})$ and $f(z) = f(u)$, so $z, u \in\BProxSub{t}{f}{\bar{z}}$, while $z, u \not\in \cX_f$, contradicting the single-valuedness of broximal operator. This contradiction completes the proof.
\end{proof}

The following bounds on $c_t(x)$, derived directly from inequality \eqref{eq:defining_rel}, play a key role in establishing the convergence result.

\begin{corollary}\label{cor:c_upper_rel}
    Let \Cref{as:defining_rel} hold. Then 
    \begin{align*}
        c_t(x) \norm{x-u}^2 \leq f(x)-f(u)
    \end{align*}
    for all $x \in {\rm dom} f$, where $u \in \BProxSub{t}{f}{x}$.
\end{corollary}
\begin{remark}\label{rem:c_upper_rel}
    In particular, as long as $x_k\not\in\cX_f$ (meaning that $x_k \neq x_{k+1}$), the iterates of \ref{eq:bppm} satisfy
    \begin{align*}
        c_k(x_k) \leq \frac{f(x_k)-f(x_{k+1})}{\norm{x_k-x_{k+1}}^2}.
    \end{align*}
    Otherwise, if $x_k = x_{k+1}$, then $x_k$ is a global minimizer and $c_k(x_k)=0$ by Theorem \ref{thm:c_global_min}.
\end{remark}
\begin{proof}[Proof of \Cref{cor:c_upper_rel}]
    The result follows by letting $y = x$ in inequality \eqref{eq:defining_rel}.
\end{proof}

\begin{corollary}\label{cor:ck_lower}
    Let Assumption \ref{as:defining_rel} hold. Then, for any $x \in {\rm dom} f \backslash \cX_f$
    \begin{align}
        c_t(x) \geq \frac{f(u)-f_{\star}}{\norm{u-x} \norm{u-x_{\star}}}
    \end{align}
    for any $u \in \BProxSub{t}{f}{x}$ and $x_{\star}\in\cX_f$ such that $u \neq x_{\star}$.
\end{corollary}
\begin{remark}\label{rem:c_lower}
    In particular, as long as $x_k\not\in\cX_f$ and $x_{k+1}\neq x_{\star}$, the iterates of \ref{eq:bppm} satisfy
    \begin{align*}
        c_t(x_k) \geq \frac{f(x_{k+1})-f_{\star}}{\norm{x_{k+1}-x_k} \norm{x_{k+1}-x_{\star}}}.
    \end{align*}
\end{remark}
\begin{proof}[Proof of \Cref{cor:ck_lower}]
    Taking $y=x_{\star}\in\cX_f$ in \eqref{eq:defining_rel}, we get
    \begin{align*}
        f_{\star} \geq f(u) + c_t(x) \inp{x-u}{x_{\star}-u}
        \geq f(u) - c_t(x) \norm{x-u} \norm{x_{\star}-u},
    \end{align*}
    where the second inequality follows from Cauchy-Schwarz inequality. Rearranging gives the result.
\end{proof}

\begin{corollary}\label{cor:cxx_decr}
    Let Assumption \ref{as:defining_rel} hold and choose $x\in {\rm dom} f$.
    If $B_t(x) \cap \cX_f = \emptyset$, then $\BProxSub{t}{f}{x}$ is a singleton and
    \begin{eqnarray*}
        \norm{u-w} c_t(u) \leq \norm{x-u} c_t(x),
    \end{eqnarray*}
    for all $w\in\BProxSub{t}{f}{u}$, where $u = \BProxSub{t}{f}{x}$.
\end{corollary}
\begin{remark}\label{rem:cxx_decr}
    In particular, the iterates of \ref{eq:bppm} satisfy
    \begin{align*}
        c_t(x_{k+1}) \norm{x_{k+1}-x_{k+2}} \leq c_t(x_k) \norm{x_{k} - x_{k+1}}.
    \end{align*}
    Hence, an immediate consequence of \Cref{cor:cxx_decr} and \Cref{prop:t_step} is that the constants $c_t(x_k)$ generated by \newalg\ form a non-increasing sequence.
\end{remark}
\begin{proof}[Proof of \Cref{cor:cxx_decr}]
    Let $x\in {\rm dom} f$ be such that $B_t(x) \cap \cX_f = \emptyset$.
    Then, by \Cref{prop:single_val}, $\BProxSub{t}{f}{x}$ is a singleton, so let us denote $u=\BProxSub{t}{f}{x}$ and choose any $w\in\BProxSub{t}{f}{u}$.
    From \Cref{cor:c_upper_rel}, we have
    \begin{align*}
        c_t(u) \norm{u-w}^2 \leq f(u) - f(w),
    \end{align*}
    and using Assumption \ref{as:defining_rel} with $y=w$, we can write
    \begin{align*}
        f(w) \geq f(u) + c_t(x) \inp{x-u}{w-u}.
    \end{align*}
    Hence, applying the Cauchy-Schwarz inequality
    \begin{align*}
        c_t(u) \norm{u-w}^2 \leq f(u) - f(w)
        \leq c_t(x) \inp{x-u}{u-w}
        \leq c_t(x) \norm{x-u} \norm{u-w}.
    \end{align*}
    Since $B_t(x) \cap \cX_f = \emptyset$, it follows that $u\not\in\cX_f$, so $u\neq w$ by \Cref{thm:c_global_min}.
    Dividing by $\norm{u-w}$ yields the result.
\end{proof}

The results above do not require $f$ to be differentiable. Under this additional assumption, a closed-form expression for $c_t(x)$ can be derived.
\begin{lemma}\label{lemma:ck_diff}
    Let $f$ be a differentiable function satisfying \Cref{as:defining_rel}. Then, for any $x\in {\rm dom} f$
    \begin{align*}
        c_t(x) = \frac{\norm{\nabla f(u)}}{t},
    \end{align*}
    where $u\in\BProxSub{t}{f}{x}$.
\end{lemma}
It follows that when $f$ is differentiable, $c_t(x_k)$ in \newalg\ is well-defined and equals
\begin{align*}
    c_t(x_k) = \frac{\norm{\nabla f(x_{k+1})}}{t}.
\end{align*}
\begin{proof}
    Suppose first that $u\not\in\cX_f$. Then, by Proposition \ref{prop:t_step}, we have $\norm{x - u} = t$, and the optimality condition states that
    \begin{align*}
        c_t(x) (x-u) = \nabla f(u)
    \end{align*}
    for some $c_t(x) \geq 0$ (\Cref{thm:3}). Taking norms, we get
    \begin{align*}
        \norm{\nabla f(u)} = c_t(x) \norm{x - u} = c_t(x) t
    \end{align*}
    as required.
    
    The conclusion holds trivially when $u\in\cX_f$, as both sides are equal to $0$ (\Cref{thm:c_global_min}).
\end{proof}

Building on the results above, we can establish convergence guarantees that are fully analogous to those in the convex setting.
\begin{theorem}
    \label{thm:bppm-ballconvex}
    Assume $f: \R^d \mapsto \R \cup \{+\infty\}$ is proper, closed and satisfies \Cref{as:defining_rel}, and let $\{x_k\}_{k\geq0}$ be the iterates of \newalg\ run with any sequence of positive radii $\{t_k\}_{k\geq0}$, where $x_0\in {\rm dom} f$. Then 
    \begin{align*}
        \squeeze f(x_K) - f_{\star} \leq \prod \limits_{k=0}^{K-1} \parens{1 + \frac{t_k}{d_{k+1}}}^{-1} \parens{f(x_0) - f_{\star}},
    \end{align*}
    where $d_k \eqdef \norm{x_k - x_{\star}}$ and $x_\star \in \cX_f$.
    Moreover, if $\sum_{k=0}^{K-1} t_k^2 \geq \textnormal{dist}^2(x_0, \cX_f)$, then $x_K\in\cX_f$.
\end{theorem}
\begin{proof}
    The proof closely mirrors that of \Cref{thm:conv-bppm-lin}, with the primary difference being the starting point. Here, we begin with the defining property from Assumption \ref{as:defining_rel}, using $y=x_{\star}\in\cX_f$, which leads to
    \begin{eqnarray*}
        f(x_{k+1}) - f_{\star} &\leq& - c_t(x_k) \inp{x_k-x_{k+1}}{x_{\star}-x_{k+1}}.
    \end{eqnarray*}
    The remainder of the argument proceeds exactly as in the proof of \Cref{thm:conv-bppm-lin}.
\end{proof}

\begin{theorem}\label{thm:conv-bppm-ball-convex}
    Assume $f: \R^d \mapsto \R \cup \{+\infty\}$ is proper, closed and satisfies \Cref{as:defining_rel}, and choose $x_0\in {\rm dom} f$. Then, for any $K\geq 1$, the iterates of \newalg\ run with $t_k \equiv t > 0$ satisfy
    \begin{align*}
        \squeeze f(x_K) - f_\star \leq \frac{2d_0}{2d_0 + t} \cdot\frac{d_0^2}{2Kt^2} \parens{f(x_0) - f_{\star}}.
    \end{align*}
\end{theorem}
\begin{proof}
    The proof is again entirely analogous to the one presented for the convex case in \Cref{thm:conv-bppm-convex}.
\end{proof}

\subsection{Linear convergence under weaker assumption}\label{sec:bpm_lin_conv_weak}

In fact, we can establish a linear convergence rate without relying on the property that the solution of the local optimization problem lies on the boundary of the ball. Specifically, consider the $k$th iteration of \newalg\ with $t_k\equiv t>0$. For the algorithm to continue progressing, there must exist $x_{k+1} \in B_t(x_k)$ such that $f(x_{k+1}) < f(x_k)$. Similarly, there must exist $x_{k+2} \in B_t(x_{k+1})$ satisfying $f(x_{k+1}) < f(x_{k+2})$ (unless $x_{k+1}\in\cX_f$). Consequently, $x_{k+2}\not\in B_t(x_k)$ (since otherwise, the algorithm would transition directly from $x_k$ to $x_{k+2}$, skipping $x_{k+1}$). This implies that $\norm{x_k - x_{k+2}} > t$, meaning that every \emph{second} iterate is separated by a distance of at least~$t$.

Building on this observation, let us consider the following weaker assumption as a replacement for \Cref{as:defining_rel}:

\begin{assumption}[Weak $B_t$--convexity]\label{as:defining_weak}
    A proper function $f:\R^d\to\R \cup \{+\infty\}$ is said to be \emph{weakly $B_t$--convex} if there exists an optimal point $x_{\star}\in\cX_f$ such that for all $x \in {\rm dom} f$ there exists a function $c:\R^d\to\R_{\geq 0}$ such that
    \begin{align}
        f(u) - f_{\star} &\leq c_t(x) \inp{x-u}{u - x_{\star}}, \label{eq:defining_weak1} \\
        f(x) - f(u) &\geq c_t(x) \norm{x-u}^2 \label{eq:defining_weak2}
    \end{align}
    for any $u \in \BProxSub{t}{f}{x}$.
\end{assumption}

\begin{remark}
    Clearly, there always exist $t>0$ such that the function $f$ is (weakly) $B_t$--convex on a bounded domain (which is sufficient for our application since the iterates of \newalg\ remain bounded, as shown in \Cref{prop:f_dist_decr}). Indeed, for $t$ large enough, we have $\BProxSub{t}{f}{x} \subseteq \cX_f$ for all $x$, and both inequalities hold with $c_t(x)=0$. However, in practice, the radius $t$ can often be chosen much smaller.
\end{remark}

\begin{remark}
    The results in \Cref{thm:c_global_min}, \Cref{lemma:neg_inp}, \Cref{prop:conv_comb}, \Cref{prop:local_escape}, \Cref{prop:f_dist_decr}, \Cref{cor:c_upper_rel} and \Cref{cor:ck_lower} still hold when \Cref{as:defining_weak} is used instead of \Cref{as:defining_rel}.
\end{remark}

The convergence result under \Cref{as:defining_weak} is analogous to the one in Theorems \ref{thm:conv-bppm-lin} and \ref{thm:bppm-ballconvex}, with the difference that the exponent is halved, since only every second step, rather than every iteration, is separated by a distance of $t$.

\begin{theorem}\label{thm:bpm_weak_lin}
    Assume $f: \R^d \mapsto \R \cup \{+\infty\}$ is proper, closed and satisfies \Cref{as:defining_weak}, and choose $x_0\in {\rm dom} f$. Then for any $K\geq 1$, the iterates of \newalg\ satisfy
    \begin{eqnarray}\label{eq:bpm_weak_lin_rate}
        \squeeze f(x_K) - f_{\star} \leq \parens{1 + \frac{t}{d_0}}^{-\ceil{\frac{K-1}{2}}}  \parens{f(x_0) - f_{\star}}.
    \end{eqnarray}
    where $d_0 \eqdef \norm{x_0 - x_{\star}}$ and $x_\star \in \cX_f$.
\end{theorem}
\begin{proof}
    Consider some iteration $k$ such that $x_{k+1} \not\in \cX_f$ (otherwise, the problem is solved in $1$ step) and let $x_{\star}\in\cX_f$ be the optimal point for which \Cref{as:defining_weak} holds.
    We again proceed similarly to the proof of~\Cref{thm:conv-bppm-lin}. Recall the inequality~\eqref{eq:oiqwnvroan}, which states that
    \begin{eqnarray*}
        f(x_{k+1}) - f_{\star} \leq f(x_k) - f_{\star} - \parens{f(x_{k+1}) - f_{\star}} \frac{\norm{x_k-x_{k+1}}}{\norm{x_{k+1} - x_{\star}}}.
    \end{eqnarray*}
    Rearranging, we obtain
    \begin{align}
        f(x_{k+1}) - f_{\star}
        \leq \parens{1 + \frac{\norm{x_k-x_{k+1}}}{\norm{x_0 - x_{\star}}}}^{-1} \parens{f(x_k) - f_{\star}},
    \end{align}
    and applying the bound iteratively gives
    \begin{align}\label{eq:odsinvon}
        f(x_K) - f_{\star}
        \leq \parens{1 + \frac{\norm{x_{K-1}-x_{K}}}{\norm{x_0 - x_{\star}}}}^{-1} \parens{1 + \frac{\norm{x_{K-2}-x_{K-1}}}{\norm{x_0 - x_{\star}}}}^{-1} \ldots \parens{1 + \frac{\norm{x_0-x_1}}{\norm{x_0 - x_{\star}}}}^{-1} \parens{f(x_0) - f_{\star}}.
    \end{align}
    Now, for any $k\geq 0$, we have
    \begin{eqnarray}\label{eq:banvoin}
        &&\hspace{-1cm}\parens{1 + \frac{\norm{x_{k-1}-x_k}}{\norm{x_0 - x_{\star}}}} \parens{1 + \frac{\norm{x_{k-2}-x_{k-1}}}{\norm{x_0 - x_{\star}}}} \nonumber \\
        &=& 1 + \frac{\norm{x_{k-1}-x_k} + \norm{x_{k-2}-x_{k-1}}}{\norm{x_0 - x_{\star}}} + \frac{\norm{x_{k-1}-x_k} \norm{x_{k-2}-x_{k-1}}}{\norm{x_0 - x_{\star}}^2} \nonumber \\
        &\geq& 1 + \frac{\norm{x_{k-2}-x_{k}}}{\norm{x_0 - x_{\star}}} \nonumber \\
        &>& 1 + \frac{t}{\norm{x_0 - x_{\star}}}.
    \end{eqnarray}
    Finally, observing that there are $\ceil{\frac{K-1}{2}}$ pairs of brackets on the right-hand side of inequality \eqref{eq:odsinvon} and using~\eqref{eq:banvoin}, we obtain
    \begin{align*}
        f(x_K) - f_{\star}
        \leq \parens{1 + \frac{t}{\norm{x_0 - x_{\star}}}}^{-\ceil{\frac{K-1}{2}}}  \parens{f(x_0) - f_{\star}}.
    \end{align*}
\end{proof}

\newpage

\section{Non-smooth Optimization}\label{sec:ap_nonsmooth}

\subsection{\texorpdfstring{\algname{PPM} reformulation}{PPM reformulation}}

We now focus on proving the results from \Cref{sec:nonsmooth}, first establishing the correspondence between the proximal and broximal operators.

\LEMMABROXPROX*
Therefore, \newalg\ is equivalent to \algname{PPM} with a specific choice of step size.
\begin{proof}
    Consider the algorithm
    \begin{align*}
        z_{k+1} = \ProxSub{\frac{t_k}{\norm{\nabla f(x_{k+1})}} f}{x_k} \eqdef \underset{z\in\R^d}{\arg\min} \brac{f(z) + \frac{\norm{\nabla f(x_{k+1})}}{2 t_k} \norm{z - x_k}^2}.
    \end{align*}
    Solving the local optimization problem leads to the update rule
    \begin{align*}
        z_{k+1} \overset{\eqref{eq:prox_impl}}{=} x_k - \frac{t_k}{\norm{\nabla f(x_{k+1})}} \nabla f(z_{k+1}),
    \end{align*}
    which shows that $z_{k+1}$ satisfies
    \begin{align}\label{eq:noawsu}
        u = x_k - \frac{t_k}{\norm{\nabla f(x_{k+1})}} \nabla f(u).
    \end{align}
    Now, according to \Cref{thm:alg_update_diff}, the iterates of \newalg\ satisfy
    \begin{align*}
        x_{k+1} = x_k - \frac{t_k}{\norm{\nabla f(x_{k+1})}} \cdot \nabla f(x_{k+1}),
    \end{align*}
    implying that $x_{k+1}$ is also a solution to the same fixed-point equation \eqref{eq:noawsu}. Since the optimization problem associated with the proximal operator is strongly convex, its minimizer is unique, meaning that $z_{k+1} = x_{k+1}$.
\end{proof}

\begin{corollary}\label{cor:prox_brox}
    Let $f:\R^d \to \R$ be a differentiable convex function. Then the iterates of \ref{eq:ppm_update} with $\gamma_k = \nicefrac{t_k}{\norm{\nabla f(\ProxSub{\gamma_k f}{x})}}$ satisfy
    \begin{align*}
        x_{k+1} = \ProxSub{\gamma_k f}{x_k} = \underset{y\in B_{t_k}(x_k)}{\arg\min} \brac{f(y)}.
    \end{align*}
\end{corollary}
\begin{proof}
    The result follows from \Cref{lemma:brox_prox}.
\end{proof}

\subsection{\texorpdfstring{\algname{$\|$GD$\|$} as a practical implementation}{||GD|| as a practical implementation}}

Each iteration of \newalg\ requires solving a constrained optimization problem 
\begin{align}\label{eq:obasbdv}
    \argmin \limits_{z \in B_{t_k}(x_k)} f(z),
\end{align}
the difficulty of which depends on the function $f$ and the step size $t_k$.
In practice, finding an exact solution is often infeasible.
To address this, we propose an implementable modification.

Suppose that $f$ is convex and differentiable. Then, the broximal operator can be expressed as
\begin{align*}
    \BProxSub{t}{f}{x} = x + \argmin_{\norm{u}\leq t} f(x+u) \overset{\eqref{thm:1stbroxthm}}{=} x + \argmin_{\norm{u} = t} f(x+u),
\end{align*}
which can be approximated as
\begin{align*}
    \BProxSub{t}{f}{x} = x + \argmin_{\norm{u} = t} f(x+u)
    \approx x + \argmin_{\norm{u} = t} \brac{f(x) + \inp{\nabla f(x)}{u}}
    = x + \argmin_{\norm{u} =t} \inp{\nabla f(x)}{u}.
\end{align*}
Now, by Cauchy-Schwarz inequality, we find that $\inp{\nabla f(x)}{u} \geq - \norm{\nabla f(x)} \norm{u} = - \norm{\nabla f(x)} t$, with equality achieved when $u = - t \frac{\nabla f(x)}{\norm{\nabla f(x)}}$. Hence
\begin{align*}
    \BProxSub{t}{f}{x} \approx x - t \frac{\nabla f(x)}{\norm{\nabla f(x)}}.
\end{align*}

Building on this idea, we propose an approximate version of \newalg: rather than minimizing $f$ directly, we minimize its linear approximation at the current iterate~$x_k$, replacing step \eqref{eq:obasbdv} by
\begin{align}
    \label{eq:min-linear}
    \argmin \limits_{z \in B_{t_k}(x_k)} \cbrac{f_k(z) \eqdef f(x_k) + \inner{\nabla f(x_k)}{z - x_k}},
\end{align}
resulting in the update rule
\begin{align}
    \label{eq:alg:linearized-BPPM}
    x_{k+1} = \BProxSub{t_k}{f_k}{x_k}. \tag{Linearized \newalg}
\end{align}

\paragraph{Linearized \newalg\ as Normalized \algname{GD}.}
Unlike for the standard \newalg, the local optimization problems \eqref{eq:min-linear} of Linearized \newalg\ always have an explicit closed-form solution.
Indeed, as illustrated above and formalized in \Cref{thm:linearized-bppm}, a simple calculation demonstrates that \eqref{eq:min-linear} is equivalent to 
\begin{align}
    \label{eq:alg:NGD}
    x_{k+1} = x_k - t_k \cdot \frac{\nabla f(x_k)}{\norm{\nabla f(x_k)}} \tag{\algname{$\|$GD$\|$}}
\end{align}
This reformulation establishes that Linearized \newalg\ is exactly \algname{$\|$GD$\|$} applied to the same objective.
Unlike standard \algname{GD}, \algname{$\|$GD$\|$} ignores the gradient's magnitude while preserving its direction.

\THMLINEARPROX*
\begin{proof} 
    The function $f_k$ is linear, and hence convex, so the unique minimizer $x_{k+1}$ of the local problem must lie on the boundary of the ball $B_{t_k}(x_k)$ according to \Cref{thm:1stbroxthm}~\ref{pt:singleton_bdry}.
    Obviously, among those boundary points, $f_k$ is minimized by $x_k - t_k\cdot\frac{\nabla f(x_k)}{\norm{\nabla f(x_k)}}$.
    Consequently, we get
    \begin{align*}
        x_{k+1} = \BProxSub{t_k}{f_k}{x_k} = x_k - t_k \cdot \frac{\nabla f(x_k)}{\norm{\nabla f(x_k)}}.
    \end{align*}
\end{proof}

We establish two convergence guarantees for the linearized variant of \newalg. The first, presented in \Cref{thm:conv-linearized-bppm}, assumes \emph{constant} radii $t_k\equiv t$. Under this setting, the algorithm converges only to a neighborhood of the minimizer. However, this limitation is not fundamental and can be overcome by using \emph{adaptive} step sizes, as detailed in \Cref{thm:conv-linearized-bppm-ada}.
 
\begin{theorem}\label{thm:conv-linearized-bppm}
    Let $f:\R^d \mapsto \R$ be a differentiable convex function.
    Then, for any $K\geq1$, the iterates of Linearized \newalg\ (\algname{$\|$GD$\|$}) run with $t_k\equiv t>0$ satisfy
    \begin{align*}
        \Exp{f(\tilde{x}_K)} - f_\star \leq \frac{G}{2tK}\norm{x_0 - x_\star}^2 + \frac{Gt}{2},
    \end{align*}
    where $\tilde{x}_K$ is chosen randomly from the first $K$ iterates $\cbrac{x_0, x_1, \hdots, x_{K-1}}$ and $G = \sup_{k\in\cbrac{0, 1, \hdots, K-1}}\norm{\nabla f(x_k)}$.
\end{theorem}

\begin{remark}
    Note that Linearized \newalg\ (\algname{$\|$GD$\|$}) converges only to a neighborhood of the minimizer, with the size of this neighborhood determined by the step size $t>0$. In this case, as we are approximating the broximal operator via linearization, increasing the step size does not result in one-step convergence but instead leads to convergence to a larger neighborhood around the minimizer.
\end{remark}

\begin{proof}
    We start with the decomposition
    \begin{align*}
        \norm{x_{k+1} - x_\star}^2 = \norm{x_k - t \frac{\nabla f(x_k)}{\norm{\nabla f(x_k)}} - x_\star}^2
        =\norm{x_k - x_\star}^2 - 2t \frac{\inner{x_k - x_\star}{\nabla f(x_k)}}{\norm{\nabla f(x_k)}} + t^2.
    \end{align*}
    Rearranging the terms, we get 
    \begin{align*}
        \inner{x_k - x_\star}{\nabla f(x_k)} \leq \frac{\norm{\nabla f(x_k)}}{2t}\rbrac{\norm{x_k - x_\star}^2 - \norm{x_{k+1} - x_\star}^2 + t^2}.
    \end{align*}
    Now, notice that by convexity
    \begin{align*}
        f(x_k) - f_\star \leq \inner{\nabla f(x_k)}{x_k - x_\star}.
    \end{align*}
    As a result
    \begin{align*}
        f(x_k) - f_\star \leq \frac{\norm{\nabla f(x_k)}}{2t}\rbrac{\norm{x_k - x_\star}^2 - \norm{x_{k+1} - x_\star}^2 + t^2}.
    \end{align*}
    Summing up both sides for $k\in\cbrac{0, 1, \hdots, K-1}$, where $K$ is the total number of iterations, we get 
    \begin{align*}
        \sum_{k=0}^{K-1}\rbrac{f(x_k) - f_\star} \leq \frac{\sup_{k}\norm{\nabla f(x_k)}}{2t}\cdot \rbrac{\norm{x_0 - x_\star}^2 + Kt^2}.
    \end{align*}
    Lastly, dividing both sides by $K$ and letting $G = \sup_{k\in\cbrac{0, 1, \hdots, K-1}}\norm{\nabla f(x_k)}$ gives 
    \begin{align*}
        \Exp{f(\tilde{x}_K)} - f_\star \leq \frac{\sup_{k}\norm{\nabla f(x_k)}}{2tK}\cdot \rbrac{\norm{x_0 - x_\star}^2 + Kt^2}
        = \frac{G}{2tK}\norm{x_0 - x_\star}^2 + \frac{Gt}{2},
    \end{align*}
    where $\tilde{x}_K$ is chosen randomly from the first $K$ iterates $\cbrac{x_0, x_1, \hdots, x_{K-1}}$.
\end{proof}

\begin{theorem}\label{thm:conv-linearized-bppm-ada}
    Let $f: \R^d \mapsto \R \cup \{+\infty\}$ be proper, closed and convex, and let $\{x_k\}_{k\geq0}$ be the iterates of Linearized \newalg\ (\algname{$\|$GD$\|$}) run with a sequence of positive radii $\{t_k\}_{k\geq0}$ such that
    \begin{align}\label{eq:norm_gd_step_bd}
        t_k \leq \frac{\inp{\nabla f(x_k)}{x_k - x_\star}}{\norm{\nabla f(x_k)}},
    \end{align}
    where $x_0\in {\rm dom} f$ and $x_{\star}\in\cX_f$. Then
    \begin{eqnarray}\label{eq:norm_gd_step_decr}
        \norm{x_{k+1} - x_{\star}}^2 \leq \norm{x_k-x_{\star}}^2 - t_k^2.
    \end{eqnarray}
\end{theorem}
\begin{remark}
    \begin{enumerate}
        \item Condition \eqref{eq:norm_gd_step_bd} is satisfied, for example, by choosing the radius
        \begin{align*}
            t_k = \frac{f(x_k) - f_\star}{\norm{\nabla f(x_k)}}.
        \end{align*}
        For this choice of the radii, Linearized \newalg\ is equivalent to \algname{GD} with Polyak stepsize.

        \item The distance decrease result in \eqref{eq:norm_gd_step_decr} matches the guarantee of the standard \newalg\ in part \ref{pt:dist_decr_bpm} of \Cref{thm:conv-bppm-lin}. However, unlike in \Cref{thm:conv-bppm-lin}, the radii here are \emph{not} arbitrary, and must satisfy the upper bound given in \eqref{eq:norm_gd_step_bd}. This highlights a fundamental trade-off between the potential for arbitrarily fast convergence when minimizing a perfect model of the objective (i.e., the function $f$ itself, as done by \newalg), and \emph{computational feasibility}. While the exact \newalg\ offers strong convergence guarantees, it relies on the access to an exact broximal oracle. When we instead approximate this subproblem--e.g., via linearization, as in the Linearized \newalg--we must restrict the stepsize to ensure the model remains a sufficiently accurate surrogate for $f$; overly large radii would invalidate this approximation.
        
        \item \label{pt:aosbvr} Note that
        \begin{eqnarray*}
            \norm{x_{k+1} - x_{\star}}^2
            = \norm{x_k - t_k \frac{\nabla f(x_k)}{\norm{\nabla f(x_k)}} - x_{\star}}^2
            = \norm{x_k-x_{\star}}^2 - 2 t_k \frac{\inp{\nabla f(x_k)}{x_k - x_{\star}}}{\norm{\nabla f(x_k)}} + t_k^2.
        \end{eqnarray*}
        Hence, the upper bound in \eqref{eq:norm_gd_step_bd} maximizes the one-step decrease of $\norm{x_{k+1} - x_{\star}}^2$. 

        \item By convexity,
        \begin{align*}
            \inp{\nabla f(x_k)}{x_k - x_\star} \geq f(x_k) - f_\star > 0
        \end{align*}
        for $x_k \not\in \cX_f$, so the radii are positive unless the algorithm has already found the optimal solution.
    \end{enumerate}
\end{remark}
\begin{proof}
    Consider some iteration $k$ such that $\cX_f\cap \cB(x_k, t_k)=\emptyset$.
    Applying \Cref{thm:3}, we obtain
    \begin{align}\label{eq:aeorgba}
        f_k(x_{k+1}) - f_k(x_\star) \leq c_{t_k}(x_k) \inp{x_k - x_{k+1}}{x_{k+1} - x_\star},
    \end{align}
    where $c_{t_k}(x_k) = \frac{\norm{\nabla f_k(x_{k+1})}}{t_k} = \frac{\norm{\nabla f(x_k)}}{t_k} > 0$ (\Cref{lemma:ck_diff}).
    Now, note that
    \begin{align*}
        f_k(x_{k+1}) - f_k(x_\star) &= f(x_k) + \inp{\nabla f(x_k)}{x_{k+1} - x_k} - (f(x_k) + \inp{\nabla f(x_k)}{x_\star - x_k}) \\
        &= \inp{\nabla f(x_k)}{x_{k+1} - x_\star} \\
        &= \inp{\nabla f(x_k)}{x_k - x_\star} - t_k \norm{\nabla f(x_k)},
    \end{align*}
    and hence, if
    \begin{align*}
        t_k \leq \frac{\inp{\nabla f(x_k)}{x_k - x_\star}}{\norm{\nabla f(x_k)}},
    \end{align*}
    then $\inp{x_k - x_{k+1}}{x_{k+1} - x_\star} \geq 0$. This in turn means that
    \begin{eqnarray*}
        \norm{x_{k+1} - x_{\star}}^2 &=& \norm{x_k-x_{\star}}^2 - 2 \inp{x_k - x_{k+1}}{x_{k+1} - x_{\star}} - \norm{x_{k+1} - x_k}^2 \\
        &\leq& \norm{x_k-x_{\star}}^2 - t_k^2.
    \end{eqnarray*}
\end{proof}

\newpage

\section{Higher-order Proximal Methods and Acceleration}\label{sec:ap_accel}

This section provides a more detailed examination of the subject introduced in \Cref{sec:acceleration}.

\subsection{\texorpdfstring{\algname{PPM}${\green \sf ^p}$ reformulation}{PPMp reformulation}}

When $f$ is differentiable, it is possible to derive a closed-form update rule for \newalg.

\THMALGUPDIFF*
\begin{proof}
    Following the same reasoning as in the proof of Lemma \ref{lemma:ck_diff}, the optimality condition for the local optimization problem states that
    \begin{align*}
        \nabla f(x_{k+1}) = c_t(x_k) (x_k-x_{k+1})
        \overset{\eqref{lemma:ck_diff}}{=} \frac{\norm{\nabla f(x_{k+1})}}{t_k} (x_k-x_{k+1}).
    \end{align*}
    Rearranging gives the result.
\end{proof}

A doubly implicit update rule similar to the one in \eqref{eq:alg_update_diff} arises in $p$-th order proximal point methods. In particular, recall that the $p$-th order proximal operator is defined as
\begin{align*}
    \ProxPSub{\gamma f}{x}
    \eqdef \underset{z\in\R^d}{\arg\min} \brac{\gamma f(z) + \frac{1}{(p+1)} \cdot \norm{z - x}^{p+1}}.
\end{align*}

Using this definition, \ref{eq:ppmp_update} can be expressed in a more explicit form.

\begin{theorem}\label{thm:ppmpimplicit}
    Let $f:\R^d\to \R$ be a differentiable convex function. Then, the main step of \ref{eq:ppmp_update} can be written in the form
    \begin{align}\label{eq:boaqundf}
        x_{k+1} = x_k - \parens{\frac{\gamma}{\norm{\nabla f(x_{k+1})}^{p-1}}}^{\nicefrac{1}{p}} \nabla f(x_{k+1}).
    \end{align}
\end{theorem}
\begin{proof}
    The result follows directly from the definition of the $p$-th order proximal operator by solving the associated local optimization problem.
\end{proof}

\begin{theorem}\label{thm:bppmp_rate}
    Let $f:\R^d \to \R$ be a differentiable convex function, and let $\{x_k\}_{k\geq0}$ be the iterates of \ref{eq:bppm} with $t_k = \parens{\gamma \norm{\nabla f(x_{k+1})}}^{1/p}$. Then, the algorithm converges with $\cO(\nicefrac{1}{K^p})$ rate.
\end{theorem}

\begin{proof}
    In this case, the algorithm iterates
    \begin{align}\label{eq:904nhasr}
        x_{k+1} = \textnormal{brox}_f^{t_k}(x) \overset{\eqref{thm:alg_update_diff}}{=} x_k - \frac{t_k}{\norm{\nabla f(x_{k+1})}} \nabla f(x_{k+1}).
    \end{align}
    Substituting $t_k = \parens{\gamma \norm{\nabla f(x_{k+1})}}^{1/p}$, \eqref{eq:904nhasr} becomes equivalent to \eqref{eq:boaqundf}. Consequently, the convergence rate is $\cO(\nicefrac{1}{K^p})$, as established in Theorem $1$ by \citet{nesterov2023inexact}.
\end{proof}

\subsection{\texorpdfstring{\algname{AGM} reformulation}{AGM reformulation}}

\THMAGMBRPOX*

\begin{proof}
    Using \Cref{lemma:brox_prox}, the update rule of \ref{eq:agm_bprox_update} can be rewritten as
    \begin{align*}
        x_{k+1} &= \BProxSub{t^x_{k+1}}{l_{y_k}}{x_k} = \ProxSub{\gamma_{k+1} l_{y_k}}{x_k}, \\
        y_{k+1} &= \BProxSub{t^y_{k+1}}{u_{y_k}}{x_{k+1}} = \ProxSub{\gamma_{k+1} u_{y_k}}{x_{k+1}},
    \end{align*}
    where $\gamma_k =\nicefrac{k}{2L}$. Hence, the result is a direct consequence of the analysis of \algname{AGM} by \citet{anh2020understanding}
    (Section $4.2$).
\end{proof}

\newpage

\section{Minimization of the Envelope Function}\label{sec:ap_envelope_gd}

\subsection{Properties of Ball Envelope}
\label{sec:ball-envelope:app}

The concept of the Moreau envelope \cite{moreau1965proximite} has been employed in many recent studies to analyze the (Stochastic) Proximal Point Method (\algname{(S)PPM}).
This is due to the property that solving the proximal minimization problem for the original objective function $f$ is equivalent to applying gradient-based methods to the envelope objective \cite{ryu2014stochastic,li2024power,li2024convergence}.
In this section, we elaborate on the topic introduced in \Cref{sec:moreau} and demonstrate that a similar analysis can be conducted for broximal algorithms.

Following the introduction of the \emph{ball envelope} in \Cref{def:ball-envelope}, we proceed to derive and analyze its key properties. First, the ball envelope offers a lower bound for the associated function $f$.
\begin{lemma}
    \label{lemma:N1}
    Let $f: \R^d \mapsto \R$. Then
    \begin{align*}
        \BMoreauSub{t}{f}{x} \leq f(x)
    \end{align*}
    for any $x \in \R^d$.
\end{lemma}
\begin{proof}
    The proof of the lemma is immediate once we notice that $\BMoreauSub{t}{f}{x} = \min_{z \in B_t(x)} f(z) \leq f(x)$.
\end{proof}

Similar to the Moreau envelope, the ball envelope can be expressed as an infimal convolution of two functions.
\begin{definition}[Infimal convolution]
    \label{def:infimal}
    The \emph{infimal convolution} of two proper functions $f, g: \R^d \mapsto \R\cup\cbrac{+\infty}$ is the function $\rbrac{f \square g}: \R^d \mapsto \R\cup\cbrac{\pm\infty}$ defined by
    \begin{equation*}
        \rbrac{f \square g}(x) = \min_{z \in \R^d}\cbrac{f(z) + g(x - z)}.
    \end{equation*}
\end{definition}

Using the definition of the ball envelope, we obtain
\begin{align*}
    \BMoreauSub{t}{f}{x} &= \min \cbrac{f(z) : \norm{z - x} \leq t} \\
    &= \min_{z \in \R^d}\cbrac{f(z) + \delta_{B_t(x)}(z)} \\
    &= \min_{z \in \R^d}\cbrac{f(z) + \delta_{B_t(0)}(x - z)}  = f \square \delta_{B_t(0)}\\
    &= \min_{u \in \R^d}\cbrac{\delta_{B_t(0)}(u) + f(x - u)} = \delta_{B_t(0)} \square f.
\end{align*}

The next two lemmas are consequences of the above reformulation.
\begin{lemma}
    \label{lemma:N2}
    Let $f: \R^d \mapsto \R \cup \{+\infty\}$ be proper, closed and convex.
    Then $\BMoreauSub{t}{f}{x}$ is convex.
\end{lemma}
\begin{proof}
    We have already shown that $N^t_{f} = \delta_{B_t(0)} \square f$, where $\delta_{B_t(0)}$ is a proper convex function and $f$ is a real-valued convex function.
    Hence, according to Theorem 2.19 of \citet{beck2017first}, $N^t_{f}$ is convex.
\end{proof}

\begin{lemma}
    \label{lemma:N3}
    Let $f: \R^d \mapsto \R$ be convex and $L$--smooth.
    Then $\BMoreauSub{t}{f}{x}$ is $L$--smooth and
    \begin{align*}
        \nabla \BMoreauSub{t}{f}{x} = \nabla f(u)
    \end{align*}
    for any $x\in \R^d$ and $u \in \BProxSub{t}{f}{x}$
\end{lemma}
\begin{proof}
    We know that $N^t_f = \delta_{B_t(0)} \square f$, where $\delta_{B_t(0)}: \R^d \mapsto \R\cup\cbrac{+\infty}$ is proper, closed and convex.
    Since $f$ is convex and $L$--smooth, and the function $\delta_{B_t(0)} \square f$ is real valued, using Theorem $5.30$ of \citet{beck2017first}, we know that $N^t_f = \delta_{B_t(0)} \square f$ is $L$--smooth, and for any $x\in \R^d$ and $u$ that minimizes
    \begin{align*}
       \delta_{B_t(0)}(u) + f(x - u),
    \end{align*}
    we have 
    \begin{align*}
       \nabla \BMoreauSub{t}{f}{x} = \nabla f(x - u).
    \end{align*} 
    This means that $x - u = z$ minimizes $\BMoreauSub{t}{f}{x}$, and hence
    \begin{align*}
       \nabla \BMoreauSub{t}{f}{x} = \nabla f(x - (x - z)) = \nabla f(z).
    \end{align*}
\end{proof}

Unlike the Moreau envelope, which shares the same set of minimizers as the original objective $f$, the ball envelope does not preserve this property. Nevertheless, there exist a certain relationship between the two sets of minimizers.
\begin{lemma}
    \label{lemma:N4}
    Consider $f: \R^d \mapsto \R$ and denote the sets of minimizers of $f$ and $N^t_f$ as $\cX_f$ and $\cX_N$, respectively.
    Then $\cX_f \subset \cX_N$. In particular,
    \begin{align*}
        \cX_{N} = \cbrac{x: \dist{x}{\cX_f} \leq t} = \cX_f + B_t(0),
    \end{align*}
    where ``$+$'' denotes the Minkowski sum.
\end{lemma}
\begin{proof}
    Let us pick any $x_f \in \cX_f$. Then 
    \begin{align*}
        \BMoreauSub{t}{f}{x_f}
        \overset{\eqref{lemma:N1}}{\leq} f(x_f) = \inf f = \inf N^t_f,
    \end{align*}
    which implies that $x_f \in \cX_N$.
    Now, we prove that $\cX_{N} = \{x: \dist{x}{\cX_f} \leq t\}$.
    First, for every $x_N \in \{x: \dist{x}{\cX_f} \leq t\}$, there exists $x_f^\prime \in \cX_f$ such that $\|x_N - x_f^\prime\| \leq t$.
    Therefore
    \begin{align*}
        \BMoreauSub{t}{f}{x_N} \leq f(x_f^\prime) = \inf f,
    \end{align*} 
    which means that $x_N \in \cX_N$.
    On the other hand, for every $x_0 \notin \{x: \dist{x}{\cX_f} \leq t\}$, we know that $B_t(x_0) \cap \cX_f = \emptyset$, so $\BMoreauSub{t}{f}{x_0} > \inf f$.
\end{proof} 

Using the above lemmas, \newalg\ can be reformulated as \algname{GD} applied to the ball envelope function, as established in \Cref{thm:bppm-gd-relation} and discussed in the next section.

\subsection{\texorpdfstring{\algname{GD} reformulation}{GD reformulation}}

\algname{GD} is the cornerstone of modern machine learning and deep learning.
Its stochastic extension, the widely celebrated Stochastic Gradient Descent (\algname{SGD}) algorithm \cite{robbins1951stochastic}, remains a foundational tool in the field.
The significance of \algname{GD} is underscored by the vast array of variants, extending the algorithm to a wide range of settings.
Examples include compression \citep{alistarh2017qsgd, khirirat2018distributed, richtarik2021ef21, gruntkowska2023ef21p}, \algname{SGD} with momentum \citep{loizou2017linearly, liu2020improved}, variance reduction \citep{gower2020variance,johnson2013accelerating,gorbunov2021marina,tyurin2024dasha,li2023marina} or adaptive and matrix step sizes \citep{bach2019universal, malitsky2019adaptive, horvath2022adaptive, yang2023adaptive, li2023det}.

The existence of a link between \algname{GD} and \newalg\ is a promising sign for its potential.
For clarity, we restate the relevant result.
\THMBPPMGD*
This connection between \newalg\ and \algname{GD} opens the door to incorporating established techniques and analyses into \newalg.

\begin{proof}
    According to \Cref{lemma:N3}, we have $\nabla N^{t_k}_f(x_k) = \nabla f(x_{k+1})$.
    Since $f$ is $L$--smooth, it is differentiable, so \Cref{thm:3} gives
    \begin{align*}
        c_{t_k}(x_k)(x_k - x_{k+1}) = \nabla f(x_{k+1}) = \nabla N^{t_k}_f(x_k).
    \end{align*}
    Now, if $c_{t_k}(x_k)=0$, then $\nabla f(x_{k+1}) = \nabla N^{t_k}_f(x_k)=0$, so $x_k$ and $x_{k+1}$ are minimizers of $N^{t_k}_f$ and $f$, respectively, and the algorithm terminates. Otherwise, $x_{k+1}\not\in\cX_f$ by \Cref{thm:c_global_min}, and rearranging terms gives
    \begin{align*}
        x_{k+1} = x_k - \frac{1}{c_{t_k}(x_k)}\cdot \nabla N^{t_k}_f(x_k),
    \end{align*}
    which is exactly gradient descent on $N^{t_k}_f$ with a step size of
    \begin{align*}
        \frac{1}{c_{t_k}(x_k)}
        \overset{\eqref{lemma:ck_diff}}{=} \frac{t_k}{\norm{\nabla f(x_{k+1})}}
        \overset{\eqref{lemma:N3}}{=} \frac{t_k}{\norm{\nabla N^{t_k}_f(x_k)}}.
    \end{align*}
\end{proof}

\newpage

\section{Stochastic Case}\label{sec:ap_stochastic}

\begin{algorithm}[t]
    \caption{Stochastic Ball-Proximal Point Method (\algname{SBPM})}
    \begin{algorithmic}[1]\label{alg:sbpm}
    \STATE \textbf{Input:} radii $t_k > 0$ for $k \geq 0$, starting point $x_0 \in {\rm dom} f$
    \FOR{$k = 0,1,2,\dots$}
    \STATE Sample $\xi_k \sim \mathcal{U}\{1,\ldots,n\}$
    \STATE $x_{k+1} = \Pi\rbrac{x_k, \BProxSub{t_k}{f_{\xi_k}}{x_k}}$
    \ENDFOR
    \end{algorithmic}
\end{algorithm}

In this section, we extend \newalg\ to the stochastic setting.
Specifically, we consider the distributed optimization problem 
\begin{align*}
    \min_{x \in \R^d}\cbrac{f(x) \eqdef \frac{1}{n}\sum_{i=1}^{n}f_i(x)},
\end{align*}
where each function $f_i: \R^d \mapsto \R$, $i\in[n]$ is a local objective associated with the $i$th client. 
A natural extension of \newalg\ to the stochastic case would be
\begin{align}\label{eq:no_proj_upd}
    x_{k+1} \in \BProxSub{t_k}{f_{\xi_k}}{x_k},
\end{align}
where $\xi_k \in [n]$ is the index of the selected client, sampled uniformly at random.
However, due to the additional stochasticity, the algorithm fails to converge even in the simplest case when a sufficiently large constant step size $t$ is used.
This is demonstrated by the following example.
\begin{example}
    Consider the case where $n=2$, and both $f_1$ and $f_2$ are convex and smooth functions.
    Let $\cX_{f_1}$ and $\cX_{f_2}$ denote their respective sets of minimizers, and assume $\cX_{f_1} \cap \cX_{f_2} \neq \emptyset$.
    Suppose that algorithm \eqref{eq:no_proj_upd} is initialized at a point $x_0 \in \cX_{f_1} \backslash \cX_{f_2}$ with a sufficiently large step size $t$ such that $\cX_{f_1} \subseteq B_t(z)$ for any $z \in \cX_{f_2}$ and $\cX_{f_2} \subseteq B_t(z)$ for any $z \in \cX_{f_1}$.
    In this scenario, the next iterate is not uniquely defined, and the algorithm can alternate between $\cX_{f_2}$ and $\cX_{f_1}$ without converging.
\end{example}

Fortunately, this issue can be resolved with a simple modification. To handle the stochastic case, we propose the Stochastic Ball-Proximal Point Method (\Cref{alg:sbpm}), which iterates
\begin{align}
    \label{alg:bppm:stochastic}
    x_{k+1} = \Pi\rbrac{x_k, \BProxSub{t_k}{f_{\xi_k}}{x_k}}, \tag{\algname{SBPM}}
\end{align}
where $\xi_k \sim \mathcal{U}\{1,\ldots,n\}$ and $\Pi(\cdot, \cX)$ denotes the Euclidean projection onto the set $\cX$.
The projection step is crucial for handling discrepancies in the minimizer sets across different client objectives and managing the potential multi-valuedness of the broximal operator.

Before presenting the convergence result, we first introduce several essential lemmas. 
For the purpose of analyzing the algorithm, we assume each local objective function $f_i$ to be convex and $L_i$--smooth.
Hence, Theorems \ref{thm:1stbroxthm} and \ref{thm:3} hold directly.
However, the constant $c_t(x_k)$ in \Cref{thm:3} depends on both the current iterate $x_k$ and the function $f$, leading to variability across different client functions.
To reflect this dependency, we denote the constant associated with iterate $x_k$ and function $f_{\xi_k}$ as $c_t(x_k, \xi_k)$.

\begin{lemma}[Projection]
    \label{lemma:sameproj}
    Let $k\geq 0$ be an iteration of \ref{alg:bppm:stochastic} such that $B_t(x_k) \cap \cX_{f_{\xi_k}} \neq \emptyset$.
    Then
    \begin{align*}
        x_{k+1} = \Pi\rbrac{x_k, \BProxSub{t}{f_{\xi_k}}{x_k}} = \Pi\rbrac{x_k, B_t(x_k) \cap \cX_{f_{\xi_k}}} = \Pi\rbrac{x_k, \cX_{f_{\xi_k}}}.
    \end{align*}
\end{lemma}
\begin{proof}
    First, suppose that $B_t(x_k) \cap \cX_{f_{\xi_k}} \neq \emptyset$. Using the definition of $\BProxSub{t}{f_{\xi_k}}{x_k}$, it is obvious that 
    \begin{align*}
        \BProxSub{t}{f_{\xi_k}}{x_k} = B_t(x_k) \cap \cX_{f_{\xi_k}}.
    \end{align*}
    Now, assume that $x^\prime_{k+1} \eqdef \Pi(x_k, \cX_{f_{\xi_k}}) \neq x_{k+1}$. Then $x^\prime_{k+1} \notin \BProxSub{t}{f_{\xi_k}}{x_k}$, since otherwise one would have $\Pi(x_k, \BProxSub{t}{f_{\xi_k}}{x_k}) = x^\prime_{k+1}$, in which case $x_{k+1} = x^\prime_{k+1}$.
    However, if $x^\prime_{k+1} \notin \BProxSub{t}{f_{\xi_k}}{x_k}$, then $\norm{x_{k+1} - x_k} \leq t < \norm{x^\prime_{k+1} - x_k}$.
    Since $x_{k+1} \in \cX_{f_{\xi_k}}$, this contradicts the fact that $x^\prime_{k+1}$ is a projection.
\end{proof}

The above lemma allows us to rewrite \ref{alg:bppm:stochastic} as
\begin{align*}
    x_{k+1} = \begin{cases*}
        \BProxSub{t}{f_{\xi_k}}{x_k} & \text{  if $B_t(x_k) \cap \cX_{f_{\xi_k}} = \emptyset$,}\\
        \Pi\rbrac{x_k, \cX_{f_{\xi_k}}} & \text{  otherwise}.
    \end{cases*}
\end{align*}
Note that when $B_t(x_k) \cap \cX_{f_{\xi}} \neq \emptyset$, we have $\Pi(x_k, \cX_{f_{\xi_k}}) \in \BProxSub{t}{f_{\xi_k}}{x_k}$, and hence many existing tools developed for the single-node case remain applicable in the distributed setting.

The extra projection enables us to establish additional properties that guarantee convergence of the method.

\begin{lemma}[Descent lemma I]
    \label{lemma:descent:1}
    Let each local objective function $f_i: \R^d \mapsto \R$ be convex and $L_i$--smooth.
    Then, the iterates of \ref{alg:bppm:stochastic} satisfy
    \begin{align*}
        -c_t(x_k, \xi_k)\inner{x_{k+1} - x_\star}{x_k - x_{k+1}} \leq \rbrac{f_{\xi_k}(x_\star) - f_{\xi_k}\rbrac{\SBProjProxSub{t}{f_{\xi_k}}{x_k}}},
    \end{align*}
    where $x_\star$ is any minimizer of $f_{\xi_k}$.
\end{lemma}
\begin{proof}
    According to \Cref{thm:3}, we have 
    \begin{align*}
        c_t(x_k, \xi_k)\rbrac{x_k - x_{k+1}} = \nabla f_{\xi_k}(x_{k+1}).
    \end{align*}
    Therefore, by convexity of $f_{\xi_k}$,
    \begin{align*}
        - c_t(x_k, \xi_k) \inner{x_{k+1} - x_\star}{x_k - x_{k+1}} &=  \inner{x_\star - x_{k+1}}{\nabla f_{\xi_k}(x_{k+1})} \\
        &\leq f_{\xi_k}(x_\star) - f_{\xi_k}(x_{k+1}) \\
        &= f_{\xi_k}(x_\star) - f_{\xi_k}\rbrac{\SBProjProxSub{t}{f_{\xi_k}}{x_k}}
    \end{align*}
    as needed.
\end{proof}

\begin{lemma}[Descent lemma II]
    \label{lemma:descent:2}
    Let each local objective function $f_i: \R^d \mapsto \R$ be convex and $L_i$--smooth.
    Then, the iterates of \ref{alg:bppm:stochastic} satisfy
    \begin{align*}
        \norm{x_{k} - x_{k+1}}^2 \geq \frac{1}{\frac{L_{\xi_k}}{2} + c_t(x_k, \xi_k)}\rbrac{f_{\xi_k}(x_k) - f_{\xi_k}(x_{k+1})}.
    \end{align*}
\end{lemma}
\begin{proof}
    Since $f_{\xi_k}$ is $L_{\xi_k}$--smooth, we have 
    \begin{align*}
        f_{\xi_k}(x_k) - f_{\xi_k}(x_{k+1}) - \inner{\nabla f_{\xi_k}(x_{k+1})}{x_k - x_{k+1}} \leq \frac{L_{\xi_k}}{2}\norm{x_k - x_{k+1}}^2.
    \end{align*}
    Next, by \Cref{thm:3}, it holds that
    \begin{align*}
        \inner{\nabla f_{\xi_k}(x_{k+1})}{x_k - x_{k+1}} = c_t(x_k, \xi_k)\norm{x_k - x_{k+1}}^2,
    \end{align*}
    which implies
    \begin{align*}
        f_{\xi_k}(x_k) - f_{\xi_k}(x_{k+1}) \leq \rbrac{\frac{L_{\xi_k}}{2} + c_t(x_k, \xi_k)}\norm{x_k - x_{k+1}}^2.
    \end{align*}
    Rearranging the terms gives the result.
\end{proof}

\begin{lemma}[Descent lemma III]
    \label{lemma:descent:3}
    Let each local objective function $f_i: \R^d \mapsto \R$ be convex and $L_i$--smooth.
    Then, the iterates of \ref{alg:bppm:stochastic} satisfy
    \begin{align*}
        \norm{x_{k+1} - x_\star}^2 \leq \norm{x_k - x_\star}^2 - \frac{1}{\frac{L_{\xi_k}}{2} + c_t\rbrac{x_k, \xi_k}}\rbrac{f_{\xi_k}\rbrac{x_k} - f_{\xi_k}\rbrac{x_\star}},
    \end{align*}
    where $x_\star$ is any minimizer of $f_{\xi_k}$.
\end{lemma}
\begin{proof}
    We start with the simple decomposition 
    \begin{align}
        \label{eq:noname1}
        \norm{x_{k+1} - x_\star}^2 = \norm{x_k - x_\star}^2 - \norm{x_k - x_{k+1}}^2 - 2\inner{x_{k+1} - x_\star}{x_k - x_{k+1}}.
    \end{align}
    Now, let us consider two cases.
    \paragraph{Case 1: $c_t\rbrac{x_k, \xi_k} > 0$.} In this case, combining \Cref{lemma:descent:1} and \Cref{lemma:descent:2} gives
    \begin{align*}
        \norm{x_{k+1} - x_\star}^2 \leq \norm{x_k - x_\star}^2 - \frac{1}{\frac{L_{\xi_k}}{2} + c_t(x_k, \xi_k)}\rbrac{f_{\xi_k}(x_k) - f_{\xi_k}(x_{k+1})}
        - \frac{2}{c_t(x_k, \xi_k)}\rbrac{f_{\xi_k}(x_{k+1}) - f_{\xi_k}\rbrac{x_{\star}}}.
    \end{align*}
    Now, notice that 
    \begin{align*}
        \min\cbrac{\frac{2}{c_t\rbrac{x_k, \xi_k}}, \frac{1}{\frac{L_{\xi_k}}{2} + c_t\rbrac{x_k, \xi_k}}} = \frac{1}{\frac{L_{\xi_k}}{2} + c_t\rbrac{x_k, \xi_k}},
    \end{align*}
    and by the definition of broximal operator, it holds that
    \begin{align*}
        f_{\xi_k}(x_k) - f_{\xi_k}(x_{k+1}) \geq 0.
    \end{align*}
    Moreover, since $x_\star$ is a minimizer of $f_{\xi_k}$, it is obvious that
    \begin{align*}
        f_{\xi_k}(x_{k+1}) - f_{\xi_k}(x_\star) \geq 0.
    \end{align*}
    Combining the above inequalities gives
    \begin{align*}
        \norm{x_{k+1} - x_\star}^2 \leq \norm{x_k - x_\star}^2 - \frac{1}{\frac{L_{\xi_k}}{2} + c_t\rbrac{x_k, \xi_k}}\rbrac{f_{\xi_k}\rbrac{x_k} - f_{\xi_k}\rbrac{x_\star}}.
    \end{align*}

    \paragraph{Case 2: $c_t\rbrac{x_k, \xi_k} = 0$.}
    The condition $c_t(x_k, \xi_k) = 0$ implies that $x_{k+1} \in \cX_{f_{\xi_k}}$ (\Cref{thm:c_global_min}), so $x_{k+1} = \Pi(x_k, \BProxSub{t}{f_{\xi_k}}{x_k}) \in \cX_{f_{\xi_k}}$.
    By \Cref{lemma:sameproj}, we know that $x_{k+1} = \Pi(x_k, \cX_{f_{\xi_k}})$, which implies that $\inner{x_{k+1} - x_{\star}}{x_k - x_{k+1}} \geq 0$ by the second projection theorem (Theorem 6.14 of \citep{beck2017first}).
    Hence, using \Cref{lemma:descent:2}, inequality \eqref{eq:noname1} simplifies to
    \begin{align*}
        \norm{x_{k+1} - x_\star}^2 &\leq \norm{x_k - x_\star}^2 - \frac{1}{\frac{L_{\xi_k}}{2} + c_t(x_k, \xi_k)}\rbrac{f_{\xi_k}(x_k) - f_{\xi_k}(x_{k+1})} \\
        &= \norm{x_k - x_\star}^2 - \frac{1}{\frac{L_{\xi_k}}{2} + c_t\rbrac{x_k, \xi_k}}\rbrac{f_{\xi_k}\rbrac{x_k} - f_{\xi_k}\rbrac{x_\star}},
    \end{align*}
    which finishes the proof.
\end{proof}

Finally, one can prove the following convergence guarantee:
\begin{theorem}
    \label{thm:sbppm}
    Let each local objective function $f_i: \R^d \mapsto \R$ be convex  and $L_i$--smooth, and assume that there exists $x_\star$ such that $\nabla f_i(x_\star) = 0$ for all $i\in[n]$.
    Then, for any $K\geq 1$, the iterates of \ref{alg:bppm:stochastic} with $t_k \equiv t > 0$ satisfy
    \begin{align*}
        \Exp{f(\tilde{x}_K)} - f_\star \leq L_{\max}\rbrac{1 + \frac{\norm{x_0 - x_\star}^2}{t^2}}\cdot \frac{\norm{x_0 - x_\star}^2}{2K}.
    \end{align*}
    where $\tilde{x}_K$ is chosen uniformly at random from the first $K$ iterates $\cbrac{x_0, \hdots, x_{K-1}}$ and $L_{\max} \eqdef \max_{i \in [n]} L_i$.
\end{theorem}
\begin{remark}[Semi-adaptivity]
    An observation from \Cref{thm:sbppm} is that the algorithm converges regardless of the step size $t$.
    Notably, similar to \algname{SPPM}, smoothness is not required to determine the step size.
    A smaller $t$ results in a slower convergence rate, but it simplifies the local subproblems, making them easier to solve.
    Conversely, a larger $t$ improves the convergence rate, but increases the complexity of each subproblem, thereby requiring more local computation.
\end{remark}

\begin{proof}[Proof of \Cref{thm:sbppm}]
    Let $x_\star$ be a common minimizer of all client functions.
    We start with the inequality from \Cref{lemma:descent:3}
    \begin{align*}
        \norm{x_{k+1} - x_\star}^2 \leq \norm{x_k - x_\star}^2 - \frac{1}{\frac{L_{\xi_k}}{2} + c_t\rbrac{x_k, \xi_k}}\rbrac{f_{\xi_k}\rbrac{x_k} - f_{\xi_k}\rbrac{x_\star}}.
    \end{align*}
    Taking expectation conditional on $x_k$, we have 
    \begin{align*}
        \ExpCond{\norm{x_{k+1} - x_\star}^2}{x_k} &\leq \norm{x_k - x_\star}^2 - \frac{1}{n}\sum_{i=1}^{n}\frac{1}{\frac{L_i}{2} + c_t(x_k, i)}\rbrac{f_i(x_k) - f_i\rbrac{x_\star}}. 
    \end{align*}
    We can further simplify the recursion as 
    \begin{align*}
        \ExpCond{\norm{x_{k+1} - x_\star}^2}{x_k} \leq \norm{x_k - x_\star}^2 - \frac{1}{\frac{L_{\max}}{2} + \cmax{k}}\rbrac{f(x_k) - f(x_\star)},
    \end{align*}
    where $\cmax{k} = \max_{i \in [n]}c_t(x_k, i)$.
    Taking expectation again and using the tower property gives 
    \begin{align*}
        \Exp{\norm{x_{k+1} - x_\star}^2} \leq \Exp{\norm{x_k - x_\star}^2} - \Exp{\frac{1}{\frac{L_{\max}}{2} + \cmax{k}}\rbrac{f(x_k) - f(x_\star)}},
    \end{align*}
    and hence, unrolling the recurrence,
    \begin{align*}
        \sum_{k=0}^{K-1}\frac{1}{\frac{L_{\max}}{2} + \cmax{k}}\Exp{f(x_k) - f_\star} \leq \norm{x_0 - x_\star}^2.
    \end{align*}
    Denoting $c_{\max} = \sup_{k \in \cbrac{0, 1, \hdots, K-1}}\cmax{k}$, we obtain 
    \begin{align}
        \label{eq:temp-res-1}
        \Exp{f(\tilde{x}_K)} - f_\star \leq \rbrac{\frac{L_{\max}}{2} + c_{\max}} \cdot \frac{\norm{x_0 - x_\star}^2}{K},
    \end{align}
    where $\tilde{x}_K$ is sampled randomly from the first $K$ iterates $\cbrac{x_0, x_1, \hdots, x_{K-1}}$.

    Now, we proceed to obtain an upper bound on~$c_{\max}$.
    Consider some client function $f_{i}$, $i\in[n]$. Since, by definition, $c_t(x_k, i) \geq 0$, it suffices to consider the case $c_t(x_k, i) \neq 0$. From \Cref{thm:3} we know that $\SBProjProxSub{t}{f_{i}}{x_k} \notin \cX_{f_{i}}$, so $\norm{x_k - \SBProjProxSub{t}{f_{i}}{x_k}} = t$ (\Cref{prop:t_step}). Using \Cref{cor:c_upper_rel}, we can deduce that
    \begin{align*}
        c_t(x_k, i) \leq \frac{f_i(x_k) - f_i\rbrac{\SBProjProxSub{t}{f_i}{x_k}}}{\norm{x_k - \SBProjProxSub{t}{f_i}{x_k}}^2} = \frac{f_i(x_k) - f_i\rbrac{\SBProjProxSub{t}{f_i}{x_k}}}{t^2}
        \leq \frac{f_i(x_k) - f_i(x_\star)}{t^2}.
    \end{align*}
    Using the $L_i$--smoothness of $f_i$, we get 
    \begin{align*}
        c_t(x_k, i) \leq \frac{L_i\norm{x_k - x_\star}^2}{2t^2},
    \end{align*}
    and consequently 
    \begin{align*}
        \cmax{k} \leq \frac{L_{\max}\norm{x_k - x_\star}^2}{2t^2}. 
    \end{align*}
    Now, since 
    \begin{align*}
        \norm{x_{k+1} - x_\star}^2 = \norm{x_k - x_\star}^2 - \norm{x_k - x_{k+1}}^2 - 2\inner{x_{k+1} - x_\star}{x_k - x_{k+1}},
    \end{align*}
    and by \Cref{lemma:descent:1} 
    \begin{align*}
        \inner{x_{k+1} - x_\star}{x_k - x_{k+1}} \geq 0,
    \end{align*}
    it follows that $\norm{x_{k+1} - x_\star}^2 \leq \norm{x_k - x_\star}^2$.
    As a result, we have 
    \begin{align*}
        \cmax{k} \leq \frac{L_{\max}\norm{x_0 - x_\star}^2}{2t^2}.
    \end{align*}
    Plugging this back to \eqref{eq:temp-res-1} gives
    \begin{align*}
        \Exp{f(\tilde{x}_K)} - f_\star \leq L_{\max}\rbrac{1 + \frac{\norm{x_0 - x_\star}^2}{t^2}}\cdot \frac{\norm{x_0 - x_\star}^2}{2K}.
    \end{align*}
\end{proof}

\newpage

\section{Bregman Broximal Point Method}\label{sec:bregman}
A natural extension of \newalg\ is to replace the ball constraint with a more general one.
In this section, we propose a generalization based on the Bregman divergence.
\begin{definition}[Bregman divergence]
    \label{def:bregman}
    Let $h: \R^d \mapsto \R$ be a continuously differentiable function. 
    The \emph{Bregman divergence} between two points $x, y \in \R^d$ associated with $h$ is the mapping $\R^d\times\R^d\to\R$ defined as
    \begin{align*}
        \breg{h}{x}{y} = h(x) - h(y) - \inner{\nabla h(y)}{x - y}.
    \end{align*}
\end{definition}
The Bregman divergence can be intuitively understood by fixing a point  $x_0 \in \R^d$ and interpreting $\breg{h}{x}{x_0}$ as the difference between the function $h$ and its linear approximation at $x_0$, evaluated at $x$.
When $h(x) = \norm{x}^2$, the Bregman divergence simplifies to $\breg{h}{x}{y} = \norm{x - y}^2 = \breg{h}{y}{x}$.
In general, however, the Bregman divergence is not symmetric.

In this section, we address the minimization problem \eqref{eq:main-prob}, assuming that $f:\R^d\to \R$ is a differentiable convex function.
We propose the following algorithm:
\begin{align}
    \label{eq:alg:bregmanbppm}
    x_{k+1} \in \BBProxSub{t}{f}{h}{x_k} \eqdef \argmin_{z \in \R^d}\cbrac{f(z): \breg{h}{z}{x_k} \leq t^2}. \tag{\algname{BregBPM}}
\end{align}
We refer to $\BBProxSub{t}{f}{h}{\cdot}$ as the \emph{Bregman Broximal Operator}, and name the corresponding algorithm the \emph{Bregman Ball-Proximal Method} (\algname{BregBPM}).

At each iteration, \algname{BregBPM} minimizes $f$ within a localized region around $x_k$, defined by the constraint $\cH_k \eqdef \{z: \breg{h}{z}{x_k} \leq t^2\}$. This translates to solving the constrained optimization problem
\begin{align*}
    \min\limits_{z \in \cH_k} f(z) \quad \Leftrightarrow \quad \min\limits_{z \in \R^d} f(z) + \delta_{\cH_k}(z).
\end{align*}
In this case, the optimality condition states that
\begin{align*}
    0 \in \partial \rbrac{f + \delta_{\cH_k}}(x_{k+1}).
\end{align*}
The function $f$ is differentiable and convex, and the indicator function of a closed convex set is proper, closed and convex. One can also show using an argument similar to that in the proof of \Cref{thm:3} that $\ri(\cH_k) \cap \ri({\rm dom}(f)) \neq \emptyset$. Therefore, according to \Cref{fact:2:sum-rule-subdiff} we have
\begin{align*}
    0 \in \nabla f(x_{k+1}) + \partial \delta_{\cH_k}(x_{k+1}),
\end{align*}
which implies 
\begin{align*}
    - \nabla f(x_{k+1}) \in \partial \delta_{\cH_k}(x_{k+1}).
\end{align*}

To proceed with the analysis, we first establish several essential results.
For analytical convenience, we assume that $h$ is strictly convex, thus ensuring that $\breg{h}{x}{y}$ is strictly convex with respect to its first argument, as established in \Cref{lemma:strict-convexity-breg}.
\begin{lemma}
    \label{lemma:strict-convexity-breg}
    Let $h: \R^d \mapsto \R$ be a continuously differentiable and strictly convex function.
    Then, for any fixed $y \in \R^d$, the Bregman divergence $\breg{h}{x}{y}$ is strictly convex with respect to $x$.
\end{lemma}
\begin{proof}
    For any two distinct points $x_1, x_2 \in \R^d$ and $\lambda \in (0, 1)$, we have 
    \begin{align*}
        \breg{h}{\lambda x_1 + \rbrac{1 - \lambda} x_2}{y} &= h\rbrac{\lambda x_1 + \rbrac{1 - \lambda} x_2} - h(y) - \lambda\inner{\nabla h(y)}{x_1} - \rbrac{1 - \lambda}\inner{\nabla h(y)}{x_2} \\
        &< \lambda\rbrac{h(x_1) - h(y) - \inner{\nabla h(y)}{x_1}} + \rbrac{1 - \lambda}\rbrac{h(x_2) - h(y) - \inner{\nabla h(y)}{x_2}} \\
        &= \lambda\breg{h}{x_1}{y} + \rbrac{1 - \lambda}\breg{h}{x_2}{y}
    \end{align*}
    as needed.
\end{proof}

The following lemma demonstrates that strict convexity ensures that the Bregman broximal operator is single-valued, possibly except for the last iteration of the algorithm.
\begin{lemma}
    \label{lemma:breg-singleton}
    Let $h: \R^d \mapsto \R$ be a continuously differentiable and strictly convex function.
    If $\BBProxSub{t}{f}{h}{x}\not\subseteq \cX_f$, then the mapping $x \mapsto \BBProxSub{t}{f}{h}{x}$ is single-valued and $u = \BBProxSub{t}{f}{h}{x}\in\bdry\cH$, where $\cH \eqdef \cbrac{z: \breg{h}{z}{x} \leq t^2}$.
\end{lemma}

\begin{remark}
    For $z \in \bdry\cH_k$, we always have $\breg{h}{z}{x_k} = t^2$, which means that $\bdry \cH_k$ is a level set of $\breg{h}{z}{x_k}$.
\end{remark}

\begin{proof}
    Suppose that there exists $u \in \BBProxSub{t}{f}{h}{x}$ such that $u \in \interior\cH$ and take any $x_\star \in \cX_f$. Since $\cH \cap \cX_f = \emptyset$, we have $x_\star \notin \cH$.
    Hence, the line segment connecting $x_\star$ and $u$ must intersect $\bdry\cH$ at a point $\tilde{u} \eqdef \lambda u + \rbrac{1 - \lambda}x_\star$ for some $\lambda \in (0, 1)$.
    Using strict convexity of $f$, we obtain 
    \begin{align*}
        f(\tilde{u}) < \lambda f(u) + (1 - \lambda)f(x_\star) < f(u),
    \end{align*}
    which contradicts the fact that $u$ minimizes $f$ on $\cH$.
    As a result, $u$ must lie on the boundary of $\cH$.
    
    Now, suppose that there exist two distinct points $u_1, u_2 \in \BBProxSub{t}{f}{h}{x}$.
    The strict convexity of $\breg{h}{z}{x}$ in its first argument guarantees that $\cH$ is strictly convex as well.
    Hence, the line segment connecting $u_1$ and $u_2$ lies in the interior of $\cH$, and for any $\alpha \in(0,1)$
    \begin{align*}
        f\rbrac{\alpha u_1 + \rbrac{1 - \alpha} u_2} \leq \alpha f(u_1) + \rbrac{1 - \alpha}f(u_2) = f(u_1) = f(u_2),
    \end{align*} 
    which implies that $\alpha u_1 + (1 - \alpha) u_2 \in \interior \cH$ is also a minimizer of $f$ on $\cH$. This contradicts the fact that a minimizer must lie on the boundary. 
\end{proof}

\begin{lemma}
    \label{lemma:normal-cone-convex}
    Let $\phi: \R^d \mapsto \R$ be a differentiable convex function, $c > \inf \phi$ be a constant, and denote $\cC = \cbrac{x \in \R^d: \phi(x) \leq c}$.
    Then for any $x \in \bdry \cC$, it holds that
    \begin{align*}
        \cN_{\cC}(x) = \cbrac{\lambda\nabla \phi(x), \lambda \geq 0}.
    \end{align*}  
\end{lemma}
\begin{proof}
    Let $z \in \bdry \cC$ and denote  
    \begin{align*}
        \cH(z) &\eqdef \cbrac{y \in \R^d: \phi(z) + \inner{\nabla \phi(z)}{y - z} = \phi(z)} \\
        &= \cbrac{y \in \R^d: \inner{\nabla \phi(z)}{y - z} = 0}.
    \end{align*}
    Then, $\cH(z)$ is a supporting hyperplane of the convex set $\cC$, and $\nabla \phi(x)$ is a normal vector to this hyperplane.
    Now, recall the definition of the normal cone
    \begin{align*}
        \cN_{\cC}(x) = \cbrac{y \in \R^d : \inner{y}{z - x} \leq 0 \quad \forall z \in \cC}.
    \end{align*}
    For any $z \in \cC$, using convexity, we have 
    \begin{align*}
        \phi(z) \geq \phi(x) + \inner{\nabla \phi(x)}{z - x},
    \end{align*}
    which indicates that
    \begin{align*}
        \inner{\lambda\nabla \phi(x)}{z - x} \leq 0 \quad \forall \lambda \geq 0,
    \end{align*}
    for all $z \in \cC$, implying that $\lambda\nabla \phi(x) \in \cN_{\cC}(x)$.

    Now, assume that there exists $v \in \cN_{\cC}(x) \neq \lambda\nabla \phi(x)$ for any $\lambda \geq 0$.
    Since $\nabla \phi(x) \neq 0$ and $v \neq 0$, there exists $h \in \R^d$ such that 
    \begin{align*}
        \inner{\nabla \phi(x)}{h} < 0 \text{ and } \inner{v}{h} > 0.
    \end{align*}
    Let $\varepsilon > 0$ and consider a point $x + \varepsilon h$. Since $f(x)$ is differentiable, for $\varepsilon$ small enough, we have
    \begin{align*}
        \phi(x + \varepsilon h) = \phi(x) + \varepsilon\inner{\nabla \phi(x)}{h} + r(\varepsilon h),
    \end{align*}
    where $r(\varepsilon h)$ satisfies $\lim_{\varepsilon \rightarrow 0} \frac{r(\varepsilon h)}{\varepsilon \cdot \norm{h}} = 0$.
    Therefore, $\phi(x + \varepsilon h) < \phi(x)$, and hence $x + \varepsilon h \in \cC$.
    However, we also have
    \begin{align*}
        \inner{v}{x + \varepsilon h - x} = \varepsilon\inner{v}{h} > 0, 
    \end{align*}
    so $v \notin \cN_{\cC}(x)$.
    This contradiction shows that there are no directions other than $\lambda\nabla\phi(x)$, $\lambda \geq 0$ in $\cN_{\cC}(x)$.
\end{proof}

The following corollary is a direct consequence of \Cref{lemma:normal-cone-convex}:
\begin{corollary}
    \label{cor:gradient-k+1}
    Let $f: \R^d \mapsto \R$ be a differentiable convex function and $h: \R^d \mapsto \R$ be a continuously differentiable strictly convex function. Then
    \begin{align*}
        \partial \delta_{\cH_k}(x_k) = \cbrac{\lambda\rbrac{\nabla h(x_{k+1}) - \nabla h(x_k)}: \lambda \geq 0},
    \end{align*}
    where $\{x_k\}_{k\geq0}$ are the iterates generated by \ref{eq:alg:bregmanbppm}.
    Hence, there exists a function $c_{t, h}: \R^d\to\R_{\geq 0}$ such that 
    \begin{align*}
        \nabla f(x_{k+1}) = c_{t, h}(x_k)\rbrac{\nabla h(x_k) - \nabla h(x_{k+1})}.
    \end{align*}
\end{corollary}
\begin{remark}
    A similar result applies when $f: \R^d \mapsto \R \cup \cbrac{+\infty}$. In this case, the subdifferential is given by
    \begin{align*}
        \partial f(x_{k+1}) = \cbrac{\lambda\rbrac{\nabla h(x_k) - \nabla h(x_{k+1})}: \lambda \geq 0}.
    \end{align*}
    In both scenarios, convexity ensures that
    \begin{align}
        \label{eq:ineq:breg}
        f(y) \geq f(x_{k+1}) + c_{t, h}(x_k) \inner{\nabla h(x_k) - \nabla h (x_{k+1})}{y - x_{k+1}}
    \end{align}
    for some $c_{t, h}(x_k) \geq 0$ and any $y \in \R^d$.
    Consequently, $c_{t, h}(x_k)$ can be bounded above as follows:
    \begin{align}
        \label{eq:upperbound-c-1}
        c_{t, h}(x_k) \leq \frac{f(x_k) - f(x_{k+1})}{\inner{\nabla h(x_k) - \nabla h(x_{k+1})}{x_k - x_{k+1}}} \leq \frac{f(x_0) - f(x_\star)}{\breg{h}{x_k}{x_{k+1}} + \breg{h}{x_{k+1}}{x_k}} \leq \frac{f(x_0) - f(x_\star)}{t^2}.
    \end{align}
\end{remark}
\begin{proof}[Proof of \Cref{cor:gradient-k+1}]
    Using Example $3.5$ of \citep{beck2017first}, we know that $\delta_{\cH_k}(x_{k+1}) = \cN_{\cH_k}(x_{k+1})$.
    By \Cref{lemma:breg-singleton}, $\BBProxSub{t}{f}{h}{x_k}$ is a singleton and $x_{k+1} = \BBProxSub{t}{f}{h}{x_k} \in \bdry \cH_k$.
    Next, invoking \Cref{lemma:normal-cone-convex}, the Bregman divergence $\breg{h}{z}{x_k}$ is differentiable and convex in its first argument, with $\nabla_z \breg{h}{z}{x_k} = \nabla h(z) - \nabla h(x_{k+1})$.
    Hence, 
    \begin{align*}
        \delta_{\cH_k}(x_{k+1}) = \cN_{\cH_k}(x_k)
        = \cbrac{\lambda \rbrac{\nabla h (x_{k+1}) - \nabla h(x_k)}: \lambda \geq 0}.
    \end{align*}
    Since by the optimality condition $-\nabla f(x_{k+1}) \in \delta_{\cH_k}(x_{k+1})$, we conclude that there exists $c_{t, h}(x_k) \geq 0$ such that 
    \begin{align*}
        \nabla f(x_{k+1}) = c_{t, h}(x_k)\rbrac{\nabla h(x_k) - \nabla h(x_{k+1})}.
    \end{align*}
\end{proof}

Equipped with the necessary analytical tools, we now derive the convergence guarantee for \ref{eq:alg:bregmanbppm}.
\begin{theorem}
    \label{thm:bregbppm}
    Let $f: \R^d \mapsto \R \cup \cbrac{+\infty}$ be proper, closed and convex, and $h: \R^d \mapsto \R$ be continuously differentiable and strictly convex.
    Then, for any $K\geq 1$, the iterates of \ref{eq:alg:bregmanbppm} satisfy 
    \begin{align*}
        f(x_K) - f(x_\star) \leq \frac{\rbrac{f(x_0) - f(x_\star)}\breg{h}{x_\star}{x_0}}{Kt^2}.
    \end{align*}
\end{theorem}
\begin{proof}
    Let us consider some iteration $k$ such that $x_{k+1} \not\in \cX_f$.
    Taking $y = x_\star \in \cX_f$ in \eqref{eq:ineq:breg} and rearranging the terms, we have
    \begin{align*}
        f(x_{k+1}) - f_{\star} \leq c_{t, h}(x_k)\inner{\nabla h(x_k) - \nabla h(x_{k+1})}{x_{k+1} - x_\star}.
    \end{align*}
    Now, using the four point identity (\Cref{lemma:four-point-identity}) gives
    \begin{align*}
        f(x_{k+1}) - f_{\star} &\leq c_{t, h}(x_k)\rbrac{\breg{h}{x_{k+1}}{x_{k+1}} + \breg{h}{x_\star}{x_k} - \breg{h}{x_{k+1}}{x_k} - \breg{h}{x_\star}{x_{k+1}}} \notag \\
        &= c_{t, h}(x_k)\rbrac{\breg{h}{x_\star}{x_k} - \breg{h}{x_{k+1}}{x_k} - \breg{h}{x_\star}{x_{k+1}}}.
    \end{align*} 
    Rearranging, we have 
    \begin{align*}
        f(x_{k+1}) - f_{\star}
        \leq f(x_{k+1}) - f_{\star} + c_{t, h}(x_k)\breg{h}{x_{k+1}}{x_k}
        \leq c_{t, h}(x_k)\rbrac{\breg{h}{x_\star}{x_k} - \breg{h}{x_\star}{x_{k+1}}},
    \end{align*}
    and hence, applying the bound in \eqref{eq:upperbound-c-1} gives
    \begin{align*}
        f(x_{k+1}) - f_{\star} \leq \frac{f(x_0) - f_{\star}}{t^2}\rbrac{\breg{h}{x_\star}{x_k} - \breg{h}{x_\star}{x_{k+1}}}.
    \end{align*}
    Finally, averaging over $k\in\{0, 1, \hdots, K-1\}$ and noticing that the function values are decreasing, we obtain
    \begin{align*}
        f(x_K) - f_{\star}
        \leq \frac{1}{K} \sum_{k=0}^{K-1} f(x_{k+1}) - f_{\star}
        \leq \frac{\rbrac{f(x_0) - f_{\star}}\breg{h}{x_\star}{x_0}}{Kt^2}.
    \end{align*}
\end{proof}

\begin{remark}
    By following the same steps, one can establish a convergence guarantee for a general proper, closed and convex function $f: \R^d \to \R \cup \cbrac{+\infty}$.
\end{remark}

\begin{remark}
    For $h = \norm{\cdot}^2$, the convergence guarantee becomes 
    \begin{align*}
        f(x_K) - f(x_\star) \leq \frac{\rbrac{f(x_0) - f(x_\star)}\norm{x_0 - x_\star}^2}{2Kt^2},
    \end{align*}
    which matches the result from \Cref{thm:conv-bppm-convex} up to a constant factor.
    The discrepancy arises due to the asymmetry of the Bregman divergence.
\end{remark}


\clearpage

\section{Notation}
\label{app:notation}

\bgroup
\def\arraystretch{1.3}
\begin{table}[H]
\label{table:unbalanced}
	\small
	\centering
	\begin{tabular}{|c|p{12cm}|}
	\hline
	\multicolumn{2}{|c|}{Notation} \\
	\hline
	    $x_k$ & $k$-th iterate of an algorithm \\
        $\norm{\cdot}$ & Standard Euclidean norm \\
        $\inp{\cdot}{\cdot}$ & Standard Euclidean inner product \\
        $[k]$ & $\eqdef \{1,\ldots,k\}$ \\
        $d$ & Dimensionality of the problem \\
        $n$ & Number of clients (\Cref{sec:ap_stochastic}) \\
        $B_t(x)$ & $\eqdef \brac{z \in \R^d : \norm{z - x} \leq t}$ \\
        $\nabla f(x)$ & Gradient of function $f$ at $x$ \\
        $\partial f(x)$ & Subdifferential of function $f$ at $x$ \\
        $\cX_f$ & $\eqdef \{x\in\R^d: \nabla f(x) = 0\}$ \\
        $f_\star$ & Minimum of $f$ \\
        $\inf f$ & Infimum of $f$ \\
        $h_k$ & $\eqdef f(x_k) - f_\star$ \\
        $d_k$ & $\eqdef \norm{x_k - x_\star}$ for a given minimizer $x_\star\in \cX_f$ \\
        $\BProxSub{t}{f}{x}$ & Broximal operator associated with function $f$ with radius $t>0$ \\
        $\BMoreauSub{t}{f}{x}$ & Ball envelope function associated with function $f$ with radius $t$ \\
        $\breg{f}{x}{y}$ & The Bregman divergence associated with $f$ at $(x, y)$ \\
        $\Pi(\cdot, \cX)$ & Projection onto a set $\cX$ \\
        $\delta_{\cX}(y)$ & $=\begin{cases} 0, & y\in \cX \\ +\infty, & y\notin \cX \end{cases}$ \\
        $\textnormal{dist}(x, \cX)$ & $\eqdef \inf_{z\in\cX} \norm{x-z}$ \\
        $\interior(\cX)$ & Interior of the set $\cX$ \\
        $\ri(\cX)$ & Relative interior of the set $\cX$ \\
        $\bdry \cX$ & Boundary of the set $\cX$ \\
        $\textnormal{Fix}(\mathbb{A})$ & The set of fixed points of operator $\mathbb{A}$ \\
        $\cN_{\cX}(x)$ & $ \eqdef \brac{g\in\R^d : \inp{g}{z - x} \leq 0 \,\forall z \in \cX}$ -- the normal cone of $\cX$ at $x$ \\
       $ \R_{\geq0}(z)$ & $\eqdef \{\lambda z \;:\; \lambda \geq 0\}$ \\
        \hline
	\end{tabular}
	\caption{Frequently used notation.}
	\label{table:notation}
\end{table}
\egroup

\end{document}